\documentclass[10pt,a4paper]{amsart}[2000]
\usepackage{geometry}
\usepackage{comment}
\usepackage{amsmath}
\usepackage{amssymb}
\usepackage{amsthm}
\usepackage{mathrsfs}
\usepackage[hang]{footmisc}

\usepackage[linktocpage]{hyperref}
\usepackage{tikz, tikz-cd}
\usepackage{color}
\usepackage{mathtools}
\usepackage{enumitem}
\usepackage{stackengine}

\usepackage{graphicx}
\usepackage{esvect}
\makeatletter
\DeclareRobustCommand{\cev}[1]{%
  {\mathpalette\do@cev{#1}}%
}
\newcommand{\do@cev}[2]{%
  \vbox{\offinterlineskip
    \sbox\z@{$\m@th#1 x$}%
    \ialign{##\cr
      \hidewidth\reflectbox{$\m@th#1\vec{}\mkern4mu$}\hidewidth\cr
      \noalign{\kern-\ht\z@}
      $\m@th#1#2$\cr
    }%
  }%
}
\makeatother

\thanks{}

\allowdisplaybreaks

\theoremstyle{plain}

\newtheorem{Thm}{Theorem}[section]

\theoremstyle{definition}

\theoremstyle{plain}
\newtheorem{thm}[Thm]{Theorem}
\newtheorem{lem}[Thm]{Lemma}
\newtheorem{cor}[Thm]{Corollary}
\newtheorem{prop}[Thm]{Proposition}

\theoremstyle{definition}
\newtheorem{defn}[Thm]{Definition}

\newtheorem{eg}[Thm]{Example}
\newtheorem{rmk}[Thm]{Remark}

\newtheorem{innercustomthm}{Theorem}
\newenvironment{customthm}[1]
{\renewcommand\theinnercustomthm{#1}\innercustomthm}
{\endinnercustomthm}

\newcommand{\B}{B}
\newcommand{\A}{A}
\newcommand{\J}{J}
\newcommand{\K}{\mathcal{K}}
\newcommand{\D}{D}
\newcommand{\Ch}{D}
\newcommand{\Zh}{\mathcal{Z}}
\newcommand{\E}{E}
\newcommand{\Oh}{\mathcal{O}}

\newcommand{\T}{{\mathbb T}}
\newcommand{\R}{{\mathbb R}}
\newcommand{\N}{{\mathbb N}}
\newcommand{\Z}{{\mathbb Z}}
\newcommand{\C}{{\mathbb C}}
\newcommand{\Q}{{\mathbb Q}}

\newcommand{\aut}{\mathrm{Aut}}
\newcommand{\supp}{\mathrm{supp}}

\newcommand{\eps}{\varepsilon}
\numberwithin{equation}{section}

\newcommand{\id}{\mathrm{id}}

\newcommand{\halpha}{\widehat{\alpha}}
\newcommand{\calpha}{\widehat{\alpha}}

\newcommand{\tih}{\widetilde {h}}

\newcommand\set[1]{\left\{#1\right\}}  % Menge
\newcommand\mset[1]{\left\{\!\!\left\{#1\right\}\!\!\right\}}

 \newcommand{\IN}[0]{\mathbb{N}}

\newcommand{\CA}[0]{\mathcal{A}} \newcommand{\CB}[0]{\mathcal{B}}
\newcommand{\CC}[0]{\mathcal{C}} \newcommand{\CD}[0]{\mathcal{D}}
\newcommand{\CE}[0]{\mathcal{E}} \newcommand{\CF}[0]{\mathcal{F}}
\newcommand{\CG}[0]{\mathcal{G}} \newcommand{\CH}[0]{\mathcal{H}}
 
 \newcommand{\CL}[0]{\mathcal{L}}
\newcommand{\CM}[0]{\mathcal{M}} 
\newcommand{\CO}[0]{\mathcal{O}} 
\newcommand{\CQ}[0]{\mathcal{Q}} 
\newcommand{\CS}[0]{\mathcal{S}} \newcommand{\CT}[0]{\mathcal{T}}
\newcommand{\CU}[0]{\mathcal{U}} 
 
 \newcommand{\CZ}[0]{\mathcal{Z}}

\newcommand{\Ra}[0]{\Rightarrow}
\newcommand{\La}[0]{\Leftarrow}
\newcommand{\LRa}[0]{\Leftrightarrow}

\newcommand{\quer}[0]{\overline}
\newcommand{\eins}[0]{\mathbf{1}}			% Eine Eins in allgemeinerem Kontext, z.B. in einem Ring
\newcommand{\diag}[0]{\operatorname{diag}}
\newcommand{\ad}[0]{\operatorname{Ad}}
\newcommand{\ev}[0]{\operatorname{ev}}
\newcommand{\fin}[0]{{\subset\!\!\!\subset}}
\newcommand{\diam}[0]{\operatorname{diam}}
\newcommand{\Hom}[0]{\operatorname{Hom}}
\newcommand{\dst}[0]{\displaystyle}
\newcommand{\spp}[0]{\operatorname{supp}}
\newcommand{\lsc}[0]{\operatorname{Lsc}}
\newcommand{\del}[0]{\partial}
\newcommand{\GU}[0]{\CG^{(0)}}

\theoremstyle{definition}

\numberwithin{equation}{Thm}

\title[]{Full Gabor frames, its existence problem, and a non-uniform Balian-Low type theorem}	

\begin{document}
\global\long\def\floorstar#1{\lfloor#1\rfloor}
\global\long\def\ceilstar#1{\lceil#1\rceil}	

\global\long\def\B{B}
\global\long\def\A{A}
\global\long\def\J{J}
\global\long\def\K{\mathcal{K}}
\global\long\def\D{D}
\global\long\def\Ch{D}
\global\long\def\Zh{\mathcal{Z}}
\global\long\def\E{E}
\global\long\def\Oh{\mathcal{O}}

\global\long\def\T{{\mathbb{T}}}
\global\long\def\BR{{\mathbb{R}}}
\global\long\def\N{{\mathbb{N}}}
\global\long\def\Z{{\mathbb{Z}}}
\global\long\def\C{{\mathbb{C}}}
\global\long\def\Q{{\mathbb{Q}}}

\global\long\def\aut{\mathrm{Aut}}
\global\long\def\supp{\mathrm{supp}}

\global\long\def\eps{\varepsilon}

\global\long\def\id{\mathrm{id}}

\global\long\def\halpha{\widehat{\alpha}}
\global\long\def\calpha{\widehat{\alpha}}

\global\long\def\tih{\widetilde{h}}

\global\long\def\opFol{\operatorname{F{\o}l}}

\global\long\def\opRange{\operatorname{Range}}

\global\long\def\opIso{\operatorname{Iso}}
\global\long\def\opisom{\operatorname{Isom}}
\global\long\def\dimnuc{\dim_{\operatorname{nuc}}}
\global\long\def\dimtow{\dim_{\operatorname{tow}}}

\global\long\def\set#1{\left\{  #1\right\}  }

% Menge

\global\long\def\mset#1{\left\{  \!\!\left\{  #1\right\}  \!\!\right\}  }

\global\long\def\Ra{\Rightarrow}
\global\long\def\La{\Leftarrow}
\global\long\def\LRa{\Leftrightarrow}

\global\long\def\quer{\overline{}}
\global\long\def\eins{\mathbf{1}}
\global\long\def\diag{\operatorname{diag}}
\global\long\def\ad{\operatorname{Ad}}
\global\long\def\ev{\operatorname{ev}}
\global\long\def\fin{{\subset\!\!\!\subset}}
\global\long\def\diam{\operatorname{diam}}
\global\long\def\Hom{\operatorname{Hom}}
\global\long\def\dst{{\displaystyle }}
\global\long\def\spp{\operatorname{supp}}
\global\long\def\spo{\operatorname{supp}_{o}}
\global\long\def\del{\partial}
\global\long\def\lsc{\operatorname{Lsc}}
\global\long\def\GU{\CG^{(0)}}
\global\long\def\HU{\CH^{(0)}}
\global\long\def\AU{\CA^{(0)}}
\global\long\def\BU{\CB^{(0)}}
\global\long\def\CUU{\CC^{(0)}}
\global\long\def\DU{\CD^{(0)}}
\global\long\def\QU{\CQ^{(0)}}
\global\long\def\TU{\CT^{(0)}}
\global\long\def\CUUU{\CC'{}^{(0)}}
\global\long\def\dom{\operatorname{dom}}
\global\long\def\ran{\operatorname{ran}}
\global\long\def\AUl{(\CA^{l})^{(0)}}
\global\long\def\BUl{(B^{l})^{(0)}}
\global\long\def\HUp{(\CH^{p})^{(0)}}
\global\long\def\sym{\operatorname{Sym}}
\global\long\def\stab{\operatorname{Stab}}
\newcommand{\cat}[0]{\operatorname{CAT}(0)}
\global\long\def\properlength{proper}
\global\long\def\deg{\operatorname{deg}}
\global\long\def\isom{\operatorname{Isom}}
\global\long\def\interior#1{#1^{\operatorname{o}}}
	\global\long\def\ln{\operatorname{ln}}

\author{Rui Liu}
	
	\address{R. Liu:  School of Mathematical Sciences and LPMC, Nankai University,
Tianjin, 300071, PR China.}
  \email{ruiliu@nankai.edu.cn}

\author{Xin Ma}
	\address{X. Ma:  Institute for Advanced Study in Mathematics, Harbin Institute of Technology, Harbin, China, 150001}
\email{xma17@hit.edu.cn}

\author{Yuxuan Zheng}
\address{Y. Zheng: School of Mathematical Sciences and LPMC, Nankai University,
Tianjin, 300071, PR China.}
\email{yuxuanzheng@mail.nankai.edu.cn}


%\subjclass[2010]{46L35, 37B05, 57S30, 20F36, 20F55, 20E08}

\date{\today}

\begin{abstract}
For a broad class of Delone sets in $\R^n$
that are of significance in both mathematics and physics, we prove a non‑uniform Balian–Low type theorem and settle the converse problem on the existence of Gabor frames, for arbitrary dimension $n$. To this end, we introduce a class of Gabor frames, termed full Gabor frames, and prove that the existence of such a frame on the Delone set with  Schwartz window functions is equivalent to the condition that the lower Beurling density be strictly greater than one. In fact, the usual Balian--Low direction using window functions from the Feichtinger's algebra can be proven for arbitrary point sets, thereby improving an earlier density theorem by Christensen, Deng, and Heil. The corresponding dual result for Riesz sequences is also obtained. The main technical tools employed in this paper are tiling groupoid constructions and $C^*$-algebraic methods. As a byproduct, we resolve an open question from Ito's thesis concerning the bounded dynamical asymptotic dimension of tiling groupoids. Furthermore, this result allows us to extend the classification theorem of Ito, Whittaker, and Zacharias to the twisted case.
\end{abstract}

%\footnote{This work was supported by the National Natural Science Foundation of China (Grant Nos. 11971348, 12071230 and 12471131).}

\maketitle

\section{Introduction}
The structure theory of Gabor frames and Gabor Riesz sequences has long been a central topic in time-frequency analysis. In the classical Euclidean setting, for $z = (x, \omega) \in \mathbb{R}^{2d}$ and $g \in L^2(\mathbb{R}^d)$, the translation operator $T_x$, modulation operator $M_\omega$, and time-frequency shift $\pi(z)$ are defined as
\[   T_x g(t) = g(t-x), \ \ \ \ 
    M_\omega g(t) = e^{2\pi i \omega t} g(t), \ \ \ \ 
    \pi(z) g(t) = M_\omega T_x g(t) = e^{2\pi i \omega t} g(t-x),
\]
with $t \in \mathbb{R}^d$. The function $g$ is typically called a \textit{window function}. Let $\Lambda \subset \mathbb{R}^{2d}$ be a discrete set. The associated \textit{Gabor system} is defined as $G(g, \Lambda) = \{\pi(\lambda) g : \lambda \in \Lambda\}$. Such a system is called a \textit{Gabor frame} if there exist constants $0 < A \leq B < \infty$ such that
$$
A \|f\|_2^2 \leq \sum_{\lambda \in \Lambda} |\langle f, \pi(\lambda) g \rangle|^2 \leq B \|f\|_2^2 \quad \text{for all } f \in L^2(\mathbb{R}^d).
$$
Regarding interpolation, $\mathcal{G}(g, \Lambda)$ is said to be a \textit{Gabor Riesz sequence} if there exist constants $0 < A \leq B < \infty$ such that
$$
A \|c\|_2 \leq \big\| \sum_{\lambda \in \Lambda} c_\lambda \pi(\lambda) g \big\|_2 \leq B \|c\|_2 \quad \text{for all } c \in \ell^2(\Lambda).
$$

Deep connections have been established between the existence of a Gabor frame (or, respectively, a Gabor Riesz sequence) $G(g, \Lambda)$ with window function $g$ in Feichtinger's algebra or Schwartz space and the Beurling densities of the sampling (or interpolating) set $\Lambda$ in $\mathbb{R}^{2d}$. These connections are known as the Balian–Low type theorem and its converse. For example, in the one‑dimensional case $d = 1$ with a Gaussian window $g$, Lyubarskii \cite{Lyubarskii}, Seip \cite{Seip}, and Seip–Wallstén \cite{SW} proved that $G(g, \Lambda)$ is a frame for $L^2(\mathbb{R})$ if and only if $D^{-}(\Lambda) > 1$, and a Riesz sequence if and only if $D^{+}(\Lambda) < 1$. For other special windows—such as totally positive functions
%, certain rational functions, 
and functions of hyperbolic secant type—we refer to \cite{BB,G23,GRS18} for results showing that $G(g, \Lambda)$ over certain semi‑regular set or certain lattice $\Lambda$ is a frame precisely when $D^{-}(\Lambda) >1$.

In higher dimensions, however, the problem becomes significantly more intricate. Suppose $d \ge 2$. Gröchenig and Lyubarskii constructed counterexamples in \cite{GL} showing that time‑frequency lattices $\Lambda$ with the density $D(\Lambda)>1$ may fail to generate a Gabor frame, even for Gaussian windows. Nevertheless, using Kähler geometry, Luef and Wang \cite{LW} provided a sufficient condition for the construction of Gaussian Gabor frames for almost all lattice. In a related development, Gröchenig extended the complex variable techniques from the univariate case to produce various examples of complex lattices with covolume less than $1$ for which the associated Gaussian Gabor system forms a frame \cite{G11}. Romero, Ulanovskii, and Zlotnikov \cite{RUZ} further supplied numerous examples of bivariate Gaussian frames for certain semi‑regular sets $\Lambda \times \mathbb{Z}^2 \subset \mathbb{R}^4$ satisfying $D^{-}(\Lambda \times \mathbb{Z}^2) > 1$. 

In noncommutative geometry, projective modules over noncommutative $C^*$-algebras serve as a noncommutative analogue of vector bundles over manifolds. Connes constructed projective modules over smooth noncommutative tori \cite{Con}, and Rieffel later extended this construction to higher‑dimensional noncommutative tori, introducing what are now known as Heisenberg modules \cite{Ri}. Since this pioneering work, such projective modules have found numerous applications in mathematics and physics, particularly in Gabor analysis. Luef \cite{Luef} first established a link between the duality theory of regular Gabor frames with windows in modulation spaces \cite{F83} and the Morita–Rieffel equivalence of Heisenberg modules over noncommutative tori. Consequently, the generators of projective modules over noncommutative tori are Gabor atoms of multi‑window Gabor frames for modulation spaces, and there exist well‑behaved multi‑window Gabor frames with atoms in modulation spaces or in the Schwartz space. Building on \cite{Luef} and on Rieffel's work on strict comparison of projections and cancellation in noncommutative tori \cite{Ri}, Jakobsen and Luef \cite{JL} resolved the existence problem for Gabor frames with atoms in  Feichtinger's algebra on non‑rational lattices. More recently, Enstad, Thiel, and Vilata \cite{ETV} treated the rational lattice case for Schwartz Gabor frames. In addition, Luef \cite{Luef18} invoked a  constant curvature connections on Heisenberg modules over noncommutative tori to interpret a version of the Balian–Low theorem in terms of noncommutative geometry.

In a recent paper \cite{BEV}, B\'edos, Enstad, and van Velthoven extended the Hilbert module framework of \cite{Ri} from locally compact abelian groups to the setting of nilpotent Lie groups and generalized the existence theorem for Gabor frames established in \cite{Luef, JL}.  In \cite{BEV}, the authors demonstrated a converse result to the Balian–Low type theorem in  \cite{AFK, FK, GOR, GRRV20} for smooth vectors. For further details, see \cite[Theorems 1.3]{BEV} and the parallel results for Riesz sequences. Another key ingredient in their approach is the simplicity of the twisted group $C^*$-algebra and its strict comparison of projections, which follows from recent progress on the Toms–Winter conjecture (see, e.g., \cite{Winter, Ror}). Moreover, we refer to \cite{ETV} and \cite{EV} for results obtained in the case of non‑simple $C^*$-algebras.

Beyond the cases of lattices and semi‑regular sets, Gr\"ochenig, Ortega‑Cerd\`a, and Romero \cite{GOR} studied the deformation of Gabor systems and obtained a non‑uniform Balian–Low theorem \cite[Corollary 1.2]{GOR}. We now introduce several key concepts. Let $\Lambda_1, \dots, \Lambda_n$ be a finite family of discrete sets in $\mathbb{R}^{2d}$ and let $g_1, \dots, g_n \in L^2(\mathbb{R}^d)$. In \cite[Section 6.1]{GRS20}, Gr\"ochenig, Romero, and St\"ockler defined the \textit{multi‑window Gabor system} as
\[
G(\Lambda_1,\dots,\Lambda_n, g_1,\dots, g_n) := \{\pi(z)g_i : z \in \Lambda_i,\ i = 1,\dots, n\} = \bigcup_{i=1}^n G(\Lambda_i, g_i).
\]
We note that the term ``multi‑window Gabor system'' is also widely used in the literature when all $\Lambda_i$ coincide with a single $\Lambda$. 
A multi‑window Gabor system $G(\Lambda_1,\dots,\Lambda_n, g_1,\dots, g_n)$ is a frame if there exist constants $0 < A \le B < \infty$ such that
\[
A\|f\|_2^2 \le \sum_{i=1}^n \sum_{\lambda \in \Lambda_i} |\langle f, \pi(\lambda) g_i \rangle|^2 \le B\|f\|_2^2 \qquad \text{for all } f \in L^2(\mathbb{R}^d).
\]
Remarkably, \cite[Theorem 6.1]{GRS20} provides a connection between the Gabor frame property of  multi‑window Gabor system and a sampling problem  for certain discrete sets, as a vector‑valued extension of \cite[Theorem 2.3]{GRS18}.
%Remarkably, \cite[Theorem 6.1]{GRS20} provides a characterization of generalized multi‑window Gabor frames in $L^2(\mathbb{R})$ for certain discrete sets, as a vector‑valued extension of \cite[Theorem 2.3]{GRS18}.

Using the results of \cite{GOR}, Kreisel \cite{K} constructed multi‑window Gabor frames for any quasicrystal $\Lambda \subset \mathbb{R}^{2d}$. It is worth emphasizing that when $\Lambda$ is no longer a lattice, one can work with the tiling groupoid $R_\Lambda$ of $\Lambda$, as in \cite{K}. This approach was subsequently extended by \cite{ER} to locally compact groups, yielding a more abstract density theorem; see \cite[Theorem 4.1]{ER}.

Motivated by \cite{GOR, ETV, K, ER}, in this paper we investigate the Balian–Low type theorem and its converse for point sets in $\mathbb{R}^{2d}$ that are not necessarily lattices, using techniques from groupoids and $C^*$-algebras. We first introduce a concept, which we call a \textit{full Gabor system}. Let $\Lambda$ be a discrete set in $\mathbb{R}^{2d}$ with a disjoint decomposition $\Lambda = \bigsqcup_{i=1}^n \Lambda_i$, and let $\vec{g} = (g_1,\dots, g_n) \in L^2(\mathbb{R}^d)^n$. We define
\[
G_F(\Lambda, \vec{g}) = G(\Lambda_1,\dots,\Lambda_n, g_1,\dots, g_n) := \{\pi(z)g_i : z \in \Lambda_i,\ i = 1,\dots, n\}
\]
as a \textit{full Gabor system} for $\Lambda = \bigsqcup_{i=1}^n \Lambda_i$ and $\vec{g} = (g_1,\dots, g_n)$. We say $G_F(\Lambda, \vec{g})$ a \textit{full Gabor frame} (resp. \textit{full Riesz sequence}) if it is a frame (resp. a Riesz sequence). This notion is essentially a piecewise constant version of a single‑window Gabor system and shares the spirit of \textit{full groups} arising from a single group action (see, e.g., \cite{GPS}). Although not explicitly named as such, the full Gabor system has already appeared in the literature, and a density theorem for it was established by Christensen, Deng, and Heil \cite[Theorem 1.1]{CDH}. 

Our main contribution in this paper is to solve the existence problem of full Schwartz Gabor frames and provide a Balian–Low type theorem for repetitive, aperiodic Delone sets with finite local complexity (FLC), using window functions from the Schwartz space $\mathcal{S}(\mathbb{R}^d)$. In fact, we obtain a characterization of the existence of full Gabor frames for such Delone sets in higher‑dimensional Euclidean spaces in terms of density conditions. To the best of our knowledge, this is the first bidirectional result on this topic for the class of Delone sets in arbitrary high dimensional Euclidean spaces $\R^n$ that are of both mathematical and physical significance.

\begin{eg}
A natural way to describe such a Delone set is to use tilings, by regarding the Delone set as the puncture points set of a FLC, repetitive, aperiodic tiling of $\mathbb{R}^d$, which is the set of all the translates of any chosen points from the \textit{prototiles} of the tiling. We refer to \cite[Definition 5.2]{BG} for these concepts. In particular, this class of tilings includes the famous Penrose tilings (see \cite[Proposition 6.3]{BG}) and Ammann–Beenker tilings (see \cite[Proposition 6.2]{BG}). We refer to, e.g., \cite[Figure 6.44, p. 238]{BG} and \cite[Figure 6.41, p. 236]{BG} for excellent pictures for these two tilings.  In addition, a large class of higher‑dimensional examples arises from the well‑known process of \textit{substitution tilings} (see, e.g., \cite[Section 3.2]{Ito}), such as the $d$-dimensional chair tiling for arbitrary $d \ge 2$ (see \cite[Example 6.8]{BG}). We refer to \cite[Figure 6.35, p. 228]{BG} for a picture of $3$-dimensional chair tiling.
Moreover, such FLC, repetitive, aperiodic tilings have deep connections to noncommutative geometry and $C^*$-algebras through the construction of tiling $C^*$-algebras. For further information, we refer to \cite{Kel0, Kel, KP, K, Ito, IWZ}. 
\end{eg}

\begin{customthm}{A}[Theorem \ref{Thm:converse Balian-low}, Theorem \ref{thm: improved Balian-Low}]\label{thm: main 1}
		 Let $\Lambda  \subset \mathbb{R}^{2d}$ be a FLC, repetitive and aperiodic Delone set. Then the following hold:
    \begin{enumerate} [label=(\roman*)]
        \item $D^-(\Lambda)  >1$ if and only if  there exist disjoint sets $\Lambda_i \subset \Lambda$ for $i=1,\dots,n$ and $\vec{g}=(g_1,\dots, g_n)\in \CS(\R^d)^n$ such that $\bigsqcup\limits_{i=1}^n \Lambda_i=\Lambda$ and 
        $G_F(\Lambda, \vec{g})$ is a frame for $L^2(\mathbb{R}^d)$.
        \item  $D^+(\Lambda)  <1$ if and only if  there exist disjoint sets $\Lambda_i \subset \Lambda$ for $i=1,\dots,n$ and $\vec{g}=(g_1,\dots, g_n)\in \CS(\R^d)^n$ such that $\bigsqcup\limits_{i=1}^n \Lambda_i=\Lambda$ and 
        $G_F(\Lambda, \vec{g})$ is a Riesz sequence for $L^2(\mathbb{R}^d)$. 
    \end{enumerate}
	\end{customthm}
The aperiodicity condition on $\Lambda$ prevents it from being \textit{crystallographic}, i.e., a finite shift of a lattice (see \cite[Definition 2.5]{BG}). Therefore, Theorem~\ref{thm: main 1} has a fundamentally different flavor compared to the lattice cases considered in the literature, such as \cite{Luef} and \cite{BEV}.

For the ``if'' part of Theorem \ref{thm: main 1}, the key ingredients is the construction of the tiling groupoid $R_\Lambda$ and the twisted groupoid $C^*$-algebra $C^*_r(R_\Lambda, \sigma_\Lambda)$. Unlike the lattice case studied in \cite{Ri, Luef, BEV}, a Delone set $\Lambda$ yields a tiling groupoid $R_\Lambda$ whose unit space is a Cantor set rather than a singleton. It is precisely this fiberwise nature of $R_\Lambda$ that creates the main obstacles. Two serious consequences arise. First, instead of a single fiber as in the lattice case, there are uncountably many range fibers. Consequently, the projective representation $\pi_\Lambda$ does not act uniformly; one must sum contributions over all range fibers, a task that has no analogue for lattices. Second—and more critically—the fiberwise structure prevents the construction of a single‑window Gabor frame for $L^2(\mathbb{R}^d)$. Indeed, because $\Lambda$ is not a lattice, each translate of $\Lambda$ by $z \in \Lambda$ is generally different, and the groupoid $R_\Lambda$ forces us to consider all such translates simultaneously.  In this more complex setting, finite generation of a Hilbert $C^*$-module over $C^*_r(R_\Lambda, \sigma_\Lambda)$ is equivalent only to the existence of multi‑frame vectors in the sense of \cite[Definition 3.1]{ER} for the projective representations $\pi_\Lambda$, taken in an averaging sense with respect to a given invariant probability measure on the Cantor set, as demonstrated in Proposition~\ref{single}.

To address the challenges outlined above, we introduce new techniques tailored to the groupoid setting. First, we perform a perturbation (see Proposition~\ref{prop: pre module}) of the existing Hilbert module framework from \cite{Ri, Luef, BEV}—specifically, the completion of admissible pairs as defined in \cite[Definition 4.1]{BEV}—thereby constructing a module $\mathcal{E}$ over $C_r^*(R_\Lambda, \sigma_\Lambda)$. Second, we employ a topological dynamical argument (Proposition~\ref{equ}) to convert average multi‑frame vectors into genuine multi‑frame vectors. Finally, assuming that $C^*_r(R_\Lambda, \sigma_\Lambda)$ enjoys the strict comparison property and that the relevant density condition is satisfied, we obtain a single generating element for $\mathcal{E}$, which in turn yields a full Gabor frame. The same strategy also handles the Riesz sequence part of the second statement.

It therefore remains to verify the strict comparison property for $C^*_r(R_\Lambda, \sigma_\Lambda)$ when $\Lambda$ is an FLC, repetitive, aperiodic Delone set—a question that has so far remained open. To resolve it, we prove that $R_\Lambda$ has finite dynamical asymptotic dimension (finite d.a.d.), a combinatorial invariant for groupoids introduced in \cite{GWY}. In fact, we establish this result for the more general class of transverse groupoids, which is of independent interest. As a corollary, we answer a question raised in \cite[Section 6.4]{Ito}. Finite nuclear dimension and strict comparison then follow from \cite{Winter, Ror, BL}; see also \cite{GWY, CDGHV}. Moreover, by using results from \cite{BL, CGSTW, EGLN, GLN, TWW}, our work extends the classification theorem of \cite{IWZ} to the twisted case.  

\begin{thm}[Corollary \ref{cor:C^*_r(R_varphi,sigma) has the strict comparison}]\label{thm: main2}
   Let $\R^d\curvearrowright \Omega$ be a free action on a compact metrizable space $\Omega$. Suppose the induced transverse groupoid (see Definition \ref{defn: groupoid def}) $R_\varphi$ is minimal and principal. Then its dynamical asymptotic dimension $\operatorname{d.a.d}(R_\varphi)\leq 6^d-1$. Therefore, for any twist $\Sigma$ (could come from a continuous $2$-cocycle $\sigma$) on $R_\varphi$, the nuclear dimension $\dimnuc(C^*_r(R_\varphi, \Sigma))\leq 6^d$ and thus $C^*_r(R_\varphi, \Sigma)$ has the strict comparison (of projections). In addition, $C^*_r(R_\varphi, \Sigma)$ is classified by its Elliott invariant.
\end{thm}

For the ``only-if'' direction in Theorem~\ref{thm: main 1} — that is, the Balian–Low type result — we actually establish in Theorem~\ref{thm: improved Balian-Low} the same conclusion even for a general discrete set and for $\vec{g} \in M^1(\mathbb{R}^d)^n$. This part is motivated by \cite[Corollary 1.7]{AFK}, \cite[Corollary 1.2]{GOR} and \cite{FK} and the proof is a small modification of that in \cite{GOR}. This particularly improves upon the density theorem in \cite[Theorem 1.1]{CDH}, which states that if the full Gabor system $G_F(\Lambda, \vec{g})$ is a frame, then $D^{-}(\Lambda) \ge 1$. 
We further note that \cite[Corollary 3.7]{CDH} shows that if a full Gabor system $G_F(\Lambda, \vec{g})$ for a discrete set $\Lambda$ is a Riesz basis for $L^2(\mathbb{R}^d)$, then necessarily $D^{+}(\Lambda) = 1$. However, by our Theorem~\ref{thm: main 1}, the full Riesz sequence $G_F(\Lambda, \vec{g})$ is never a Riesz basis for $\vec{g} \in M^1(\mathbb{R}^d)^n$. 

\subsection*{Outline of the paper:} In Section \ref{sec: prelim}, we mainly recall necessary backgrounds on time-frequency analysis, Delone sets, Groupoid, Hilber modules, and $C^*$-algebras. In Section \ref{sec: tower} and \ref{sec: finite dad}, we establish Theorem \ref{thm: main2} by introducing tower dimension for ample groupids. In Section \ref{sec: groupoid frame}, we extend the framework in \cite{ER} to discuss (average) frame and Riesz vectors for groupoids. In addition, we introduce corresponding analysis and synthesis operators for them. In Section \ref{sec: tiling groupoid}, we prove the ``if'' part of Theorem \ref{thm: main 1} using Theorem \ref{thm: main2} and Hilbert $C^*$-modules. In Section \ref{sec: Balian-Low}, motivated by \cite{GOR}, we prove the Balian-Low type result for full Gabor frame and full Riesz sequences, i.e., the ``only if '' part of Theorem \ref{thm: main 1}.

\section{Preliminaries}\label{sec: prelim}

\subsection{Time-frequency analysis}
In this section we collect some elementary concepts from time-frequency analysis. For more details, we refer readers to \cite{FZ,G01}.

Let $\Gamma$ be a countable index set. A set $\{e_\gamma:\gamma \in \Gamma\}$ in $\CH$ is called a frame if there exist constants $0< A \le B <\infty$ such that for all $f\in \CH$   
\begin{equation}\label{def:frame}
    A \|f\|^2 \le \sum_{\gamma \in \Gamma}|\langle f, e_\gamma \rangle |^2\le B \|f\|^2.
\end{equation}
Any constants $A,B$ satisfying (\ref{def:frame}) are called frame bounds.
If $A=B=1,$ the set $\{e_\gamma\}_{\gamma\in \Gamma}$ is said to constitute a Parseval frame. If only the upper frame bounds exist, we call $\{e_\gamma\}_{\gamma\in \Gamma}$ a Bessel sequence.

For a Bessel sequence $\{e_\gamma\}_{\gamma\in \Gamma},$ the associated analysis operator is given by 
\[C:\CH \to \ell^2(\Gamma), \quad Cf=(\langle f, e_\gamma\rangle )_{\gamma \in \Gamma}.\]
The synthesis operator $D=C^*$  is defined by 
\[D:\ell^2(\Gamma)\to \CH, \quad Dc=\sum_{\gamma \in \Gamma}c_\gamma e_\gamma.\]
The frame operator associated to $\{e_\gamma\}_{\gamma \in \Gamma}$ is given by $S=DC.$ Recall the definitions of (multi-window, full) Gabor systems, (multi-window, full) Gabor frames, (multi-window, full) Riesz sequences introduced in Introduction.

For a window function $g \in L^2(\R^d),$ the short-time Fourier transform (STFT) of $f\in L^2(\R^d)$ with respect to $g$ is given by 
\begin{equation}\label{def of STFT}
    V_gf(z)=\langle f, \pi(z)g\rangle=\int_{\R^d}f(t)\overline{g(t-x)}e^{-2\pi i t \omega} \text{d}t \ \text{for} \ z=(x,\omega)\in \R^{2d}.
\end{equation}

The time-frequency shifts satisfy the following  commutation relations 
\begin{equation} \label{eq:commutation relation}
    \pi (z_1)\pi(z_2)=\sigma(z_1,z_2)\pi(z_1 ,z_2), \ z_i=(x_i,\omega_i)\in \R^{d} \times \R^d, i=1,2,
\end{equation}
where $\sigma(z_1,z_2)=e^{-2\pi i x_1 \omega_2}$ is a symplectic $2$-cocycle on $\R^{2d}$. 
As a consequence, one has the following covariance principle 
\[V_g \pi (z_1)f(z_2) = e^{-2\pi i x_1( \omega_2 -\omega_1)}V_gf(x_2-x_1 ,\omega _2 -\omega_1), \ z_i=(x_i, \omega_i) \in \R^d \times \R^d, i=1,2.\]

The modulation spaces introduced by Feichtinger \cite{F81,F83} are appropriate function spaces for time-frequency analysis.
\begin{defn}\cite{F83}\label{def:modulation space}
    Fix a non-zero window $g$ in Schwartz space $\mathcal{S}(\R^d).$ For $1\le p\le \infty$, the \textit{modulation space} $M^p(\R^d)$ consists of all tempered distributions $f\in \mathcal{S}'(\R^d)$ such that $V_gf\in L^p(\R^{2d}).$ The norm on $M^p(\R^d)$ is defined as 
    \begin{equation}
        \|f\|_{M^p}=\|V_gf\|_{L^p}.
    \end{equation}
\end{defn}
The modulation spaces are independent of the particular choice of non-zero windows $g 
\in \CS(\R^d)$. 
The space $M^1(\R^d)$ is the well-known \textit{Feichtinger's algebra} \cite{F81} and it is known that $M^2(\R^d)=L^2(\R^d).$ Moreover, it follows from \cite[Theorem 9]{F81} that $M^1(\R^d)$
contains $\CS(\R^d)$. 
 Fix a non-zero window $g \in \CS(\R^d)$, the space $M^0(\R^d)$ is defined as \[M^0(\R^d)=\{f \in  M^{\infty}(\R^d): V_g f \in C_0(\R^{2d})\}.\]  With respect to the duality $\langle f,h\rangle:= \langle V_gf ,V_gh\rangle$, one has $M^0(\R^d)^*=M^1(\R^d)$ and $M^1(\R^d)^*=M^{\infty}(\R^d)$.
 
In addition,  \textit{Wiener amalgam space} $W(L^{\infty},l^p)(\R^{2d})$ consists of all measurable functions $f$ on $\R^{2d}$ satisfying
\begin{equation}\label{def:Wiener amalgam sp}
    \|f\|_{W(L^{\infty},l^p)}:=(\sum_{k\in \Z^{2d}}\|f\cdot T_k \chi_{[0,1]^{2d}}\|^p_{\infty})^{1/p} < \infty.
\end{equation} The subspace of $W(L^{\infty},\ell^1)(\R^{2d})$ consisting of continuous functions is $W(C_0 ,\ell^1)(\R^{2d})$.  
 If  $f,g\in M^1(\R^d)$, it follows from \cite[Theorem~12.1.11]{G01} that $V_gf\in W(C_0,l^1)(\R^{2d})$.

We denote the space of complex regular Borel measures on $\R^{2d}$ by $\mathcal{M}(\R^{2d})$. The dual space of $W(C_0 ,\ell^1)(\R^{2d})$ is $W(\mathcal{M}, L^{\infty})(\R^{2d})$ which consists of all complex-valued Borel measures $\mu$ such that \[\|\mu\|_{W(\mathcal{M},L^{\infty})}:=\sup_{x\in \R^{2d}}|\mu|(B(x,1)) < \infty.\] Feichtinger constructed general theory of Wiener amalgam space in \cite{F80}.

\subsection{Delone sets in $\R^d$ and their dynamics}
We now recall some basic backgrounds on point sets, especially Delone sets in Euclidean spaces $\R^d$. We refer to \cite{BG}, \cite{Ito}, \cite{IWZ}, \cite{K}, \cite{Kel0}, \cite{Kel}, and \cite{KP} for more information. We use the notation $B_Z(z,r)$ for the open ball with center $z$ and radius $r$ in a metric space $(Z,d)$ and also denote $B(z,r)$ if the metric space $(Z, d)$ is understood. Moreover, we denote by $\bar{B}_Z(z, r)$ or $\bar{B}(z, r)$ the corresponding closed balls.

\begin{defn}
  Let $(M, d)$ be a metric space and $\Lambda\subset M$ a discrete subset. 
  \begin{enumerate}[label=(\roman*)]
      \item $\Lambda$ is said to be \textit{uniformly separated}(or \textit{uniformly discrete}) if $\inf\{d(x, y): x\neq y\in \Lambda\}>0$.
      \item $\Lambda$ is said to be \textit{relatively dense} if there exists $c>0$ such that the collection $\{B_M(x,c):x\in \Lambda\}$ covers $M$.
      \item $\Lambda$ is said to be a \textit{Delone set} if $\Lambda$ is both uniformly separated and relatively dense.
  \end{enumerate}
\end{defn}

For $\R^d$, 
%(or more generally, a locally compact second countable unimodular groups), 
one makes the following definition.

\begin{defn}\label{defn: relative separate}
   Let $\Lambda\subset \R^d$ be a discrete set. We define the \textit{hole} $\rho(\Lambda)$ of $\Lambda$ to be 
   \[\rho(\Lambda):=\sup_{x\in \R^d}\inf_{\lambda\in \Lambda}|x-\lambda|.\]
     The $\Lambda$ is called \textit{relatively separated} if 
   \[\operatorname{rel}(\Lambda):=\sup \{|\Lambda\cap C_1(z)|: z\in \R^d\}\]
   is finite, where $C_1(z)$ is the cube centered at $z$, whose edges are of length $1$. 
   \end{defn}

\begin{rmk}\label{Delone set is relative separated}
Let $\Lambda\subset \R^d$ be a discrete subset. Then $\Lambda$ is relatively dense if and only if its hole $\rho(\Lambda)<\infty$. In adition, the relative separability in Definition \ref{defn: relative separate} means bounded geometry for the metric space $\Lambda$ as a subspace of $(\R^d, \|\cdot\|_\infty)$. It is direct to see that if $\Lambda$ is a Delone set, then $\operatorname{rel(\Lambda)}$ is finite by looking at the volume of cubes.
\end{rmk}

There is a metric on the collection  of Delone sets in the following way. Given two Delone set $\Lambda_1$ and $\Lambda_2$, define 
\[R(\Lambda_1,\Lambda_2)=\sup\{r>0:  \exists z\in \R^d\text{ with }\|z\|<1/r\text{ such that }B(0,r)\cap(\Lambda_1-z)=\Lambda_2\cap B(0,r)\},\]
Then define the metric to be $d(\Lambda_1, \Lambda_2)=\min\{1, 1/R(\Lambda_1,\Lambda_2)\}$. Denote by $(\CD,d)$ the metric space consisting of all Delone sets in $\R^d$. Then there is natural $\R^d$ action on $\CD$ by translation. 

\begin{defn}\label{defn: continuous hull}
Let $\Lambda$ be a FLC Delone set. The \textit{continuous hull}, denoted by $\Omega(\Lambda)$, is the closure of the orbit $\{\Lambda-z: z\in \R^d\}$ of $\Lambda$ under the translation action $\R^d\curvearrowright (\CD, d)$. 
\end{defn}

\begin{defn}\label{defn: Patch}
 Let $\Lambda$ be a Delone set and $r>0$. The intersection $\bar{B}(z,r)\cap \Lambda$ for a $z\in \Lambda$ is a called a $r$-\textit{patch} of $\Lambda$ centered at $z$.
\end{defn}

The following definitions could be found in, e.g. \cite[Section 2, Section 5]{BG}. See also \cite{LP}.

\begin{defn}\label{defn: properties for Delone sets}
Let $\Lambda$ be a Delone set in $\R^d$. We say
    \begin{enumerate}[label=(\roman*)]
        \item $\Lambda$ is of \textit{finite local complexity (FLC)} if for any $r>0$ there are only finitely many $r$-patches up to translation. 
        \item $\Lambda$ is \textit{non-periodic} if $\Lambda-z\neq \Lambda$ for any non-zero $z\in \R^d$. $\Lambda$ is further said to be \textit{aperiodic} if $\Gamma-z\neq \Gamma$ for any $\Gamma$ in the coninuous hull $\Omega(\Lambda)$.
        \item $\Lambda$ is \textit{repetitive} if for any $r>0$  there exists a finite number $R>0$ such that for any  $z\in \R^d$ and any $r$-patch $P$, the closed ball  $\bar{B}(z, R)$ contains the center of a $r$-patch which is a translate of $P$.  
    \end{enumerate}
\end{defn}

\begin{rmk}
\begin{enumerate}[label=(\roman*)]
    \item For the basic relation among these properties, it follows from \cite[Proposition 5.6]{BG} that if the Delone set $\Lambda$ is repetitive and non-periodic then $\Lambda$ is FLC and aperiodic.
    \item Each FLC, aperiodic, and repetitive Delone set $\Lambda\subset \R^d$ can produce FLC, aperiodic, and repetitive tilings of $\R^d$ via the well-known construction of Voronoi cells.  We refer to \cite[Chapter 5]{BG} and \cite{Ito} for what do these mean for tilings and more details. For the converse, any FLC, aperiodic, and repetitive Delone set $\Lambda$ can be regarded as the puncture point sets of a FLC, aperiodic, and repetitive tiling $\CT$. See \cite[Section 3.3]{Ito}.
\end{enumerate}

\end{rmk}

The following basic properties on  $\R^d\curvearrowright \Omega(\Lambda)$ is well-known. See, e.g., \cite[Theorem 3.1.15]{Ito}  written in the language of tiling groupoids, as we will explain in subsection \ref{subsec2}.

\begin{prop}\label{prop: basic Rd action}
Let $\Lambda$ be  a Delone set. 
\begin{enumerate}[label=(\roman*)]
    \item The continuous hull $\Omega(\Lambda)$ is compact metrizable if and only if $\Lambda$ is FLC.
    \item The translation action of $\R^d$ on $\Omega(\Lambda)$ is free if and only if $\Lambda$ is aperiodic.
    \item If $\Lambda$ is  FLC, then $\R^d\curvearrowright \Omega(\Lambda)$ is minimal if and only if  $ \Lambda$ is repetitive
    \end{enumerate}
\end{prop}

\begin{defn}
Let $\Lambda$ be a FLC Delone set.  The \textit{discrete hull} of $\Lambda$ is defined to be $\Omega_0(\Lambda)=\{T\in \Omega(\Lambda): 0\in T\}$.
\end{defn}

\subsection{Groupoids and their $C^*$-algebras}
 We refer to standard references \cite{R} and \cite{Sims} for more detailed information on groupoids and groupoid $C^*$-algebras. We only recall the following necessary notions that will be used in the paper.
 \begin{defn}
A \textit{groupoid} $\CG$ is a set equipped with a distinguished
subset $\CG^{(2)}\subset\CG\times\CG$, called the set of \textit{composable
pairs}, a product map $\CG^{(2)}\rightarrow\CG$, denoted by $(\gamma,\eta)\mapsto\gamma\eta$
and an inverse map $\CG\rightarrow\CG$, denoted by $\gamma\mapsto\gamma^{-1}$
such that the following hold 
\begin{enumerate}[label=(\roman*)]
\item If $(\alpha,\beta)\in\CG^{(2)}$ and $(\beta,\gamma)\in\CG^{(2)}$
then so are $(\alpha\beta,\gamma)$ and $(\alpha,\beta\gamma)$. In
addition, $(\alpha\beta)\gamma=\alpha(\beta\gamma)$ holds in $\CG$.
\item For all $\alpha\in\CG$ one has $(\gamma,\gamma^{-1})\in\CG^{(2)}$
and $(\gamma^{-1})^{-1}=\gamma$.
\item For any $(\alpha,\beta)\in\CG^{(2)}$ one has $\alpha^{-1}(\alpha\beta)=\beta$
and $(\alpha\beta)\beta^{-1}=\alpha$. 
\end{enumerate}
Every groupoid is equipped with a subset $\GU=\{\gamma\gamma^{-1}:\gamma\in\CG\}$
of $\CG$. We refer to elements of $\GU$ as \textit{units} and to
$\GU$ itself as the \textit{unit space}. We define two maps $s,r:\CG\rightarrow\GU$
by $s(\gamma)=\gamma^{-1}\gamma$ and $r(\gamma)=\gamma\gamma^{-1}$,
respectively, in which $s$ is called the \textit{source} map and
$r$ is called the \textit{range} map. 
\end{defn}
\label{paragraph:sections-in-groupoids}
When a groupoid $\CG$ is endowed with a locally compact Hausdorff
topology under which the product and inverse maps are continuous,
the groupoid $\CG$ is called a locally compact Hausdorff groupoid.
A locally compact Hausdorff groupoid $\CG$ is called \textit{\'{e}tale}
if the range map $r$ is a local homeomorphism from $\CG$ to itself,
which means for any $\gamma\in\CG$ there is an open neighborhood
$U$ of $\gamma$ such that $r(U)$ is open and $r|_{U}$ is a homeomorphism. Let $u\in \GU$. We define $\CG_u:=\{\gamma\in \CG: s(\gamma)=u\}$ and $\CG^u:=\{\gamma\in \CG: r(\gamma)=u\}$, which are called \textit{source fiber} and \textit{range fiber} at $u$, respectively.

A set $B$ in $\CG$ is called a \emph{bisection} if there exists an open set $U$ containing $B$ such that $s|_U$ and $r|_U$ are homeomorphisms from $U$ to $s(U)$ and $r(U)$, respectively. 
It is not hard to see a locally compact
Hausdorff groupoid is \'{e}tale if and only if its topology has a basis
consisting of open bisections. We say a locally compact Hausdorff
\'{e}tale groupoid $\CG$ is \textit{ample} if its topology has a basis
consisting of compact open bisections. A groupoid $\CG$ is said to be \textit{minimal} if $r(\CG\cdot u)$ is dense in $\GU$ for every $u\in \GU$. Moreover,  the groupoid $\CG$ is said to be \textit{principal} if the isotropy subgroupoid $\operatorname{Iso}(\CG):=\{\gamma\in \CG:s(\gamma)=r(\gamma)\}=\GU$. \textbf{Throught the paper, for simplicity, we mean by ``groupoids'' locally compact Hausdorff \'{e}tale groupoids}.

\begin{defn}\label{defn: tower}
    Let $\CG$ be a groupoid and $\CB=\{B_{i,j}: i, j\in I\}$ be a finite collection of open bisections. We say $\CB$ is an \textit{open tower} (or an \textit{open multisection}) if it satisfies
    \begin{enumerate}[label=(\roman*)]
        \item $B_{i,j}B_{j, k}=B_{i,k}$ for $i,j,k\in I$ and 
        \item $\{B_{i, i}: i\in I\}$ is a disjoint family of subsets of $\GU$.
    \end{enumerate}
\end{defn}

\begin{prop}\label{prop: groupoid tower}
    Let $\CG$ be a  groupoid with a compact unit space $\GU$. Let $u\in \GU$ and $\gamma_1, \dots, \gamma_n\in \CG$ be such that $s(\gamma_i)=u$ for any $i=1,\dots, n$ and $r(\gamma_1),\dots, r(\gamma_n)$ are pairwise distinct. Then there exists bisections $U_1, \dots, U_n$ such that
    \begin{enumerate}[label=(\roman*)]
        \item $s(U_1)=s(U_2)=\dots=s(U_n)$ and 
        \item $\{r(U_i): i=1,\dots,n\}$ is a disjoint family.
    \end{enumerate}
\end{prop}
\begin{proof}
First choose disjoint open neighborhoods $O_i$ of $r(\gamma_i)$ by the Hausdorff-ness of $\GU$. Then because $\CG$ is \'{e}tale, choose open bisections $B_i\ni \gamma_i$. Define $V_i=O_i\cdot B_i$ for each $i=1,\dots, n$ and define $O=\bigcap_{i=1}^ns(V_i)$, which contains $u$. Now it is ready to see that open bisections $U_i=V_i\cdot O$ for $i=1,\dots, n$ satisfy the required conditions.
\end{proof}

\begin{rmk}
Let $\CU=\{U_1, \dots, U_n\}$ be the family of open bisections satisfying Proposition \ref{prop: groupoid tower}. Then define open bisections $B_{i, j}=U_i\cdot U^{-1}_j$
for any $i, j=1,\dots, n$. Then direct calculation shows that the family
\[\CB=\{B_{i,j}: i,j=1,\dots,n\}\]
is an open tower in the sense of Definition \ref{defn: tower}. Moreover, if all $U_i$ are precompact, then so are all $B_{i, j}$.
\end{rmk} 

\begin{defn}\label{defn: castle}
 Let $J$ be a finite set and for each $l$, denote by $I_l$ a finite set.   A collection $\CC=\{B^l_{i, j}: i,j\in I_l, l\in J\}$ of open bisections is said to be an \textit{open castles} if the following holds.
 \begin{enumerate}
     \item For each $l\in J$, the collection $\{B^l_{i, j}: i, j\in I_l\}$ is an open tower.
     \item the intersection $B^l_{i, j}\cap B^{l'}_{i',j'}=\emptyset$ whenever $l\neq l'\in J$.
 \end{enumerate}
\end{defn}

\begin{rmk}\label{rmk: basic property of castles}
 An open castle can be viewed as a collection of disjoint open towers. It is straightforward to see that if $\CC$ is an open castle, then the union $\bigcup \CC$ is a subgroupoid of $\CG$, denoted by $\CG(\CC)$. If each $B^l_{i, j}$ is precompact, then so is the subgroupoid $\CG(\CC)=\bigcup \CC$. Such the groupoid $\CG(\CC)$ is also referred as an open \text{elementary} groupoid. 
 This construction particularly applies to an open tower $\CB$, as an open castle.
\end{rmk}

\begin{defn}
    Let $\CG$ be a groupoid. Denote by $Z^2(\CG, \T)$ the set of all continuous $2$-cocycles, which are continuous functions $\sigma:\CG^{(2)}\to \T$ satisfying
    \begin{enumerate}
        \item $\sigma(x,y)\sigma(xy, z)=\sigma(x, yz)\sigma(y,z)$ whenever $(x, y), (y,z)\in \CG^{(2)}$, and
        \item $\sigma(x, s(x))=\sigma(r(x), x)=1$ for any $x\in \CG$.
    \end{enumerate}
\end{defn}

We denote by $C_c(\CG)$ the set of compact supported continuous functions $f:\CG\to \C$. Let $\sigma\in Z^2(\CG, \T)$. For $f, g\in C_c(\CG)$, define the convolution $f*_\sigma g\in C_c(\CG)$ by
\[f*_\sigma g(x)=\sum_{x=yz}\sigma(y,z)\cdot f(y)\cdot g(z)\]
and the involution by
\[f^*(x)=\overline{f(x^{-1})}\cdot\overline{\sigma(x, x^{-1})}.\]
These make $C_c(\CG)$ a $*$-algebra, denoted by $C_c(\CG, \sigma)$. 

\begin{defn}
  Let $\CG$ be a groupoid and $\sigma$ a continuous $2$-cocycle in $Z^2(\CG, \T)$.  For $f\in C_c(\CG, \sigma)$, define the \textit{I-norm} of $f$ by
  \[\|f\|_I=\max\{\sup_{u\in \GU}\sum_{x\in \CG_u}|f(x)|, \sup_{u\in \GU}\sum_{x\in \CG^u}|f(x)|\}\]
\end{defn}

It follows from \cite[Proposition II.1.4]{R}
that for any $\sigma\in Z^2(\CG, \T)$, the I-norm $\|\cdot\|_I$ is a $*$-algebraic norm on $C_c(\CG, \sigma)$. Therefore, after the completion, $\ell_1(\CG, \sigma):=\overline{C_c(\CG, \sigma)}^{\|\cdot\|_I}$ is a Banach $*$-algebra. The \textit{full twisted groupoid $C^*$-algebra} is defined to be the $\overline{C_c(\CG, \sigma)}^{\|\cdot\|_u}$, denoted by $C^*(\CG, \sigma)$ in which 
\[\|f\|_u:=\sup\{\|\pi(f)\|: \pi: C_c(\CG)\to B(H) \text{ is a } *\text{-representation bounded by } \|\cdot\|_I\}\]
Then $\sigma$-\textit{left regular representation} of $C_c(\CG, \sigma)$ is a faithful $*$-representation $\pi_r=\oplus_{u\in \GU}\pi_u$, which is the direct sum of all $\pi_u: C_c(\CG)\to B(\ell_2(\CG_u))$ defined by
\[\pi_u(f)h=f*_\sigma h.\]
Note that the $\pi_r$ is still bounded by the I-norm. The reduced twisted groupoid $C^*_r(\CG, \sigma)$ is defined to be the completion of $C_c(\CG, \sigma)$ under the norm $\|\cdot\|_r$ induced by $\pi_r$. Moreover, since all $*$-representations $\pi: C_c(\CG, \sigma)\to B(H)$ bounded by the $\|\cdot\|_I$ can be extended to $\ell_1(\CG, \sigma)$, the Banach $*$-algebra $\ell_1(\CG, \sigma)$ can be viewed as a dense $*$-subalgebra in both $C_r^*(\CG, \sigma)$ and $C^*(\CG, \sigma)$. Moreover, it is known (see, e.g., \cite[Proposition 3.3.3]{Sims}) that there exists an injective, norm-decreasing embedding $j: C^*_r(\CG)\to C_0(\CG)$ defined by
\[j(a)(x)=\langle\pi_{s(x)}(a)\delta_{s(x)}, \delta_x\rangle\]
satisfying $j(f)=f$ for any $f\in C_c(\CG)$.
We denote by $E: C^*_r(\CG)\to C_0(\GU)$ the faithful \textit{canonical conditional expectation}. Then it is also a standard fact (see, e.g., \cite[Proposition 4.2.6]{Sims}) that $j(E(a))=j(a)|_{\GU}$ for any $a\in C^*_r(\CG)$. Then we record the following result with a proof.
\begin{prop}\label{prop: conditional expectation}
   Let $\CG$ be a groupoid on a compact unit space $\GU$ with a continuous $2$-cocycle $\sigma$. Let $a\in \ell_1(\CG,\sigma)\cap C_0(\CG)$. Then $E(a)=a|_{\GU}$.
   %\[\int_{\GU} E(a)\text{d}\mu=\int_{\GU} a|_{\GU}\text{d}\mu\]
\end{prop}
\begin{proof}
   Since $\GU$ is compact, one has $E(a)\in C(\GU)\subset C_c(\CG)$, which implies \[E(a)=j(E(a))=j(a)|_{\GU}\]
   by \cite[Proposition 3.3.3, Proposition 4.2.6]{Sims}. Now, because $a\in \ell_1(\CG,\sigma)\cap C_0(\CG)$, one still has $\pi_u(a)h=a*_{\sigma}h$ for any $h\in \ell_2(\CG_u)$ and $u\in \GU$, and thus the straightforward calculation shows $j(a)=a$. This implies $E(a)=a|_{\GU}$.
   \end{proof}

We finally note that if $\CG$ is \textit{topological amenable} (see, e.g., \cite[Definition 4.1.2]{Sims}), then $C^*_r(\CG, \sigma)$ coincides with $C^*(\CG, \sigma)$. In particular, it is known that the transversal groupoid $R_\varphi$ defined below in Definition \ref{defn: groupoid def} (and thus tiling groupoid $R_\Lambda$ in Definition \ref{defn: tiling groupoid}) are topological amenable.

\subsection{Transverse groupoids and tiling groupoids of Delone sets}\label{subsec2}
We now introduce the groupoid constructed from the free action $\R^d\curvearrowright \Omega$ on a compact metrizable space $\Omega$, that is useful to investigate tiling groupoids introduced below.

\begin{defn}\label{defn: flat cantor transversal}
    Let $d\in \N$. Let $\varphi$ be a free action of $\R^d$ on a compact metrizable space $\Omega$. We call a closed subset $X\subset \Omega$ a \textit{flat Cantor transversal} if the following are satisfied.
    \begin{enumerate}[label=(\roman*)]
        \item $X$ is homeomorphic to a Cantor set.
        \item For any $x\in \Omega$, there exists a $p\in \R^d$ such that $\varphi^p(x)\in X$.
        \item There exits a positive real number $M>0$ such that 
        \[C=\{\varphi^p(x): x\in X, p\in B(0, M)\}\]
        is open in $\Omega$ and 
        \[X\times B(0, M)\ni (x, p)\to \varphi^p(x)\in C\]
        is a homeomorphism.
        \item\label{cantor transversal etale} For any $x\in X$ and $r>0$, there exists an open neighborhood $U\subset X$ of $x$ in $X$ such that 
        \[\{p\in B(0, r): \varphi^p(x)\in X\}=\{p\in B(0, r): \varphi^p(y)\in X\}\]
        holds for any $y\in U$.
    \end{enumerate}
\end{defn}

The following is a well-known construction of \'{e}tale groupoids from flat Cantor transversals, which plays main role in this paper.

\begin{defn}\label{defn: groupoid def}
Let $\R^d\curvearrowright \Omega$ be a free action on a compact metrizable space $\Omega$. Let $X$ be a flat Cantor transversal for $\R^d\curvearrowright \Omega$. Define \textit{transversal groupoid}:
\[R_\varphi=\{(\varphi^p(x),x): x, \varphi^p(x)\in X\text{ and }p\in \R^d\},\]
which is a locally compact Hausdorff \'{e}tale ample groupoid with a compact metrizable unit space $X$.
\end{defn}

\begin{rmk}\label{rmk: basic info on the R_varphi}
We refer to \cite[Section 2]{GMPS} for some basic properties for such the groupoids $R_\varphi$. 
\begin{enumerate}
    \item First, $R_\varphi$ is equipped with the topology with the basis of sets of the form $\{(\varphi^p(x),x): x\in U\}\cap R_\varphi$
for some open set $U$ in $X$ and $p\in \R^d$. Then the condition \ref{cantor transversal etale} implies that there exists clopen bisections of the form
\[B_{p, V}=\{(\varphi^p(y), y): y\in V\}\]
form a topological basis for the topology. This shows that $R_\varphi$ is ample.
\item If $\R^d\curvearrowright \Omega$ is minimal then the groupoid $R_\varphi$ is minimal.
\item Since the action $\R^d\curvearrowright \Omega$ in Definition \ref{defn: groupoid def} is set to be free, the groupoid $R_\varphi$ is principal.
\end{enumerate}
\end{rmk}

Recall the definitions of groupoids associated to Delone sets in $\R^d$ or equivalently, a tiling of $\R^d$ with punctures, which are examples of transversal groupoids in Definition \ref{defn: groupoid def}. We refer to, e.g., \cite{ER}, \cite{BM}, \cite{Ito} and \cite{K} for more details on the notions demonstrated below.

\begin{defn}\label{defn: tiling groupoid}
Let $\Lambda$ be a FLC, aperiodic, repetitive Delone set. Define 
\[R_\Lambda=\{(T, T-z): T\in \Omega_0(\Lambda), z\in T\},\]
equipped with the multiplication defined by $(T_1, T_1-z_1)\cdot (T_2, T_2-z_2)=(T_1, T_2-z_2)$ whenever $T_1-z_1=T_2$ and the inverse defined by $(T, T-z)^{-1}=(T-z, T)$. Define the source map $s(T, T-z)=T-z$ and the range map $r(T, T-z)=T$.
\end{defn}

\begin{rmk}\label{rmk: tiling groupoid well defined}
We remark that the study of Delone set is also formulated by using (punctured) tilings of $\R^d$ in the literature, in which Delone sets exactly play the role of punctures of the tilings. All concepts above can be similarly defined to punctured tilings of $\R^d$  We refer to, e.g., \cite{Kel} and \cite{Ito} for more details. In addition, given a FLC, aperiodic, and repetitive Delone set $\Lambda$, the $R_\Lambda$ in Definition \ref{defn: tiling groupoid}, referred as a \textit{tiling groupoid}, is thus a \textbf{locally compact Hausdorff minimal principal second countable  \'{e}tale ample groupoids on a compact metrizable unit space}. Moreover, it is known (see, e.g., \cite[Example 2.3]{GMPS} and \cite[Lemma 4.4.3]{Ito}) that the discrete hull $\Omega_0(\Lambda)$ is a flat Cantor transversal of the continuous hull $\Omega(\Lambda)$ and therefore the groupoid $R_\Lambda$ is an example of transversal groupoids in Definition \ref{defn: groupoid def}.
\end{rmk}

\subsection{Hilbert $C^*$-modules}\label{sec: Hilbert mod}
We recall basic definitions of Hilbert $C^*$-modules and refer to classical references  \cite{Lance} and \cite{RW} for more details. However, we follows the convention in \cite{BEV}  using \textit{left} Hilbert $C^*$-module. 

\begin{defn}\label{defn: Hilbert Module}
    Let $A$ be a unital $C^*$-algebra. An \textit{inner product $A$-module} is an complex vector space $\CE$ with a left $A$-module structure and a map ${_\bullet \langle} \cdot, \cdot \rangle: \CE\times \CE\to A$ satisfies the following.
    \begin{enumerate}[label=(\roman*)]
        \item\label{Hilbert m1} ${_\bullet \langle} \cdot, \cdot \rangle$ is $\C$-linear on the first variable.
        \item\label{Hilbert m2}${_\bullet \langle} a\cdot x, y \rangle=a\cdot {_\bullet \langle} x, y \rangle$ hold for any $a\in A$ and $x, y\in \CE$.
        \item\label{Hilbert m3}${_\bullet \langle}x, y \rangle^*={_\bullet \langle} y, x \rangle$ hold for any $x, y\in \CE$.
        \item\label{Hilbert m4}${_\bullet \langle}x, x \rangle\geq 0$ in $A$.
        \item\label{Hilbert m5}${_\bullet \langle} x, x \rangle=0$ if and only if $x=0$.
        \end{enumerate}
    \end{defn}

An inner product $A$-module $\CE$ becomes a normed space under the $2$-norm defined by $\|x\|_\CE=\|{_\bullet \langle}x, x \rangle\|^{1/2}$ for $x\in \CE$. If $\|\cdot\|_\CE$ is a complete norm, then $\CE$ is said to be a \textit{Hilbert $A$-module}.

Let $A_0$ be a dense $*$-subalgebra of $A$. Then a \textit{pre-inner product $A_0$-module} is a complex vector space $\CE_0$ equipped with a left $A_0$-module structure and a map ${_\bullet \langle} \cdot, \cdot \rangle: \CE_0\times \CE_0\to A$ such that the conditions \ref{Hilbert m1}-\ref{Hilbert m4} in Definition \ref{defn: Hilbert Module} holds for all $a\in A_0$ and $x, y\in \CE_0$. We remark that, unlike in the literature like \cite{RW} and \cite{BEV}, we do not ask the image of the pre-inner product ${_\bullet \langle} \cdot, \cdot \rangle$ is in $A_0$. However, we still have the Cauchy-Schwarz inequality in our setting. This should be compared to, e.g., \cite[Lemma 2.5]{RW}.

\begin{lem}[Cauchy-Schwarz inequality]\label{lem: CS}
    It $\CE_0$ is a pre-inner product $A_0$-module and if $x, y\in \CE_0$, then
    \[ {_\bullet \langle} y, x \rangle^*\cdot{_\bullet \langle} y, x \rangle\le \|{_\bullet \langle} y, y \rangle\|{_\bullet \langle} x,x  \rangle\]
\end{lem}
\begin{proof}
    By \cite[Remark 2.6]{RW}, it suffices to show 
    \[\rho({_\bullet \langle} y, x \rangle^*\cdot{_\bullet \langle} y, x \rangle)\leq \|{_\bullet \langle} y, y \rangle\|\rho({_\bullet \langle} x,x  \rangle)\]
    for any state $\rho\in S(A)$. Then for ${_\bullet \langle} x, y \rangle\in A$, choose a sequence $a_n\in A_0$ such that $a_n\to {_\bullet \langle} x, y \rangle$ as $n\to \infty$. Then $A_0$-module properties \ref{Hilbert m2} and \ref{Hilbert m3} in Definition \ref{defn: Hilbert Module} imply that  
     \[{_\bullet \langle} x, a_n \cdot y \rangle={_\bullet \langle} y, x \rangle^* \cdot a_n^* \to{_\bullet \langle} y, x \rangle^* \cdot {_\bullet \langle} y, x \rangle, \]
     and
     \[{_\bullet \langle}a_n \cdot y, a_n  \cdot y \rangle=a_n \cdot{_\bullet \langle} y, y \rangle \cdot a_n^* \to {_\bullet \langle} y, x \rangle^*\cdot{_\bullet \langle} y, y \rangle \cdot{_\bullet \langle} y, x \rangle .\]
   Using the usual Cauchy-Schwarz inequality for the positive sesquilinear form $(y,x)\mapsto \rho({_\bullet \langle}y, x\rangle)$ and $b^*cb\leq \|c\|b^*b$ for any $c\in A_+$ and $b\in A$, one obtains
\begin{align}
    \rho(({_\bullet \langle} y, x\rangle^*\cdot{_\bullet \langle} y, x\rangle))&=\lim_{n\to \infty}\rho({_\bullet \langle} x, a_n \cdot y \rangle)\notag\\
    &\le \lim _{n\to \infty}\rho({_\bullet \langle} x,x \rangle)^{1/2}\rho({_\bullet \langle} a_n \cdot y,a_n \cdot y\rangle)^{1/2}\notag\\
    &=\rho({_\bullet \langle} x,x \rangle))^{1/2}\rho({_\bullet \langle} y, x \rangle^*\cdot {_\bullet \langle} y, y \rangle \cdot{_\bullet \langle} y, x \rangle )^{1/2}\notag\\
    &\le \|{_\bullet \langle} y, y \rangle\|^{1/2}\rho({_\bullet \langle} x,x \rangle))^{1/2}\rho({_\bullet \langle} y, x\rangle^*\cdot{_\bullet \langle} y, x\rangle))^{1/2}.\notag
\end{align}  
This establishes the result.
     \end{proof}

The following is a slight generalization of well-known completion theorem for pre-inner product module to a genuine inner product module as in \cite[Lemma 2.16]{RW}. The proof is the same to \cite[Lemma 2.16]{RW} using Lemma \ref{lem: CS} above and we omit it.
\begin{prop}\label{prop: completion}
Let $A_0$ be a dense $*$-subalgebra of a unital $C^*$-algebra $A$ and $\CE_0$ is a pre-inner product $A_0$-module with the pre-inner product ${_\bullet \langle} \cdot, \cdot \rangle_{\CE_0}: \CE_0\times \CE_0\to A$, which satisfies Definition \ref{defn: Hilbert Module}\ref{Hilbert m1}-\ref{Hilbert m4}. Then there exists a Hilbert $A$-module $\CE$ and a linear map $q:\CE_0\to \CE$ such that $q(\CE_0)$ is dense in $\CE$ and $a\cdot q(x)=q(a\cdot x)$ for all $x\in \CE_0$, $a\in A_0$, and ${_\bullet \langle} q(x), \cdot q(y) \rangle_{\CE}={_\bullet \langle} q(x), \cdot q(y) \rangle_{\CE_0}$
\end{prop}

Let $\CE$ and $\CF$ be Hilbert $A$-module. An adjoint of a map $T:\CE\to \CF$ is a (uniquely determined) map $T^*:\CF\to \CE$ that satisfies ${_\bullet \langle} Tx, y \rangle={_\bullet \langle} x, T^*y \rangle$ for all $x\in \CE$ and $y\in \CF.$ In this case $T$ and $T^*$ are bounded and $A$-linear. The space of all adjointable operators from $\mathcal{E}$ to $\mathcal{F}$ is denoted by $\CL_A(\CE,\CF).$ If $\CE=\CF,$ this space is also denoted by $\CL_A(\CE)$ and it is a $C^*$-algebra with the natural operations and operator norm. 

Frank and Larson \cite{FL} introduced module frames for countably generated Hilbert $C^*$-module. Let $\CE$ be a Hilbert $A$-module. The module analysis operator $\mathscr{C}:\CE\to A^n$ and the module synthesis operator $\mathscr{D}:A^n \to \CE$ associated to finite set $\{y_1,\dots,y_n\}\subseteq \CE$ is defined by 
\[\mathscr{C}x=({_\bullet \langle} x, y_i \rangle)_{i=1}^n,\]
and
\[\mathscr{D}(a_i)_{i=1}^n=\sum_{i=1}^na_iy_i\]
for $x\in \CE$ and $(a_i)_i\in A^n.$ Both these operators are adjointable with $\mathscr{C}^*=\mathscr{D}.$ The operator $\mathscr{S}=\mathscr{D}\mathscr{C}$ is the module frame operator while $\mathscr{G}=\mathscr{C}\mathscr{D}$ is the module Gramian operator. A finite set $\{y_1,\dots,y_n\}\subseteq \CE$ is a module frame for $\CE$ \cite{FL} if there exist constant $A,B>0$ such that 
    \begin{equation} \label{frame of Hilbert module}
        A {_\bullet \langle} x, x \rangle \le \sum_{j=1}^n  {_\bullet \langle} x, y_j \rangle  {_\bullet \langle} x, y_j \rangle^* \le B {_\bullet \langle} x, x \rangle  \quad \text{for all} \ x\in \CE.
    \end{equation}
    That is to say the module frame operator $\mathscr{S}$ of $\{y_i\}_{i=1}^n$ satisfies $AI_{\CE}\le \mathscr{S} \le B I_{\CE}$. 
If one can choose $A=B=1$ in (\ref{frame of Hilbert module}), the frame called Parseval. For a finite set $\{y_i\}_{i=1}^n \subseteq \CE$, its $A$-span is the set of all finite $A$-linear combinations of elements in this set. We call $\{y_i\}_{i=1}^n$ is algebraically finitely generated for $\CE$ if its $A$-span is $\CE$. And $\{y_i\}_{i=1}^n \subseteq \CE$ is called $A$-linearly independent if whenever $\{a_i\}_{i=1}^n \subset A$ is such that $\sum_{i=1}^na_ix_i=0 $, then $a_i=0$ for $1 \le i \le n$.

Let $\CE,\CF$ be a Hilbert $A$-module. We assume that $\tau$ is a faithful tracial state on $A$ and $H$ is the Hilbert space obtained from the GNS construction of $(A,\tau).$ Let $M=A'' \subseteq \mathcal{B}(H)$ be the von Neumann algebra generated by $A.$ Then the cyclic and separating vector $\xi_{\tau}$ of $(A,\tau)$ gives rise to a faithful normal trace on $M,$ which we also denote by $\tau.$  We write $L^2(M,\tau)$ for the Hilbert space underlying the GNS representation $(M,\tau),$ then $L^2(M,\tau)$ can be identified with $H.$ Following \cite[Section 3.4]{BEV}, we define a scalar-valued inner product on $\CE$ by $\langle x,y \rangle_{H_\CE^\tau} =\tau({_\bullet \langle} x, y\rangle)$ for $x,y\in \mathcal{E}$ and denote by $H_\CE^{\tau}:=\overline{\CE}^{\|\cdot\|_{H_\CE^\tau}}.$  The left action of $A$ on $\CE$ will extend to a representation $\pi_\CE^\tau$ of $A$ on $H_\CE^\tau.$ The triple pair $(H_\CE^\tau,\pi_\CE,\tau)$ is called the \textit{localization space} of $\CE$ with respect to $(A,\tau).$ According to \cite[Lemma 3.6]{BEV} every adjoint operator $T\in \CL_A(\CE,\CF)$ extends uniquely to a bounded, $M$-linear map $T^\tau:H_\CE^\tau \to H_\CF^\tau.$

\section{Strict comparison for twisted reduced groupoid $C^*$-algebras $C^*_r(\CG, \Sigma)$}\label{sec: tower}
As mentioned in the introduction, a key ingredient in establishing Theorem \ref{thm: main 1} is the strict comparison (of projections) of the $C^*$-algebra $C^*_r(R_\Lambda, \sigma_\Lambda)$. To this end, it suffices to verify the finite \textit{nuclear dimension} or the $\CZ$-\textit{stability} of  $C^*_r(R_\Lambda, \sigma_\Lambda)$ by \cite{Winter} and \cite{Ror}. 

We avoid recalling lengthy definitions. Instead, we refer to \cite{WZ} for the definition of the nuclear dimension, which is known as a noncommutative analogue of the covering dimension for topological spaces. Moreover, the notation $\CZ$ above denotes the \textit{Jiang-Su algebra} and $\CZ$-stability of a $C^*$-algebra $A$ means $A\otimes \CZ\simeq A$. These two properties, together with the \textit{strict comparison of positive elements}( see, e.g., \cite{Ror}) play important roles in the modern classification programme of nuclear simple separable $C^*$-algebras satisfying the UCT. We refer to, e.g., \cite{CGSTW}, \cite{EGLN}, \cite{GLN}, \cite{TWW} for more details.  However, in this paper, we only need the following  strict comparison of projections, which is formally weaker than the above three properties in general.

\begin{defn}\label{defn: comparison for projections}
    Let $A$ be a unital $C^*$-algebra. Denote by $T(A)$ the trace space of $A$. We say $A$ has strict comparison of projections, if whenever $p, q\in M_n(A)$ are projections in $n$-dimensional matrix algebra of $A$ with $\tau(p)<\tau(q)$, then $p\preceq q$ in the sense that there exists some $v\in M_n(A)$ such that $p=v^*v$ and $vv^*\leq q$.
\end{defn}

Let $\CG$ be a groupoid. The following concept was introduced in \cite{GWY} as an dynamical analogue of the asymptotic dimension in geometry, which can be used to bound nuclear dimension of groupoid $C^*$-algebras.

\begin{defn}\cite[Definition 5.1]{GWY}
Let $\CG$ be a groupoid. We say $\CG$ has \textit{dynamic asymptotic dimension} $d\in \N$, denoted by $\operatorname{d. a. d}(\CG)=d$ if $d$ is the smallest number with the following property: For every precompact open subset $K$ of $\CG$, there are open subsets $U_0, \dots, U_d$ of $\GU$ that covers $s(K)\cup r(K)$ such that for each $i$, the set $\{\gamma\in K: s(\gamma), r(\gamma)\in U_i\}$ is contained in a relatively compact subgroupoid of $\CG$.
\end{defn}

Based this concept, it follows from the next result, proved in \cite{BL}, on the estimation of the nuclear dimension of the twisted reduced groupoid $C^*$-algebra $C^*_r(\CG, \Sigma)$. This implies that $C^*(\CG, \Sigma)$ has the strict comparison of positive elements and thus the strict comparison of projections. See also \cite{GWY} and \cite{CDGHV}.

\begin{thm}\cite[Theorem B]{BL}\label{thm: finite dad implies finite dim nuc}
    Let $\CG$ be a second countable groupoid and let $\Sigma$ be a twist over $\CG$. Then,
    \[\dim^{+1}_{\operatorname{nuc}}(C^*_r(\CG;\Sigma))\leq \operatorname{d.a .d}^{+1}(\CG)\cdot\dim^{+1}(\GU).\]
\end{thm}

To calculate the dynamic asymptotic dimension of $R_\varphi$ introduced in Definition \ref{defn: groupoid def}, we introduce a version of \textit{tower dimension} for groupoids, motivated by \cite[Definition 4.2, Definition 4.3]{Kerr}.

\begin{defn}\label{defn: lebesgue}
    Let $\CG$ be a groupoid with a compact unit space.
     Let $K\subset \CG$ be compact and let $\{\CC_i: i\in I\}$ be a collection of open castles such that the levels $\bigcup_{i\in I}\CC^{(0)}$ covers $s(K)$, i.e., $s(K)\subset\bigcup_{i\in I}\bigcup \CC^{(0)}_i$. We say $\{\CC_i: i\in I\}$ is $K$-\textit{Lebesgue} if for any $u\in s(K)$, there exists an open castle $\CC_i$ for some $i\in I$ such that $K\cdot u\subset (\bigcup \CC_i)\cdot u$.
\end{defn}

\begin{defn}\label{defn: tower dim}
Let $\CG$ be a groupoid with a compact unit space. We say $\CG$ has tower dimension at most $d\in \N$, denoted by $\dimtow(\CG)\leq d$ if for any compact $K\subset \CG$ there exists a $K$-Lebesgue collection $\{\CC_0,\dots, \CC_d\}$, consisting of precompact open castles, of size $d+1$. 
\end{defn}

\begin{prop}\label{prop: dad implies tower dim}
Let $\CG$ be an ample groupoid with a compact unit space. Suppose $\dimtow(\CG)\leq d$. Then $\operatorname{d.a.d}(\CG)\leq d$.
\end{prop}
\begin{proof}
    Let $K\subset \CG$ be a precompact open set. Since $\CG$ is ample, one chooses finitely many  compact open bisections $O_1,\dots, O_m$ such that 
    \[K\cup K^{-1}\subset \bigcup_{i=1}^mO_i.\]
    By a standard chopping technique for all of these compact open sets $O_i$ and $s(O_i)$, one may assume that either $s(O_i)=s(O_j)$ or $s(O_i)\cap s(O_j)=\emptyset$ for any $1\leq i, j\leq m$.
    
    Denote by  $L=\bigcup_{i=1}^m O_i$, which is compact open. Then since $\dimtow(\CG)\leq d$, there exits a family $\{\CC_0, \dots, \CC_d\}$ of $L$-Lebesgue open castles such that 
    \[s(K)\cup r(K)\subset s(L)\subset \bigcup_{i=0}^d\bigcup \CC^{(0)}_i.\]
    Then for $i=0,\dots, d$, define
   \[U_i=\{u\in s(L): L\cdot u\subset (\bigcup \CC_i)\cdot u\}.\]
    We claim each $U_i$ is open. Indeed, let $u\in U_i$, define $I_u=\{1\leq i\leq m: u\in s(O_i)\}$. By our condition for all $s(O_i)$, there exits an compact open $V_u\subset \GU$ such that $V_u=s(O_i)$ for any $i\in I_u$.
This implies that for any $v\in V_u$, one has
\begin{equation}\label{equ1}
  L\cdot v=\bigcup_{j=1}^mO_j\cdot v=\bigcup_{j\in I_u}O_j\cdot v  
\end{equation}
because $V_u\cap s(O_j)=\emptyset$ holds for any $j\notin I_u$. On the other hand, for the $u$, let $C\in \CC^{(0)}_i$ be the unique level of $\CC_i$ such that $u\in C$. Then note that
\[L\cdot u=\bigcup_{j\in I_u}O_j\cdot u\subset \bigcup \CC_i\cdot u.\] 
Since each $O_j$ is a bisection, the set $O_j\cdot u$ is a singleton and we denote by $\{\gamma_j\}=O_j\cdot u\subset (\bigcup \CC_i)\cdot u$. This implies that  $\gamma_j\in C_j$ for a unique $C_j\in \CC_i$ with $s(C_j)=C$. Then choose an open bisection $W_j$ such that $\gamma_j\in W_j\subset O_j\cap C_j$.  Then we define 
\[Z_u=V_u\cap C\cap \bigcap_{j\in I_u}s(W_j)\]
which is an open neighborhood of $u$, contained in $s(L)$.  Then for any $v\in Z_u$, using (\ref{equ1}), one has 
\[L\cdot v=\bigcup_{j=1}^mO_j\cdot v=\bigcup_{j\in I_u}O_j\cdot v=\bigcup_{j\in I_u}W_j\cdot v=\bigcup_{j\in I_u}C_j\cdot v\subset \bigcup \CC_i\cdot v\]
and thus $Z_u\subset U_i$. This implies that each $U_i$ is open. 

Finally, for any $u\in s(L)$, since $\{\CC_0, \dots, \CC_d\}$ is a $L$-Lebesgue collection of castles, there has to be an $0\leq i\leq d$ such that $L\cdot u\subset (\bigcup \CC_i)\cdot u$. By definition, this implies that open sets $U_i$ for $i=0,\dots, d$ form an open cover of $s(L)\supset s(K)\cup r(K)$. Then denote by
\[H_i:=\{\gamma\in K: s(\gamma), r(\gamma)\in U_i\}.\]
Then, for any $\gamma\in H_i\subset L$ one has
\[\gamma=\gamma\cdot s(\gamma)\in (\bigcup \CC_i) \cdot s(\gamma)\subset \bigcup \CC_i,\]
which entails that  $H_i\subset \bigcup \CC_i=\CG(\CC_i)$.
Note that $\CG(\CC_i)$ is a precompact open subgroupoid by Remark \ref{rmk: basic property of castles}. Therefore, the subgroupoid $\langle H_i\rangle$ generated by $H_i$ is also precompact. This shows that $\operatorname{d.a.d}(\CG)\leq d$.
\end{proof}

\begin{rmk}
We remark that Proposition \ref{prop: dad implies tower dim} still holds for general locally compact Hausdorff \'{e}tale groupoids $\CG$ with compact unit spaces by a more complicated proof demonstrated in \cite{LM}. 
\end{rmk}

\section{Finite tower dimension and finite dynamical asymptotic dimension of tiling groupoids $R_\Lambda$}\label{sec: finite dad}

In this section, we calculate the tower dimension and dynamical asymptotic dimension of $R_\varphi$.
For any $x\in X$ and $r>0$, the set 
\[O_r(x)=\{p\in \R^d: \varphi^p(x)\in X, \|p\|_\infty< r\}\] is said to be the $r$-\textit{partial orbit} of $x$. Note that by definition, one always has $0\in O_r(x)$.

\begin{lem}\label{lem: local towers}
Let $\R^d\curvearrowright \Omega$ be a free action on a compact metrizable space and let $X$ be the flat transversal and $R_\varphi=\{(\varphi^p(x), x): x\in X, \varphi^p(x)\in X, p\in \R^d\}$ the groupoid induced from a free action $\R^d$. Then for any $r>0$ there exists finitely many clopen set $U_1,\dots, U_n$ in $X$ such that
\begin{enumerate}[label=(\roman*)]
    \item\label{local towers 1} the $2r$-partial obits $O_{2r}(x)$ stay the same for any $x\in U_i$, denoted by $F_i$;
    \item\label{local towers 2} the union $\bigcup_{i=1}^n\bigcup_{p\in E_i}\varphi^p(U_i)=X$ in which each $E_i\subset F_i$ is the $r$-partial orbit; 
    \item\label{local towers 3} and $\varphi^p(U_i)\cap \varphi^q(U_j)=\emptyset$ holds for any $p\in F_i$ and $q\in F_j$ with $\|p-q\|_\infty<r$.
\end{enumerate}
\end{lem}
\begin{proof}
Let $r>0$ be given. Since $\R^d\curvearrowright \Omega$ is free, the groupoid $R_\varphi=\{(\varphi^p(x), x): x\in X, \varphi^p(x)\in X, p\in \R^d\}$ is principal. Then for the $r$ and any $x\in X$, the points $\varphi^p(x)$ are pairwise different for any $p\in O_{2r}(x)$. Then Proposition \ref{prop: groupoid tower} implies that there exists clopen neighborhood $V_x\ni x$ and a family  $\{B_{p, V_x}: p\in O_{2r}(x)\}$ of clopen bisections in $R_\varphi$ such that $\{\varphi^p(V_x): p\in O_{2r}(x)\}$ is a disjoint family. Moreover, shrink each $V_x$ if necessary, one may assume the $2r$-partial orbits $O_{2r}(y)$ are same for all $y\in V_x$. Then by the compactness of $X$, choose a finite subcover $V_1,\dots, V_n$ of $X$ with their corresponding $F_1, \dots, F_n$. In addition, denote by $E_i\subset F_i$ the subset of the $r$-partial orbit. Now define $U_1=V_1$ and \[U_i=V_i\setminus \bigcup_{j<i}\bigcup_{p\in E_j}\varphi^p(U_j).\]
Note that the family $\{(U_i, F_i): i=1,\dots, n\}$ satisfy the condition \ref{local towers 1} by the construction.

Suppose the condition \ref{local towers 2} fails. Let $x\notin \bigcup_{i=1}^n\bigcup_{p\in E_i}\varphi^p(U_i)$. Then because $V_1,\dots, V_n$ form cover of $X$, there exits an $i_0$ such that \[x\in V_{i_0}\setminus \bigcup_{i=1}^n\bigcup_{p\in E_i}\varphi^p(U_i)\subset V_{i_0}\setminus \bigcup_{i<i_0}\bigcup_{p\in E_i}\varphi^p(U_i)=U_{i_0},\]
which is a contradiction as $0\in E_{i_0}$ by definition. Thus the condition \ref{local towers 2} holds. 

Finally, suppose the condition \ref{local towers 3} fails, which means there exists $p\in F_i$ and $q\in F_j$ with $\|p-q\|_\infty<r$ such that $\varphi^p(U_i)\cap \varphi^q(U_j)\neq \emptyset$. This necessarily implies that $i\neq j$ because $\{\varphi^p(U_i): p\in F_i\}$ is a disjoint family.
Without loss of generality, one assumes $i<j$. Let $y\in \varphi^p(U_i)\cap \varphi^q(U_j)$. Then there exists open neighborhood $W\ni x= \varphi^{-p}(y)$ such that $W\subset U_i$ and $\varphi^p(W)\subset  \varphi^p(U_i)\cap \varphi^q(U_j)$. This implies that $\varphi^{p-q}(W)\subset U_j$. Note also $p-q\in O_r(x)$ as $\|p-q\|_\infty<r$. In addition, because $x\in U_i$ and the $r$-partial orbits for any points in $U_i$ equal $E_i$, one has $p-q\in E_i$. However, this implies that 
\[U_j\cap \bigcup_{l\in E_i}\varphi^l(U_i)\neq \emptyset,\]
which is a contradiction to the definition of $U_j$. Therefore, the condition \ref{local towers 3} holds.
\end{proof}

\begin{defn}\label{defn: center cube}
  Let $m>0$ and let $D$ be a $d$-dimensional cube in $\R^d$ such that each edge of $D$ is of length $3m$. We say the cube whose edges are of length $m$:
  \[D^0=\{z\in D: \operatorname{dist}_{\|\cdot\|_\infty}(z, \partial D)\geq m\}\]
  is the \textit{center subcube} of $D$, in which $\partial D$ is the boundary of $D$ in $\R^d$.
\end{defn}

The following is an elementary fact but useful in the proof of Theorem \ref{thm: R_varphi has finite tower dim}.

\begin{rmk}\label{rmk: partition of domain}
Let $D$ be a cube whose edges are of length $3m$ and $D^0$ the center subcube of $D$. Denote by $w_l$ for $l=0,\dots, 3^d-1$ vectors in $\R^d$ whose coordinates are all either $m$, or $-m$, or $0$  Then the family 
\[\{w_l+D_0: l=0,\dots, 3^d-1\}\]
form a cover of $D$ with $\operatorname{int}(w_l+D_0)\cap \operatorname{int}(w_k+D_0)=\emptyset$ whenever $l\neq k$.
\end{rmk}

\begin{thm}\label{thm: R_varphi has finite tower dim}
   Let $R_\varphi$ be a minimal principal groupoid introduced in Definition \ref{defn: groupoid def}. Then $R_\varphi$ has finite tower dimension bounded by $6^d$, i.e. $\dimtow(R_\varphi)\leq 6^d-1$.
\end{thm}
\begin{proof}
    Let $K\subset R_\varphi$ be a compact set. Without loss of generality, one may assume the unit space $X\subset K$. Then choose finitely many clopen bisections $B_j=B_{m_j, V_j}=\{(\varphi^{m_j}(x), x): x\in V_j\}$ for $j\in J$ such that $K\subset \bigcup_{j\in J}B_j$. Define 
    \[m=\max\{\|m_j\|_\infty: j\in J\}.\]
    and denote by
    \[\CO=\bigvee\{V_j, V_j^c: j\in J\}\]
    the common refinement of covers $\{V_j, V^c_j\}$ of $X$ for $j\in J$,
    which still form a clopen cover of $X$.
    Then, for the $r=3m$, Lemma \ref{lem: local towers} implies there exists finitely many clopen sets $U_1,\dots, U_n$ in $X$ and corresponding $2r$-partial orbit sets $F_1,\dots, F_n$ and $r$-partial orbit sets $E_1,\dots, E_n$ satisfying the conditions \ref{local towers 1}, \ref{local towers 2}, and \ref{local towers 3} in Lemma \ref{lem: local towers}. Since $X$ is a Cantor set. A standard technique (using Lebesgue number) allows to chop $\varphi^p(U_i)$ for all $p\in F_i$ and $i=1,\dots, n$ such that 
    the family
    \[\{\varphi^p(U_i): p\in F_i: i=1,\dots, n\}\]
    refines the cover $\CO$. Then note 
    \[B_{\R^d}(0, r)=\{p\in \R^d: \|p\|_\infty<r\}=\bigcup_{k=0}^{2^d-1}v_k+\{p=(p_1,\dots, p_d): 0\leq p_j< r\text{ for }j=1,\dots,d \}\]
    in which the coordinates in $v_i$ are either $0$ or $-r$. For each $0\leq k\leq 2^d-1$, we denote by 
    \[D_k=v_k+\{p=(p_1,\dots, p_d): 0\leq p_j< r=3m\text{ for }j=1,\dots, n \}\]
    for simplicity. Note that for each $p\in D_k\cap F_i=D_k\cap E_i$ and $q\in D_k\cap F_j=D_k\cap E_j$, the distance $\|p-q\|_\infty<r$. This implies that 
    \[\CT_k=\{\varphi^p(U_i): p\in E_i\cap D_k, i=1,\dots,n\}\]
    is a disjoint family by Lemma \ref{lem: local towers}\ref{local towers 3}. Moreover, Lemma \ref{lem: local towers}\ref{local towers 2} entails that 
    \[\bigcup_{k=0}^{2^d-1}\bigcup\CT_k=X.\]
 Then denote by $w_l$ for $l=0,\dots, 3^d-1$ the vectors in $\R^d$ whose all coordinates are either $m$, or $-m$, or $0$. We set $w_0=0$ for simplicity. For each $k$, we shift $D_k$ to $D_{l, k}=w_l+D_k$ for $l=0,\dots, 3^d-1$ by vector $w_l$.  By definition, note that all $D_{l, k}$ is contained in $B_{\R^d}(0, 4m)\subset B_{\R^d}(0, 2r)$ and $D_{0, k}=w_0+D_k=D_k$.
    Moreover, for any $p\in D_{l, k}\cap F_i$ and $q\in D_{l, k}\cap F_j$, note that $\|p-q\|_\infty<r$, and thus one has 
    \[\CT_{l, k}=\{\varphi^p(U_i): p\in F_i\cap D_{l, k}, i=1,\dots, n\}\]
    is a disjoint family by Lemma \ref{lem: local towers}\ref{local towers 3}. We now look at the collection 
    \[\{\CT_{l, k}: k=0,\dots, 2^d-1, l=0,\dots, 3^d-1\}.\]
    Note that $\CT_{0,k}=\CT_k$ by definition.

Then for any $i=1,\dots, n$, since $R_\varphi$ is principal, each pair $\varphi^{p_1}(U_i), \varphi^{p_2}(U_i)\in\CT_{l, k}$ uniquely determines a clopen bisection $B_{p_1,p_2, i}$ such that $s(B_{p_1,p_2, i})=\varphi^{p_1}(U_i)$ and $r(B_{p_1, p_2, i})=\varphi^{p_2}(U_i)$. This allows to define a collection of bisections 
\[\CC_{l,k}=\{B_{p_1, p_2, i}: p_1, p_2\in E_i\cap D_{l, k}, i=1,\dots, n\},\]
which form a castle in the sense of Definition \ref{defn: castle} such that $\CC_{l, k}^{(0)}=\CT_{l,k}$. Then recall
\[\bigcup_{k=0}^{2^d-1}\CT_k=\{\varphi^p(U_i): p\in E_i\cap D_k, i=1,\dots,n, k=0,\dots, 2^d-1 \}\]
form a cover of $X=s(K)$. Then, so does $\bigcup_{l=0}^{3^d-1}\bigcup_{k=0}^{2^d-1}\CC_{l, k}^{(0)}$.
Finally, let $x\in X$ and denote by 
\[J_0=\{j\in J: x\in V_j=s(B_j)\}.\]  
Since the collection 
\[\{\varphi^p(U_i): p\in E_i\cap D_k, i=1,\dots,n, k=0,\dots, 2^d-1 \}\]
covers $X$ and refines the cover $\CO$, there exists a $\varphi^p(U_i)$ such that 
\[x\in \varphi^p(U_i)\subset \bigcap_{j\in J_0}V_j\]
for some $k=0,\dots, 2^d-1$, some $i=1,\dots, n$ and some $p\in D_k\cap E_i$. Then $p$ is located in a $w_l+D^0_k$  for some $l=0,\dots, 3^d-1$ by Remark \ref{rmk: partition of domain}, where $D^0_k$ is the center subcube of $D_k$ in the sense of Definition \ref{defn: center cube}. Then by our construction, it is direct to see $w_l+D^0_k$ is the center subcube of the cube $w_l+D_k=D_{l, k}$. 
For any $j\in J_0$, since $\|m_j\|_\infty< m$, one obtains that 
\[r(B_jx)=\varphi^{m_j}(x)\in \varphi^{m_j+p}(U_i)\]
which is a member of $\CT_{l, k}$ because $p$ is in the center subcube of $D_{l, k}$. Then since $\varphi^p(U_i)$ and $\varphi^{m_i+p}(U_i)$ are all in $\CT_{l, k}$, one has
\[B_jx=\{(\varphi^{m_j}(x), x)\}\subset\bigcup\CC_{l,k}\]
by the principality of $R_\varphi$. This further implies that
\[Kx\subset \bigcup_{j\in J_0}B_jx\subset(\bigcup\CC_{l, k})x.\] 
As a consequence, one obtains that $\dimtow(R_\varphi)\leq 2^d\cdot 3^d-1=6^d-1$ by Definition \ref{defn: lebesgue} and \ref{defn: tower dim}.
\end{proof}

\begin{rmk}\label{rmk: proof idea}
We illustrate the idea of the proof in the two-dimensional case in Figure~\ref{fig:figure1}, i.e., $d=2$.  Let $r = 3m$, and partition the disk $B_{\mathbb{R}^2, \|\cdot\|_\infty}(0, r)$ into $2^2 = 4$ regions, each lying in one of the four quadrants of the plane. 
Consider the region $D$ in the second quadrant. We further divide $D$ into $3^2 = 9$ smaller squares, labeled $1, \dots, 9$, each of size $m \times m$, with the center cube $D_0$ labeled by $5$. 
Observe that each of these nine subsquares, being a translate of $D_0$, serves as the center cube of a larger square of size $3m \times 3m$, and that these larger squares can be assigned in total $9=3^2$ distinct colors.
\end{rmk}

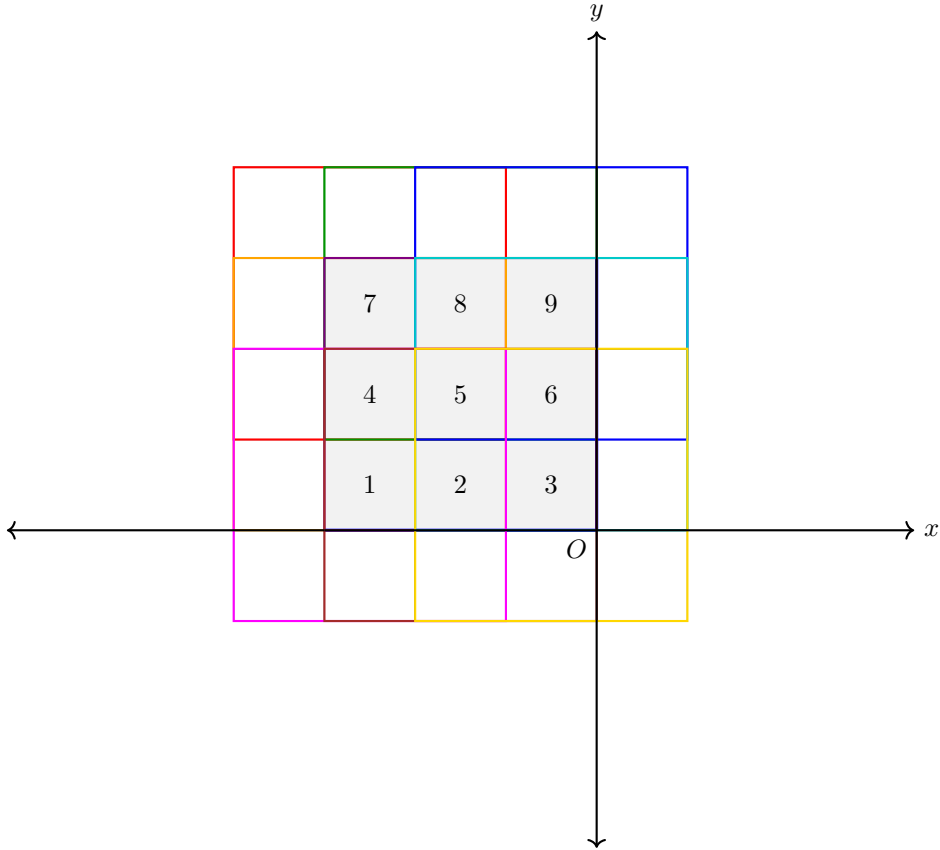
\begin{figure}[htbp]
  \centering % 居中对齐
  \begin{tikzpicture}[scale=1.2]
  
  % ========== 5×5 网格 ==========
  \foreach \x in {0,...,4}
    \foreach \y in {0,...,4}
    {
      \draw[thin] (\x,\y) rectangle (\x+1,\y+1);
    }
  
  % ========== 标记九宫格位置（中心3×3） ==========
  \fill[gray!10] (1,1) rectangle (4,4);
  
  % ========== 九宫格内部分隔线 ==========
  \draw[thin, black] (1,1) grid (4,4);
  
  % ========== 加粗九宫格的底部线和右边线 ==========
  \draw[very thick, blue] (1,1) -- (4,1);  % 底部线
  \draw[very thick, blue] (4,1) -- (4,4);  % 右边线
  
  % ========== 9个不同的3×3方形（用不同颜色边框） ==========
  % 定义9种颜色
  \definecolor{color1}{RGB}{255,0,0}      % 红色
  \definecolor{color2}{RGB}{0,150,0}      % 绿色
  \definecolor{color3}{RGB}{0,0,255}      % 蓝色
  \definecolor{color4}{RGB}{255,165,0}    % 橙色
  \definecolor{color5}{RGB}{128,0,128}    % 紫色
  \definecolor{color6}{RGB}{0,200,200}    % 青色
  \definecolor{color7}{RGB}{255,0,255}    % 品红
  \definecolor{color8}{RGB}{165,42,42}    % 棕色
  \definecolor{color9}{RGB}{255,215,0}    % 金色
  
  % 绘制9个3×3方形的彩色边框
  % 第1行
  \draw[thick, color1] (0,2) rectangle (3,5);  % 左上角
  \draw[thick, color2] (1,2) rectangle (4,5);  % 中上
  \draw[thick, color3] (2,2) rectangle (5,5);  % 右上角
  
  % 第2行
  \draw[thick, color4] (0,1) rectangle (3,4);  % 左中
  \draw[thick, color5] (1,1) rectangle (4,4);  % 正中（九宫格）
  \draw[thick, color6] (2,1) rectangle (5,4);  % 右中
  
  % 第3行
  \draw[thick, color7] (0,0) rectangle (3,3);  % 左下角
  \draw[thick, color8] (1,0) rectangle (4,3);  % 中下
  \draw[thick, color9] (2,0) rectangle (5,3);  % 右下角
  
  % ========== 坐标轴（无刻度版） ==========
  % y轴：在九宫格右边线上 (x=4)，向上向下延长
  \draw[->, thick] (4,1) -- (4,-2.5);  % 向下延长
  \draw[->, thick] (4,1) -- (4,6.5) node[above] {$y$}; % 向上延长
  
  % x轴：在九宫格底部线上 (y=1)，向左向右延长
  \draw[->, thick] (1,1) -- (-2.5,1);  % 向左延长
  \draw[->, thick] (1,1) -- (7.5,1) node[right] {$x$}; % 向右延长
  
  % ========== 九宫格内数字示例 ==========
  \node at (1.5,1.5) {1};
  \node at (2.5,1.5) {2};
  \node at (3.5,1.5) {3};
  \node at (1.5,2.5) {4};
  \node at (2.5,2.5) {5};
  \node at (3.5,2.5) {6};
  \node at (1.5,3.5) {7};
  \node at (2.5,3.5) {8};
  \node at (3.5,3.5) {9};
  
  % ========== 原点标记 ==========
  \node[below left] at (4,1) {$O$};
  
  \end{tikzpicture}
  
  % ========== 图号和注释 ==========
  \caption{$2$-dimensional case}
  \label{fig:figure1}
\end{figure}

\begin{cor}\label{cor:C^*_r(R_varphi,sigma) has the strict comparison}  
Let $\R^d\curvearrowright \Omega$ be a free action on a compact metrizable space $\Omega$. Suppose the induced transverse groupoid $R_\varphi$ is minimal and principal. Then  $\operatorname{d.a.d}(R_\varphi)\leq 6^d-1$. Therefore, for any twist $\Sigma$ (could come from a continuous $2$-cocycle $\sigma$) on $R_\varphi$, the nuclear dimension $\dimnuc(C^*_r(R_\varphi, \Sigma))\leq 6^d$ and thus $C^*_r(R_\varphi, \Sigma)$ has the strict comparison (of projections). In addition, $C^*_r(R_\varphi, \Sigma)$ is classified by its Elliott invariant.
\end{cor}
\begin{proof}
Proposition \ref{prop: dad implies tower dim} and Theorem \ref{thm: R_varphi has finite tower dim} show that $\operatorname{d.a.d}(R_\varphi)\leq 6^d-1$ and furthermore its twisted $C^*$-algebra  satisfies 
\[\dimnuc(C^*(R_\varphi, \Sigma))\leq (\operatorname{d.a.d}(R_\varphi)+1)\cdot(\dim(X)+1)-1=6^d\]
by Theorem \ref{thm: finite dad implies finite dim nuc} and the fact that $\dim(X)=0$. Since $R_\varphi$ is minimal and principal, it is known the $C^*$-algebra $C^*_r(R_\varphi, \Sigma)$ is simple (see, e.g., \cite[Theorem 7.26]{KM} for a proof)  and therefore it has the strict comparison of positive elements by \cite{Winter} and \cite{Ror}. The classification result follows from \cite{CGSTW}, \cite{EGLN}, \cite{GLN}, \cite{TWW}. 
\end{proof}

\begin{rmk}
It is worth noting that the groupoid $R_\varphi$ is almost finite in the sense of \cite[Definition 6.2]{Matui} by \cite[Remark 6.4]{Matui}. Therefore, it follows from \cite{IWZ} or \cite{MW} that the untwisted groupoid  $C^*$-algebra $C^*_r(R_\varphi)$ is $\CZ$-stable and thus has the strict comparison of positive elements by \cite{Ror}.
\end{rmk}

\section{Frame and Riesz vectors for projective unitary representations of groupoids}\label{sec: groupoid frame}
In this section, we mainly study (average) frame and Riesz vectors for projective unitary representations of groupoids, and associated analysis and synthesis operators as a preparation for Section \ref{sec: tiling groupoid}. The following definition appears in \cite{ER} in the measurable setting, which is used to define frames on groupoids. See also \cite[Definition II.1.6]{R}. However, we apply the definition to our topological setting.
\begin{defn}
    Let $\CG$ be a groupoid and $\sigma\in Z^2(\CG, \T)$. A \textit{$\sigma$-projective unitary representation} $(\pi, \CH, \mu)$ of $\CG$ on a measurable field of $\CH=\{H_{u}\}_{u\in \GU}$ of Hilbert spaces over $(\GU, \mu)$, equipped with a quasi-invariant measure $\mu$,  is a collection of unitary maps $\pi(x): H_{s(x)}\to H_{r(x)}$ for any $x\in \CG$ such that the following hold.
    \begin{enumerate}[label=(\roman*)]
        \item For any pair of measurable sections $\xi$ and $\eta$ of $\CH$, the map $\CG\to \C$ defined by $x\mapsto \langle \pi(x)\xi(s(x)), \eta(r(x)) \rangle$ is measurable and 
        \item $\pi(x)\pi(y)=\sigma(x, y)\pi(xy)$ whenever $(x,y)\in \CG^{(2)}$.
    \end{enumerate}
\end{defn}

From now on, for simplicity, we will always assume that the measure $\mu$ in a $\sigma$-projective unitary representation $(\pi, \CH, \mu)$ of $\CG$ is a $\CG$-invariant measure so that we may get rid of modular functions. Let $\eta=\{\eta(u)\}_{u\in \GU}$ be in $\CH=\{\CH_u\}_{u\in \GU}$. Then for $u\in \GU$, we define
\[\pi(\CG^u)\eta:=\{\pi(x)\eta(s(x))\}_{x\in \CG^u}.\]

\begin{defn}\cite[Definition 3.1]{ER}\label{groupoid frame vector}
Let $\pi$ be a unitary projective representation of a groupoid $\mathcal{G}$ on a measurable field $\{\mathcal{H}_u\}_{u\in \mathcal{G}^{(0)}}$ of Hilbert spaces.
\begin{enumerate}[label=(\roman*)]
    \item We say that $\eta \in \{\mathcal{H}_u\}_{u\in \GU}$ is a \textit{frame vector for $\pi$} if there exists $0<A \le B < \infty$ such that 
\[A\|\xi(u) \|^2 \le \sum_{x\in \mathcal{G}^{u}}|\langle \xi(u) ,\pi(x)\eta(s(x))\rangle | ^2 \le B\|\xi(u)\|^2  \ \text{for all}\  u \in \mathcal{G}^{(0)} \ \text{and} \  \xi \in \{\mathcal{H}_u\}_{u\in \GU}.\]
In other words, each family $\pi(\mathcal{G}^{u})\eta$ is a frame for $\mathcal{H}_u,$ with frame bounds $0 <A \le B <\infty$ independent of $u.$ If only the upper frame bounds exist in the definition of a frame vector, we call it a \textit{Bessel vector} for $\pi.$
\item We say that $\eta\in \{\mathcal{H}_u\}_{u\in \GU}$ is a \textit{Riesz vector} for $\pi$ if there exists $0<A \le B < \infty$ such that 
\[A\|c\|^2\le \|\sum_{x \in \mathcal{G}^u}c(x)\pi(x)\eta(s(x))\|^2\le B\|c\|^2 \ \text{for all} \  u \in \mathcal{G}^{(0)} \ \text{and all} \  c=\{c(x)\}_{x \in \mathcal{G}^{u}}\in \ell^2(\mathcal{G}^u).\]
In other words, each family $\pi(\mathcal{G}^{u})\eta$ is a Riesz sequence for $\mathcal{H}_u,$ with Riesz bounds $0 <A \le B <\infty$ independent of $u.$
\end{enumerate}
\end{defn}
 Suppose $\eta$ is a Bessel vector for $\pi$. We denote by $C_{\eta,u}: \CH_u\to \ell^2(\CG^u)$ the \textit{analysis operator} for $\pi(\CG^u)\eta$, defined by 
 $C_{\eta, u}(\xi)(x)=\langle\xi,\pi(x)\eta(s(x))\rangle.$
 In addition,  denote by 
 $D_{\eta, u}=C_{\eta, u}^{*}: \ell^2(\CG^u)\to \CH_u$ the \textit{synthesis operator} for $\pi(\CG^u)\eta$,  which is given by 
 $D_{\eta, u}(c)=\sum_{x\in \CG^u}c(x)\pi(x)\eta(s(x))$. Meanwhile, the frame operator $S_{\eta,u}$ and Gramian operator $G_{\eta,u}$ are  given by $S_{\eta,u}=D_{\eta,u}C_{\eta,u},G_{\eta,u}=C_{\eta,u}D_{\eta,u}$, respectively.

\begin{defn}\label{defn: multi frame}
We say that $\vec\eta=\{\eta_i\}_{i=1}^n$ for  $\eta_i\in \{\CH_u\}_{u\in \CG^{(0)}}$ is a \textit{multi-frame vector for} $\pi$ if there exists $0<A\le B<\infty$ such that 
\[A\|\xi(u) \|^2 \le \sum_{i=1}^n\sum_{x\in \mathcal{G}^{u}}|\langle \xi(u) ,\pi(x)\eta_i(s(x))\rangle | ^2 \le B\|\xi(u)\|^2  \ \text{for all}\  u \in \mathcal{G}^{(0)} \ \text{and} \  \xi \in \{\mathcal{H}_u\}_{u\in \GU}.\] If only the upper frame bounds exists, we call $\{\eta_i\}_{i=1}^n$ a \textit{multi-Bessel vectors} for $\pi.$

We say that $\vec\eta=\{\eta_i\}_{i=1}^n$ for  $\eta_i\in \{\CH_u\}_{u\in \CG^{(0)}}$ is a \textit{multi-Riesz vector for} $\pi$ if there exists $0<A\le B<\infty$ such that 
\[A\sum_{i=1}^n\|c_i\|^2\le \|\sum_{i=1}^n\sum_{x \in \mathcal{G}^u}c_i(x)\pi(x)\eta_i(s(x))\|^2\le B\sum_{i=1}^n\|c_i\|^2 \] 
holds for any  $u \in \mathcal{G}^{(0)}$ and $c_i=\{c_i(x)\}_{x \in \mathcal{G}^{u}}\in \ell^2(\mathcal{G}^u)$.
\end{defn}
 Suppose $\vec\eta=\{\eta_i\}_{i=1}^n$ is a multi-Bessel vectors for $\pi$. The analysis operator $C_{\vec\eta,u}= : \CH_u\to \ell^2(\CG^u)^n$ is defined by 
$C_{\vec\eta, u}(\xi)(x,i)=C_{\eta_i, u}(\xi)(x)=\langle\xi,\pi(x)\eta_i(s(x))\rangle$.
 In addition,  the associated synthesis operator
 $D_{\vec\eta, u}=C_{\vec\eta, u}^{*}: \ell^2(\CG^u)^n\to \CH_u$  is given by 
 $D_{\vec\eta, u}(c_1,\dots,c_n)=\sum_{i=1}^n D_{\eta_i,u}(c_i)=\sum_{i=1}^n\sum_{x\in \CG^u}c_i(x)\pi(x)\eta_i(s(x))$. The frame operator $S_{\vec\eta,u}$ and Gramian operator $G_{\vec\eta,u}$ are  given by $S_{\vec\eta,u}=D_{\vec\eta,u}C_{\vec\eta,u},G_{\vec\eta,u}=C_{\vec\eta,u}D_{\vec\eta,u}$, respectively.

Let $(\pi,\CH,\mu)$ be a unitary projective representation of a groupoid $\mathcal{G}$ on a measurable filed $\CH=\{\mathcal{H}_u\}_{u\in \mathcal{G}^{(0)}}$.
Denote by $\nu$ the measure on $\CG$ induced by $\mu$, defined by $\nu(f)=\int_{\GU}\sum_{x\in \CG^u}f(x)\text{d}\mu$ for all Borel function $f\in B(\CG, \sigma)$. Moreover,  it follows from \cite[Proposition II.1.7]{R} that there is a $*$-representation $L_\pi$ of $C_c(\CG, \sigma)$ on  the \textit{direct integral} of $\CH$  \[\int_{\mathcal{G}^{(0)}}^{\oplus}\mathcal{H}_u \text{d}\mu(u):=\{\xi=(\xi(u))_{u\in \GU}: \int_{\GU}\|\xi(u)\|^2_{\CH_u}\text{d}u<\infty\}\]
by 
\[(L_\pi f)(\xi)(u)=\sum_{x\in \CG^u}f(x)\pi(x)\xi(s(x)).\]

This allows us to introduce the following average version of Definition \ref{groupoid frame vector}, which plays a role in Section \ref{sec: tiling groupoid}.

\begin{defn}\label{defn: average frame vector}
     Let triple $(\pi, \mathcal{H}, \mu)$ be a unitary projective representation of a groupoid $\mathcal{G}.$ 
\begin{enumerate} [label=(\roman*)]
    \item We say $\eta\in \int_{\mathcal{G}^{(0)}}^{\oplus}\mathcal{H}_u \text{d}\mu(u)$ is an \textit{average frame vector} for $(\pi,\mathcal{H}, \mu)$ if there exists $0<A \le B < \infty$ such that for all  $\xi \in \int_{\mathcal{G}^{(0)}}^{\oplus}\mathcal{H}_u \text{d}\mu(u),$ we have
\begin{equation}\label{averageframe}
    A\int_{\mathcal{G}^{(0)}}\|\xi(u) \|^2 \text{d}\mu(u)\le \int_{\mathcal{G}^{(0)}}\sum_{x\in \mathcal{G}^u}\vert \langle \xi(u) ,\pi(x)\eta(s(x))\rangle \vert ^2 \text{d}\mu(u)\le B\int_{\mathcal{G}^{(0)}}\|\xi(u) \|^2 \text{d}\mu(u).
\end{equation}
If $A=B=1,$ we call $\eta$ an average Parseval frame vector. If only the upper fame bounds exist in the definition of an average frame vector, we call it an \textit{average Bessel vector} for $(\pi,\mathcal{H}, \mu).$ 
 \item We say $\eta\in \int_{\mathcal{G}^{(0)}}^{\oplus}\mathcal{H}_u \text{d}\mu(u)$ is an \textit{average Riesz vector} for $(\pi, \mathcal{H}, \mu)$ if there exists $0<A \le B < \infty$ such that  for all $\omega \in \GU $ and all $c\in L^2(\mathcal{G},\nu),$ we have
 \begin{equation}\label{averageRiesz}
     A\int_{\mathcal{G}^{(0)}}\sum_{x\in \mathcal{G}^u}\vert c(x) \vert ^2 \text{d}\mu (u)\le 
     \int_{\mathcal{G}^{(0)}}\|\sum_{x \in \mathcal{G}^u}c(x) \pi(x)\eta(s(x))\|^2 \text{d}\mu\le B\int_{\mathcal{G}^{(0)}}\sum_{x\in \mathcal{G}^u}\vert c(x) \vert ^2 \text{d}\mu (u).
 \end{equation}
\end{enumerate}
\end{defn}

\begin{defn}
Let $\eta$ be an average Bessel vector for $(\pi, \mathcal{H}, \mu)$. The associated \textit{average analysis operator} 
\[\overline{C}_{\eta}:\int_{\mathcal{G}^{(0)}}^{\oplus}\mathcal{H}_u \text{d}\mu(u) \to L^2(\mathcal{G},\nu)=\int_{\mathcal{G}^{(0)}}^{\oplus}\ell^2(\mathcal{G}_u) \text{d}\mu(u)\]
is defined by 
\[\overline{C}_\eta(\xi)(x)= \langle \xi(r(x)),\pi(x)\eta(s(x))\rangle.  \]
 In addition, the operator \[\overline{D}_{\eta}=\overline{C}_{\eta}^*: L^2(\mathcal{G},\nu) \to \int_{\mathcal{G}^{(0)}}^{\oplus}\mathcal{H}_u \text{d}\mu(u)\]  is called the \textit{average synthesis operator}, which is given by 
\[(\overline{D}_\eta c)(u):=\sum_{x\in \mathcal{G}^u}c(x)\pi(x)\eta(s(x)).\] 
\end{defn}

Let $\eta$ be an average Bassel vector. Actually $\overline{C}_\eta=\int_{\CG^{(0)}}^{\oplus}C_{\eta,u}\text{d}\mu(u)$ and 
$\overline{D}_{\eta}=\overline{C}_\eta ^*=\int_{\CG^{(0)}}^{\oplus}D_{\eta,u}\text{d}\mu(u)$. We then additionally define
the associated \textit{average frame operator} $\overline{S}_{\eta}$ by $\overline{S}_{\eta}=\overline{D}_{\eta} \overline{C}_{\eta}=\int_{\CG^{(0)}}^{\oplus}S_{\eta,u}\text{d}\mu(u)$ and the \textit{average Gramian operator} $\overline{G}_{\eta}$ by $\overline{G}_{\eta}=\overline{C}_{\eta}\overline{D}_{\eta}=\int_{\CG^{(0)}}^{\oplus}G_{\eta,u}\text{d}\mu(u).$ 

For the future use, we also introduce average multi-frame vectors. 
\begin{defn}\label{defn: average multi}
We say $\vec{\eta}=\{\eta_i\}_{i=1}^n\subset \int_{\mathcal{G}^{(0)}}^{\oplus}\mathcal{H}_u \text{d}\mu(u)$ is an \textit{average multi-frame vectors} for $(\pi,\mathcal{H}, \mu)$ if there exists $0<A \le B < \infty$ such that for all  $\xi \in \int_{\mathcal{G}^{(0)}}^{\oplus}\mathcal{H}_u \text{d}\mu(u),$ we have
\begin{equation} \label{eq:average multi-frame}
    A\int_{\mathcal{G}^{(0)}}\|\xi(u) \|^2 \text{d}\mu(u)\le \int_{\mathcal{G}^{(0)}}\sum_{i=1}^n\sum_{x\in \mathcal{G}^u}\vert \langle \xi(u) ,\pi(x)\eta_i(s(x))\rangle \vert ^2 \text{d}\mu(u)\le B\int_{\mathcal{G}^{(0)}}\|\xi(u) \|^2 \text{d}\mu(u).
\end{equation}
If $A=B=1,$ we call $\vec\eta=\{\eta_i\}_{i=1}^n$ an \textit{average multi-Parseval frame vectors.} If only the upper fame bounds exist in the definition of an average frame vector, we call it an \textit{average multi-Bessel vector} for $(\pi,\mathcal{H}, \mu).$

We say $\vec{\eta}=\{\eta_i\}_{i=1}^n\subset \int_{\mathcal{G}^{(0)}}^{\oplus}\mathcal{H}_u \text{d}\mu(u)$ is an \textit{average multi-Riesz vectors} for $(\pi,\mathcal{H}, \mu)$ if there exists $0<A \le B < \infty$ such that for all  $c_i \in L^2(\CG,\nu),$ we have
\[A\int_{\mathcal{G}^{(0)}}\sum_{i=1}^n\sum_{x\in \mathcal{G}^u}|c_i(x) | ^2 \text{d}\mu (u)\le 
     \int_{\mathcal{G}^{(0)}}\|\sum_{i=1}^n\sum_{x \in \mathcal{G}^u}c_i(x) \pi(x)\eta(s(x))\|^2 \text{d}\mu\le B\int_{\mathcal{G}^{(0)}}\sum_{i=1}^n\sum_{x\in \mathcal{G}^u}| c_i(x) |^2 \text{d}\mu (u).\]
\end{defn}

\begin{rmk}\label{rmk: frame is average frame}
It is straightforward to see every (multi-)frame vector in Definition \ref{groupoid frame vector} and \ref{defn: multi frame} is an average (multi-)frame vector. This also applies to Bessel and Parseval (multi-)frame vectors.
\end{rmk}

 The average analysis operator for an average multi-frame $\vec\eta=\{\eta_i\}_{i=1}^n$ is  \[\overline{C}_{\vec\eta}:\int_{\mathcal{G}^{(0)}}^{\oplus}\mathcal{H}_u \text{d}\mu(u) \to \int_{\mathcal{G}^{(0)}}^{\oplus}\ell^2(\mathcal{G}^u \times \{1,\dots,n\}) \text{d}\mu(u)=L^2(\CG, \nu)^n,\]
 defined by
\[\overline{C}_{\vec\eta}(\xi)(x,i)= \langle \xi(r(x)),\pi(x)\eta_i(s(x))\rangle.  \]
One similarly may define the corresponding  average synthesis operator $\overline{D}_{\vec\eta}=\overline{C}_{\vec\eta}^*$, the average Gramian operator $\overline{G}_{\vec\eta}=\overline{C}_{\vec\eta}\overline{D}_{\vec\eta}$, and the average frame operator $\overline{S}_{\vec\eta}=\overline{D}_{\vec\eta}\overline{C}_{\vec\eta}.$ 

For average multi-frame vector and average multi-Riesz vectors, we also have the following characterization in terms of average frame operators and average Gramian operators, respectively.
\begin{prop} \label{prop:average multi-frame and average frame operator}
     For a unitary projective representation $(\pi, \mathcal{H}, \mu)$ of $\mathcal{G}$ and an  average multi-Bessel vector $ \vec \eta=\{\eta_i\}_{i=1}^n \subset \int_{\mathcal{G}^{(0)}}^{\oplus}\mathcal{H}_u\text{d}\mu(u)  .$ Then the following hold:
     \begin{enumerate}[label=(\roman*)]
         \item $\vec\eta$ is an average multi-frame vector if and only if $\overline{S}_{\vec \eta}$ is an invertible operator in $\mathcal{B}(\int_{\mathcal{G}^{(0)}}^{\oplus}\mathcal{H}_u\text{d}\mu(u)). $
         \item $\vec \eta$ is an average multi-Riesz vector if and only if $\overline{G}_{\vec\eta}$ is an invertible operator in $\mathcal{B}(L^2(\mathcal{G},\nu)^n). $
     \end{enumerate}
\end{prop}
\begin{proof}
(\romannumeral 1) For any $\xi \in \int_{\mathcal{G}^{(0)}}^{\oplus}\mathcal{H}_u\text{d}\mu(u) ,$ we have
    \begin{align}
        \langle \overline{S}_{\vec\eta} \xi,\xi \rangle&=
        \langle \overline{C}_{\vec \eta} \xi,\overline{C}_{\vec \eta}\xi \rangle=\|\overline{C}_{\vec\eta}\xi\|^2_{L^2(\CG,  \nu)^n}\notag
        =\int_{\mathcal{G}^{(0)}}\sum_{i=1}^n\sum_{x\in \mathcal{G}^u}\vert \langle \xi(u) ,\pi(x)\eta_i(s(x))\rangle \vert ^2 \text{d}\mu.\notag
    \end{align}
    Therefore, if $\eta$ is an average multi-frame vector satisfying (\ref{eq:average multi-frame}), then 
    \begin{equation} \label{ieq:average multi-frame}
        0<A I \le \overline{S}_{\vec\eta} \le BI,
    \end{equation}
   where $I$ is the identity operator on $\int_{\GU}^\oplus \CH_u\text{d}\mu(u)$. It thus follows that $\overline{S}_{\vec\eta}$ is invertible in  $\mathcal{B}(\int_{\mathcal{G}^{(0)}}^{\oplus}\mathcal{H}_u\text{d}\mu(u)). $ 

    For the converse, suppose $\overline{S}_{\vec\eta}$ is invertible in  $\mathcal{B}(\int_{\mathcal{G}^{(0)}}^{\oplus}\mathcal{H}_u\text{d}\mu(u))$. Using the spectral theory of positive invertible operator,
     one has $1/\|\overline{S}_{\vec\eta}^{-1}\|=\inf\sigma(\overline{S}_{\vec\eta})>0$ and
    \[\frac{1}{\|\overline{S}_{\vec \eta}^{-1}\|}\cdot\int_{\mathcal{G}^{(0)}}\|\xi(u) \|^2 \text{d}\mu\leq\langle\overline{S}_{\vec\eta}\xi,\xi\rangle=\int_{\mathcal{G}^{(0)}}\sum_{i=1}^n\sum_{x\in \mathcal{G}^u}\vert \langle \xi(u) ,\pi(x)\eta_i(s(x))\rangle \vert ^2 \text{d}\mu \le \|\overline{S}_{\vec\eta}\|\cdot\int_{\mathcal{G}^{(0)}}\|\xi(u) \|^2 \text{d}\mu \]holds for all $\xi \in\int_{\mathcal{G}^{(0)}}^{\oplus}\mathcal{H}_u\text{d}\mu(u).$

(\romannumeral 2) The proof is similar to (\romannumeral 1) and we omit it.
\end{proof}
Parallel to \cite[Corollary 3.3]{ER}, we have the following result. 
\begin{prop} \label{Parseval}
 Let $(\pi, \mathcal{H}, \mu)$ be a unitary projective representation of $\mathcal{G}.$ Then the following hold:
\begin{enumerate} [label=(\roman*)]
    \item If $(\pi, \mathcal{H}, \mu)$ admits an average multi-frame vector, then it admits an average multi-Parseval frame vector.
    \item If $(\pi, \mathcal{H}, \mu)$ admits an average multi-Riesz vector, then it admits an average  multi-orthonormal vector.
\end{enumerate}
\end{prop}
\begin{proof}
(\romannumeral 1) Let $\vec\eta$ be an average multi-frame vector and $\overline{S}_{\vec\eta}$ the frame operator associated to $
\vec\eta.$ For $\xi \in \int_{\mathcal{G}^{(0)}}^{\oplus}\mathcal{H}_u\text{d}\mu(u),$ as $\overline{S}_{\vec \eta}$ is invertible by Proposition~\ref{prop:average multi-frame and average frame operator} and $\overline{S}_{\vec\eta}=\int_{\CG^{(0)}}^\oplus S_{\vec\eta,u}\text{d}\mu(u)$, we have $S_{\vec\eta,u} $ is invertible, a.e. $u\in \CG^{(0)}$ and $\overline{S}_{\vec\eta}^{-1}=\int_{\CG^{(0)}}^\oplus S_{\vec\eta,u}^{-1}\text{d}\mu(u)$(See \cite[Section 14.1]{KR}). We compute
\begin{align*}
   \int_{\CG^{(0)}}^\oplus \sum_{i=1}^n \sum_{x\in \CG^u} |\langle \xi(u),\pi(x)(\overline{S}_{\vec \eta}^{-1/2}\eta_i)(s(x)) \rangle |^2\text{d}\mu(u)&=\int_{\CG^{(0) }}^\oplus \sum_{i=1}^n\sum_{x\in \CG^u} |\langle \xi(u),\pi(x)S_{\vec \eta,s(x)}^{-1/2}\eta_i(s(x)) \rangle |^2\text{d}\mu(u)\\
&=\int_{\CG^{(0)}}^\oplus \| \xi(u)\|^2\text{d}\mu(u) =\|\xi\|^2,
\end{align*}
where the second equality follows from \cite[Remark 14.1.7]{KR} and the proof of \cite[Corollary 3.3]{ER}.
The computation shows that $\{\overline{S}_{\vec\eta}^{-1/2}\eta_i\}_{i=1}^n$ is the average multi-Parseval  frame vector.

(\romannumeral 2) The proof is similar to (\romannumeral 1) and we omit it.
\end{proof}

%We note the following remark.

%\begin{rmk}\label{rmk: multi version}
%Using the same proofs, Propositions \ref{afo} and \ref{Parseval} also hold for average multi-Bessel vectors $\vec\eta=\{\eta_i\}_{i=1}^n$ with respect to operators $\overline{D}_{\vec\eta}$, $\overline{S}_{\vec\eta}$, and $\overline{G}_{\vec\eta}$.
%\end{rmk}

\begin{defn}\label{sigma-positive}
    Let $\sigma$ to be a continuous 2-cocycle in $Z^2(\mathcal{G},\mathbb{T}).$ A function $f:\mathcal{G}\to \mathbb{C}$ is said to be $\sigma$-positive definite if
    $$\sum_{i,j=1}^n c_i \overline{c_j}\sigma(\gamma_j \gamma_i^{-1} ,\gamma_i)f( \gamma_j \gamma_i^{-1})\ge 0 \quad  \text{for all} \ u \in \mathcal{G}^{(0)},\gamma_i \in \mathcal{G}_{u},c_i \in \mathbb{C},i=1,\dots, n.$$
\end{defn}
\begin{prop} \label{positive}
Assume $f\in \ell_1(\mathcal{G}, \sigma).$ Then $f$ is a $\sigma$-positive definite if and only if
$(\bigoplus_{u\in\mathcal{G}^{(0)}} \pi_u)(f)$ is positive as an element in $C_r^*(\mathcal{G},\sigma).$ 
\end{prop}
\begin{proof}
Recall $C_r^*(\CG)=\overline{\ell_1(\CG, \sigma)}^{\|\cdot\|_r}$ and
    \[\langle(\bigoplus _{u\in\mathcal{G}^{(0)}} \pi_u)(f)(\xi_u)_u,(\xi_u)_u \rangle=\sum_{u\in \mathcal{G}^{(0)}}\langle\pi_u(f)\xi_u,\xi_u \rangle,\] for $(\xi_u)_u \in \bigoplus_{u\in\mathcal{G}^{(0)}}\ell^2(\mathcal{G}_u).$ Thus, $(\bigoplus _{u\in\mathcal{G}^{(0)}} \pi_u)(f)$ is positive if and only if $\pi_u(f)$ is positive for all $u \in \mathcal{G}^{(0)}.$ 

    For each $u \in \mathcal{G}^{(0)},$ let $F$ be a finite subset of $ \mathcal{G}_u$. Denote by $P_{F}$ the orthogonal projection of $\ell^2(\mathcal{G}_u)$ onto $\ell^2(F)$. Then  note $\pi_u(f)$ is positive if and only if $P_{F}\pi_{u}(f)P_{F}$ is positive for all finite $F\subset \CG_u.$ Write $F=\{\gamma_1,\dots, \gamma_n\}$ explicitly.
  Then for $\xi_u\in \ell^2(\mathcal{G}_u)$, one has $P_{F}\xi_u=\sum_{i=1}^nc_i\delta_{\gamma_i}.$ This implies
    \begin{align}  \langle\pi_u(f)\xi_u,\xi_u\rangle&=\sum_{i,j=1}^nc_i \overline{c_j}\langle\pi_u(f)\delta_{\gamma_i},\delta_{\gamma_j}\rangle=\sum_{i,j=1}^nc_i \overline{c_j}\langle\sum_{\alpha\in \mathcal{G}_{r(\gamma_i)}}\sigma(\alpha,\gamma_i)f(\alpha)\delta_{\alpha\gamma_i},\delta_{\gamma_j}\rangle \notag\\
        &=\sum_{i,j=1}^nc_i \overline{c_j}\sigma(\gamma_j\gamma_i^{-1},\gamma_i)f(\gamma_j \gamma_i^{-1}). \notag
    \end{align}
 Therefore,  $\pi_u(f)$ is positive for all $u \in \mathcal{G}^{(0)}$ is equivalent to $f$ is $\sigma$-positive definite.
\end{proof}
%In the rest of this section, we give a method to construct a left Hilbert $C_r^*(\CG,\sigma)$-module $\CE$ from $\sigma$-projective representation $(\pi,\mathcal{H},\mu)$ of groupoid $\CG$.
%We can first find some dense subspace $\mathcal{H}_0$ of $\mathcal{H}$ and 
 %define an action of $C_c(\CG,\sigma)$ on $\mathcal{H}_0$ and $A_0$-valued function on $\mathcal{H}_0 \times \mathcal{H}_0$ which can be extended to $C^*_r(\CG,\sigma)$-valued inner product.

\section{Full Gabor frames and Full Riesz sequences for Delone sets}\label{sec: tiling groupoid}
In this section, we establish the ``if'' part of Theorem \ref{thm: main 1}, i.e., the converse direction of the Balian-Low theorem for tiling groupoids $R_\Lambda$. We will always assume $\Lambda$ is a FLC, aperiodic, and repetitive Delone set in $\R^{2d}$. Then recall that its tiling groupoid $R_\Lambda$ is locally compact Hausdorff $\acute{e}$tale,  minimal principal, ample with the compact metrizable unit space $\Omega_0(\Lambda)$ by Remark \ref{rmk: tiling groupoid well defined}. 

We define the symplectic $2$-cocycle $\sigma$ on $\R^{2d}$ by $\sigma(z_1,z_2)=e^{-2\pi i x_1 \omega_2}$ for $z=(x_1,\omega_1),z_2=(x_2,\omega_2 )\in \R^{2d}$. As usual, we define a continuous $2$-cocycle  $ \sigma_{\Lambda}$ in $Z^2(R_{\Lambda},\mathbb{T})$ associated to $\sigma$  by
 \[\sigma_{\Lambda}((T,T-z_1),(T',T'-z_2))=\sigma(z_1,z_2).\]

 Let $(\{\CH_T\}_{T\in \Omega_0(\Lambda)},\mu,\pi_{\Lambda})$  be a 
 $\sigma_{\Lambda}$-projective unitary representation  of $R_{\Lambda},$ where $\mu$ is an invariant measure on $\Omega_0(\Lambda)$, all $\CH_T=L^2(\mathbb{R}^d)$, and $\pi_{\Lambda}=\{\pi_{(T,T-z)}\}_{(T,T-z)\in R_{\Lambda}}$ is given by 
 \begin{equation}\label{projetive representation of tiling groupoid case}
     \pi_{(T,T-z)}f=\pi(z)f, f \in L^2(\mathbb{R}^{d}).
 \end{equation}
Here $\pi(z)$ is the time-frequency shift on $L^2(\mathbb{R}^{d})$. Recall that $C_c(R_{\Lambda},\sigma_{\Lambda})$ is a $*$-algebra with the convolution product 
 \[(f*_{\sigma_\Lambda} g)(T,T-z)=\sum_{\omega \in T}f(T,T-\omega)g(T-\omega,T-z)\sigma(\omega,z-\omega).\]
 and the involution 
 \[f^*(T,T-z)=\overline{f(T-z,T)\sigma(z,-z)}.\]
 The completion of $C_c(R_{\Lambda},\sigma_{\Lambda})$ under the I-norm, i.e., $\ell_{1}(R_{\Lambda}, \sigma_\Lambda):=\overline{C_c(R_{\Lambda},\sigma_{\Lambda})}^{\| \cdot \|_I}$,  is a dense $*$-subalgebra in $C_r^*(R_{\Lambda},\sigma_{\Lambda})$. 
 
 This $\sigma_{\Lambda}$-projective unitary representation of $R_{\Lambda}$ above induces a $*$-representation  of $C_c(R_{\Lambda},\sigma)$ on $\CH=L^2(\Omega_0(\Lambda),L^2(\mathbb{R}^d),\mu)$ by \cite[Proposition II.1.7]{R}. We still denote this representation by $\pi_{\Lambda}$, which is defined by 
 \[\pi_{\Lambda}(f)\Psi(T)=\sum_{z \in T } f(T,T-z)\pi(z)\Psi(T-z),f \in C_c(R_{\Lambda},\sigma),\Psi \in \mathcal{H}, T\in \Omega_0(\Lambda).\] 

We first record a useful lemma that will be applied in  Proposition \ref{prop: pre module} and \ref{equ}.

\begin{lem} \label{esitimate of STFT}
    Let $ \Lambda$ be the Delone set in $\R^{2d}.$ If  $g\in M^1(\R^d) $ and $f\in M^p(\R^d),$ then for every $\epsilon >0,$ there exists $r>0$ such that for all $T\in \Omega(\Lambda)$ we have 
    \begin{equation}
        \sum_{z\in T\setminus B(0,r) }|V_gf(z)|^p  < \epsilon.
    \end{equation}
\end{lem}
\begin{proof}
By Remark~\ref{Delone set is relative separated} we know $\Lambda$ is relative separated. Moreover $\operatorname{rel(T)}=\operatorname{rel(\Lambda)}$ for all $T\in \Omega(\Lambda).$ According to \cite[Theorem~12.2.1]{G01} if $g\in M^1(\R^d)$ and $f\in M^p(\R^d),$ then $V_gf\in W(L^{\infty},l^p)(\R^{2d}).$ Hence for every $\epsilon ,$ there exists $r>0$ such that $\sum_{k\in \Z^{2d}\setminus B(0,r-1)}\|V_gf\cdot T_k \chi_{[0,1]^{2d}}\|^p_{\infty} < \epsilon/\operatorname{rel(\Lambda)}.$ For $T \in \Omega(\Lambda)$ and $z\in T\cap [k,k+1]^{2d} ,$ one has $|V_gf(z)|\le \|V_gf\cdot T_k \chi_{[0,1]^{2d}}\|_{\infty}$ and $|T\cap [k,k+1]^{2d}|\le \operatorname{rel(\Lambda)}$.  Thus, for all $T \in \Omega(\Lambda)$  we have 
\[\sum_{z\in T\setminus B(0,r) }|V_gf(z)|^p =\sum_{k\in \Z^{2d}}\sum_{z\in    T\cap[k,k+1]^{2d} \setminus B(0,r)}|V_gf(z)|^p \le \operatorname{rel(\Lambda)}\sum_{k\in \Z^{2d}\setminus B(0,r)}\|V_gf\cdot T_k \chi_{[0,1]^{2d}}\|^p_{\infty} < \epsilon. \]
\end{proof}

\subsection{Hilbert $C^*$-module $\CE$ of $C_r^*(R_\Lambda,\sigma_{\Lambda})$}\label{subsec1}
In this subsection, we aims to construct a left Hilbert $C_r^*(R_\Lambda,\sigma_{\Lambda})$-module associated to the $\sigma_{\Lambda}$-projective representation $\pi_{\Lambda}$. To this end,
we first find a dense subspace $\mathcal{H}_0$ of $\mathcal{H}$ and then
    define an action of $C_c(R_{\Lambda},\sigma_{\Lambda})$ on $\mathcal{H}_0$ and $\ell_1(R_{\Lambda}, \sigma_\Lambda)$-valued function on $\mathcal{H}_0 \times \mathcal{H}_0$, which can be extended to $C^*_r(R_{\Lambda},\sigma_{\Lambda})$-valued inner product. Denote by $\CS(\R^d)$ the  Schwartz space.
    
    %we take $$A_0=L^{1,\sigma_{\Lambda}}_I(R_{\Lambda}),A_{00}=C_c(R_{\Lambda},\sigma_{\Lambda}),$$ and
    %$\mathcal{H}_0=\{g\in C( \Omega _{\text{trans}}, M^1(\mathbb{R}^d)) \ |  \sup_{T\in \Omega_{\text{trans}}}\sum_{z\in T}|\langle \Psi(T),\pi(z)g(T-z)\rangle| < \infty \  \ \text{for all}\ \Psi\in C( \Omega _{\text{trans}}, M^1(\mathbb{R}^d))\}.$
    %$$\mathcal{H}_0=C(\Omega_\punc)\odot M^1(\mathbb{R}^d)).$$
     %As we mentioned above, the set $A_{00}$ and $ A_0$ is dense in $C_r^*(R_{\Lambda},\sigma_{\Lambda}).$ 
     %Meanwhile,  we have that 
      %\begin{align}
       % \sup_{T \in \Omega_{\text{trans}}}\sum_{z\in T}|\langle \Psi(T),\pi(z)gh(T-z) \rangle|&\le \sup_{T \in \Omega_{\text{trans}}}\sum_{z\in T}|\langle \Psi(T),\pi(z)g \rangle|\notag\\
       % &=\sup_{T \in \Omega_{\text{trans}}}\sum_{z\in T}|V_g\Psi(T)(z)|\notag\\
        %&\le \text{rel}(\Lambda)\sup_{T \in \Omega_{\text{trans}}}\| V_g\Psi(T)\|_{W(L^{\infty},L^1)}\notag\\
        %&\le \text{rel}(\Lambda) \|g\|_{M^1(\mathbb{R}^d)}\sup_{T \in %\Omega_{\text{trans}}}\| \Psi(T)\|_{M^1(\mathbb{R}^d)}<\infty,\notag
    %\end{align}
     %therefore, $gh$ is in $\mathcal{H}_0.$ 

    \begin{prop}\label{prop: pre module}
         Let $R_{\Lambda}$ be a tiling groupoid of a FLC, repetitive and aperiodic Delone set $\Lambda \subset \mathbb{R}^{2d}$ with the $2$-cocycle $\sigma_{\Lambda}$ defined above. Let $(\{\mathcal{H}_T\}_{T\in \Omega_0(\Lambda)},\mu,\pi_{\Lambda})$ be the $\sigma_{\Lambda}$-projective unitary representation of $R_{\Lambda}$ above. Denote $\mathcal{H}_0=C(\Omega_0(\Lambda))\odot \CS(\mathbb{R}^d).$ Then  $\mathcal{H}_0$ is a dense subspace of $\mathcal{H}=L^2(\Omega_0(\Lambda),L^2(\mathbb{R}^d),\mu)$ satisfying the following.
         \begin{enumerate}[label=(\roman*)]
         \item\label{pre module 1} For every $\Psi,\Phi \in C(\Omega_0(\Lambda))\odot \CS(\mathbb{R}^d),$ the map
         \begin{equation}\label{inner}
             R_{\Lambda}\ni(T,T-z) \mapsto {_\bullet \langle} \Psi, \Phi \rangle(T,T-z):= \langle \Psi(T),\pi(z)\Phi(T-z) \rangle_{L^2(\mathbb{R}^d)} \in \mathbb{C}
         \end{equation}
         defines an element in $\ell_1(R_{\Lambda},\sigma_\Lambda).$
        \item\label{pre module 2} For $f\in C_c(R_{\Lambda},\sigma_{\Lambda})$ and $ \Psi\in C(\Omega_0(\Lambda))\odot \CS(\mathbb{R}^d),$ the element $f\cdot \Psi$ defined by
        \begin{equation}\label{action}
            f \cdot \Psi (T):= \sum_{z \in T } f(T,T-z)\pi(z)\Psi(T-z)
        \end{equation}
         is in $C(\Omega_0(\Lambda) )\odot \CS(\mathbb{R}^d).$
         \end{enumerate}
    \end{prop}
\begin{proof}
Let $g\in \CS(\mathbb{R}^d)$. For any Borel function $f: \Omega_0(\Lambda)\to \C$, we define a $\CS(\mathbb{R}^d)$-valued function  $\Phi_{f, g}:\Omega_0(\Lambda) \to \CS(\mathbb{R}^d)$ by $\Phi_{f, g}(T)=f(T)\cdot g$.

For every Borel set $A$ in $\Omega_0(\Lambda)$ and $\epsilon >0,$ one can choose open set $O$ and compact set $K$ such that $K\subseteq A\subseteq O$ and $\mu(O \setminus K)\le \epsilon.$ By Urysohn's lemma, there exists a continuous function $h: \Omega_0(\Lambda) \mapsto [0,1]$ satisfying that $h|_K\equiv 1$ and $\supp(h)\subseteq  O$.  Note that the function $\Phi_{h, g}$ is in $C( \Omega_0(\Lambda), \mathcal{S}(\mathbb{R}^d))$. Then by the definition, one has
\[\|\Phi_{1_A, g}-\Phi_{h,g}\|_{L^2(\Omega_0(\Lambda),L^2(\mathbb{R}^d),\mu)}\le \mu (O\setminus K)\|g\|_{L^2(\mathbb{R}^d)}\le \epsilon\|g\|_{L^2(\mathbb{R}^d)}\] Moreover, using the fact that $\CS(\mathbb{R}^d)$ is dense in $L^2(\mathbb{R}^d),$  we obtain that $\mathcal{H}_0$ is a dense subspace of  $\mathcal{H}.$ 

We now establish \ref{pre module 1}. Let $f,g \in \CS(\mathbb{R}^d)$ and $h_1, h_2\in C(\Omega_0(\Lambda))$ and $\Psi,\Phi \in C(\Omega_0(\Lambda))\odot \CS(\mathbb{R}^d)$ be generating functions defined by $\Psi(T)=h_1(T)\cdot f$  and $\Phi(T)=h_2(T)\cdot g$, respectively. It is direct to see $_\bullet \langle \Psi, \Phi \rangle$ as defined in the statement is a continuous function on $R_{\Lambda}$ as ${_\bullet \langle} \Psi, \Phi \rangle(T,T-z)=h_1(T)\overline{h_2(T-z)}\langle f,\pi(z)g\rangle$.
     Since $f,g \in \CS(\mathbb{R}^d)$, Lemma~\ref{esitimate of STFT} shows that for every $\epsilon>0$, there exists $r>0$ such that for all $T\in\Omega_0(\Lambda)$ we have
     \[\max\{\sum\limits_{z\in T\setminus B(0,r)}|\langle f,\pi(z)g \rangle |,\sum\limits_{z\in T\setminus B(0,r)}|\langle g,\pi(z)f \rangle |\}<\frac{\epsilon}{\|h_1\|_{\infty}\|h_2\|_{\infty}}.\]
     Therefore we have that 
     \[\max\{\sup_{T\in \Omega_0(\Lambda)}\sum_{z\in T \setminus  B(0,r)}\vert _\bullet \langle \Psi, \Phi \rangle(T,T-z) \vert,\sup_{T\in \Omega_0(\Lambda)}\sum_{z\in T \setminus  B(0,r)}\vert _\bullet \langle \Psi, \Phi \rangle(T-z,T) \vert\}<\epsilon.\]
     For the $r$ and each $T\in \Omega_0(\Lambda)$, define $F_{T}:=\{ p\in B(0,r) : T-p \in \Omega_0(\Lambda)\}.$ Since $\Omega_0(\Lambda)$ is a flat Cantor transversal of $\Omega(\Lambda)$, Definition \ref{defn: flat cantor transversal} \ref{cantor transversal etale} implies that  there exists a clopen neighborhood $N_T \ni T$ such that $F_T=F_S$ for all $S\in N_T.$ Shrinking each $N_T$ if necessary, one may assume that $B_{p, N_T}=\{(S, S-p): S\in N_T\}$ is a compact open bisection for each $p\in F_T$. It then follows from compactness that there exists finite $\{N_i\}_{i=1}^n\subset \{N_T\}_{T\in \Omega_0(\Lambda)}$ forming an clopen  cover of $\Omega_0(\Lambda).$ Choosing $T_i\in N_i$ for each $i,$ and define $K:=\bigcup_{i=1}^n \bigcup_{p\in F_{T_i}}B_{p,N_i}$.  Then $K$ is a compact open subset in $R_{\Lambda}$ and thus by the choice of $r$, one has 
     \begin{equation}\label{equi l1}
         \max\{\sup_{T\in \Omega_0(\Lambda)}\sum_{ (T,T-z)\in K^c}\vert _\bullet \langle \Psi, \Phi \rangle(T,T-z) \vert,\sup_{T\in \Omega_0(\Lambda)}\sum_{ (T-z,T)\in K^c}\vert _\bullet \langle \Psi, \Phi \rangle(T-z,T) \vert \}<\epsilon.
     \end{equation}
     Now define the restriction function $H=_\bullet\langle \Psi, \Phi \rangle \cdot 1_K$, which belongs to $C_c(R_{\Lambda},\sigma)$ because $K$ is clopen. In addition, by (\ref{equi l1}), one has $\|_\bullet \langle \Psi, \Phi \rangle-H\|_I < \epsilon.$ This implies $_\bullet \langle \Psi, \Phi \rangle \in \ell^{1}(R_{\Lambda},\sigma_\Lambda).$ Finally, since such functions $\Phi$ and $\Psi$ linearly generate $C(\Omega_0(\Lambda))\odot \CS(\mathbb{R}^d)$, one obtains that \ref{pre module 1} holds.
      
    For \ref{pre module 2}, first let $f\in C_c(R_{\Lambda},\sigma_{\Lambda})$ supported on compact bisection $B_{p,U}=\{(T,T-p):T \in U\},$ where $U$ is a clopen set in $\Omega_0(\Lambda)$ and $\Psi \in C(\Omega_0(\Lambda))\odot \CS(\mathbb{R}^d) $ a generating function, i.e., $\Psi (T)=h(T)\cdot g$ for some $h\in C(\Omega_0(\Lambda))$ and $g\in \CS(\mathbb{R}^d).$  Then for each $T \in U,$ because $f$ is supported on $B_{p, U}$, one has
    \[f\cdot \Psi(T)=\sum_{z\in T}f(T, T-z)\pi(z)\Psi(T-z)=\sum_{z\in T}f(T, T-z)\pi(z)h(T-z)g=f(T, T-p)\pi(p)h(T-p)g.\]
    Therefore, $f\cdot \Psi(T)=h'(T)\pi (p)g$, where $h'(T):=f(T,T-p)h(T-p)$ defines an function in $ C(\Omega_0(\Lambda))$. On the other hand, for $T\notin U$, one obtains $f\cdot \Psi(T)=0$. Then combining these two cases and using facts that $\CS(\R^d)$ is invariant under $\pi$ and $U$ is a clopen set, one has $f\cdot\Psi \in C(\Omega_0(\Lambda)) \odot \CS(\mathbb{R}^d).$  Finally, note that $C_c(R_\Lambda, \sigma_\Lambda)$ is linearly generated by functions supported on clopen bisections by \cite[Lemma 3.1.3]{Sims} and $C(\Omega_0(\Lambda))\odot \CS(\mathbb{R}^d)$ is also linearly generated by functions $\Phi$ of the above form $\Phi(T)=h(T)g$, the claim \ref{pre module 2} holds.    
    \end{proof}

%    \textcolor{blue}{XM: Prop 6.1 above is to construct admissible pair $(A_0, H_0)$ in the sense of \cite[Definition 4.1]{BEV}}

Then it follows from Proposition \ref{prop: pre module} that one may define an inner product $C_c(R_\Lambda, \sigma_\Lambda)$-module. This will then provide a desired $C_r^*(R_\Lambda, \sigma_\Lambda)$-module as demonstrated in the following.

\begin{prop} \label{module}
   There exist a Hilbert $C_r^*(R_{\Lambda},\sigma_{\Lambda})$-module $\mathcal{E}$ such that 
   $\mathcal{E}$ is the completion of $\mathcal{H}_0$ and the action is extended by (\ref{action}) and $C_r^*(R_{\Lambda},\sigma_{\Lambda})$-valued inner product  is extended by (\ref{inner}). 
\end{prop}
\begin{proof}
It is obvious that ${_\bullet \langle} \Psi, \Phi\rangle$ is $\C$-linear in the first variable. 
    Let $f \in C_c(R_{\Lambda},\sigma_{\Lambda})$, and $\Psi,\Phi \in \mathcal{H}_0$. 
Suppose that $f$ is supported on a clopen bisection $B_{p,U}=\{(T,T-p):T\in U\},$ where $U$ is a compact open subset of $\Omega_0(\Lambda).$ For $T\in U,$ from (\ref{action}) and (\ref{inner}), we have
\begin{align}
    {_\bullet \langle} f\cdot \Psi, \Phi \rangle(T,T-z)&=\langle f \cdot \Psi(T),\pi(z)\Phi(T-z)\rangle=\langle f(T,T-p)\pi(p)\Psi(T-p),\pi(z)\Phi(T-z)\rangle\notag\\
    &=\sigma(p,z-p)f(T,T-p)\langle \Psi(T-p),\pi(z-p)\Phi(T-z))\rangle\notag\\
    &=\sigma_{\Lambda}((T,T-p),(T-p,T-z))f(T,T-p){_\bullet \langle}  \Psi, \Phi \rangle(T-p,T-z)\notag\\
    &=(f*_{\sigma_{\Lambda}}{_\bullet \langle} \Psi, \Phi \rangle)(T,T-z).\notag
\end{align}
For $T \notin U$, both of $ {_\bullet \langle} f\cdot \Psi, \Phi \rangle(T,T-z)$ and $(f*_{\sigma_{\Lambda}}{_\bullet \langle} \Psi, \Phi \rangle)(T,T-z)$ is zero.
Therefore, \begin{equation} \label{ac}
    {_\bullet \langle} f\cdot \Psi, \Phi \rangle=f*_{\sigma_{\Lambda}}{_\bullet \langle} \Psi, \Phi \rangle .
\end{equation}

Moreover, 
\begin{align}
    {_\bullet \langle} \Psi, \Phi \rangle ^*(T,T-z)&=\overline{\sigma(z,-z){_\bullet \langle} \Psi, \Phi \rangle (T-z,T)}=\overline{\sigma(z,-z)} \langle\pi(-z)\Phi(T),\Psi(T-z)\rangle\notag\\
    &=\overline{\sigma(z,-z)} \langle\Phi(T),\overline{\sigma(z,-z)} \pi(z)\Psi(T-z)\rangle\notag\\
    &=\langle \Phi(T), \pi(z)\Psi(T-z)\rangle={_\bullet \langle} \Phi, \Psi \rangle (T,T-z).\notag
\end{align}
Therefore,\begin{equation} \label{conjugate}
    {_\bullet \langle} \Psi, \Phi \rangle ^*={_\bullet \langle} \Phi, \Psi \rangle .
\end{equation}

Next we prove that ${_\bullet \langle} \Psi, \Psi\rangle\geq 0$ in $C_r^*(R_\Lambda, \sigma_\Lambda)$, i.e., $ (\bigoplus_{T\in \Omega_0(\Lambda)} \pi_T)( {_\bullet \langle} \Psi, \Psi\rangle)  $ is positive.  According to Proposition \ref{positive}, it suffices to show that $ {_\bullet \langle} \Psi, \Psi\rangle$ is a $\sigma_{\Lambda}$-positive definite function. For all   $T\in \Omega_0(\Lambda), z_i \in T,c_i \in \mathbb{C},i=1,\dots, n ,$ we have 
\begin{align}
    & \sum_{i,j=1}^nc_i \overline{c_j}{_\bullet \langle} \Psi, \Psi\rangle((T-z_j,T)\cdot (T-z_i,T)^{-1})\sigma_{\Lambda}((T-z_j, T-z_i),(T-z_i,T))\notag\\
    &=\sum_{i,j=1}^nc_i \overline{c_j}{_\bullet \langle} \Psi, \Psi\rangle(T-z_j,T-z_i)\sigma(z_i-z_j,-z_i)\notag\\
    &=\sum_{i,j=1}^nc_i \overline{c_j}\langle\Psi(T-z_j),\pi(z_i-z_j)\Psi(T-z_i)\rangle\sigma(z_i-z_j,-z_i)\notag\\
    &=\sum_{i,j=1}^nc_i \overline{c_j}\langle\Psi(T-z_j),\pi(-z_j)\pi(z_i)\Psi(T-z_i)\rangle\sigma(z_i-z_j,-z_i)\overline{\sigma(-z_j,z_i)}\notag\\
    &=\sum_{i,j=1}^nc_i \overline{c_j}\langle \pi(-z_j)^*\Psi(T-z_j),\pi(-z_i)^*\Psi(T-z_i)\rangle\notag\\
    &=\|\sum_{i=1}^n\overline{c_i} \pi(-z_i)^*\Psi(T-z_i)\|^2 \ge 0,\notag
\end{align}
thus $ (\bigoplus_{T\in \Omega_0(\Lambda)} \pi_T)( {_\bullet \langle} \Psi, \Psi\rangle)  $ is positive and therefore ${_\bullet \langle} \Psi, \Psi\rangle\geq 0$.  Finally, note that if ${_\bullet \langle} \Psi, \Psi\rangle=0,$ then $\|\Psi(T)\|^2={_\bullet \langle} \Psi, \Psi\rangle(T,T)=0$ for all $T\in \Omega_0(\Lambda), $ so $\Psi=0.$

So far we have shown that $\CH_0$ is a pre-inner product $C_c(R_\Lambda, \sigma_\Lambda)$-module. Then Proposition \ref{prop: completion} provides   a Hilbert $C_r^*(R_{\Lambda},\sigma_{\Lambda})$-module $\mathcal{E},$ where $\mathcal{E}$ be the completion of $\mathcal{H}_0$ and the action is extended by (\ref{action}) and $C_r^*(R_{\Lambda},\sigma_{\Lambda})$-valued inner product is extended by (\ref{inner}). 
\end{proof}

\subsection{Finite generation of the Hilbert module $\CE$}
Let $\tau_{\mu}$ be the faithful tracial state on $C_r^*(R_{\Lambda},\sigma_{\Lambda})$ induced by an invariant probability Radon measure $\mu$ on $\Omega_0(\Lambda)$ by $\tau_{\mu}(a)=\int_{\Omega_0(\Lambda)}E(a)\text{d}\mu$ for $a\in C_r^*(R_\Lambda,\sigma_{\Lambda})$ where $E:C_r^*(R_{\Lambda},\sigma_{\Lambda}) \to C(\Omega_0(\Lambda))$ is the faithful canonical conditional expectation. 

We denote $\mathcal{H}_{\mathcal{E}}^{\tau_{\mu}}$ the localization space of $\mathcal{E}$ (see Section \ref{sec: Hilbert mod} for the definition). For $\Psi,\Phi \in \mathcal{H}_0$, by Proposition \ref{prop: conditional expectation},  we have 
\[\tau_{\mu}(_{\bullet}\langle\ \Psi,\Phi \rangle)=\int_{\mathcal{G}^{(0)}} {_{\bullet}}\langle\Psi ,\Phi \rangle(T,T)\text{d}\mu(T)=\int_{\mathcal{G}^{(0)}}\langle\Psi(T),\Phi(T) \rangle_{L^2(\mathbb{R}^d)}\text{d}\mu(T)=\langle\Psi,\Phi \rangle_{\mathcal{H}}.\]
 This shows that $\mathcal{H}_{\mathcal{E}}^{\tau_{\mu}}\cong \mathcal{H}= L^2(\Omega_0(\Lambda),L^2(\mathbb{R}^d),\mu)$. Similarly, we have $\mathcal{H}_{C_r^*(R_{\Lambda},\sigma_{\Lambda})}^{\tau_{\mu}} \cong L^2(R_{\Lambda},\nu)$ by regarding $C_r^*(R_{\Lambda},\sigma_{\Lambda})$ as the left module over itself.

 \begin{lem}\label{lem: multi-Bessel}
      Let $\pi_{\Lambda}$ be $\sigma_{\Lambda}$-projective unitary representation on $R_{\Lambda}$ defined by (\ref{projetive representation of tiling groupoid case}) and let $\{\Phi_1,\dots, \Phi_n\}\subset \mathcal{H}_0=C(\Omega_0(\Lambda))\odot \CS(\mathbb{R}^d)$. Then $\{\Phi_1,\dots, \Phi_n\}$ is a multi-Bessel vector for $\pi_\Lambda$ in the sense of Definition \ref{defn: multi frame}.
    \end{lem}
    \begin{proof}
    Denote by $\vec\Phi=\{\Phi_i\}_{i=1}^n $. 
Without loss of generality, one may assume each $\Phi_i$ is of the form $\Phi_i(T) =h_i(T)g_i\in \CS(\mathbb{R}^d)$ for $ T\in \Omega_0(\Lambda)$ as any element in $\CH_0$ is a sum of functions of this form. Now, let $\Psi \in \{\CH_T=L^2(\mathbb{R}^d)\}_{T \in \Omega_0(\Lambda)}$. Then   
it follows from \cite[Theorem 12.2.1]{G01} that $V_{g_i} \Psi(T)\in W(L^{\infty},\ell^2)(\mathbb{R}^{2d})$ and we have
   \begin{align}
   &\sum_{i=1}^n\sum_{(S, S-z)\in R^T_{\Lambda}}|\langle \Psi(T), \pi(S, S-z)\Phi_i(s(z))\rangle|^2=
  \sum_{i=1}^n\sum_{z\in T}\vert \langle \Psi(T) ,h_i(T-z)\pi(z)g_i\rangle \vert ^2 \\
  &\le (\max\limits_{i=1,\dots,n}\|h_i\|_{\infty})^2\cdot \text{rel}(T)\cdot\sum_{i=1}^n \|V_{g_i}\Psi(T)\|_{W(L^{\infty},\ell^2)}^2 \notag\\
  &\lesssim (\max\limits_{i=1,\dots,n}\|h_i\|_{\infty})^2 \cdot\text{rel}(\Lambda)\cdot\sum_{i=1}^n \|g_i\|_{M^1(\mathbb{R}^d)}^2\cdot\|\Psi(T)\|_{L^2(\mathbb{R}^d)}^2\notag
 %&=B'\|\Psi(T)\|_{L^2(\mathbb{R}^d)}^2,\notag
   \end{align}
  % where $B':=(\max\limits_{i=1,\dots,n}\|h_i\|_{\infty})^2\cdot \text{rel}(\Lambda)\cdot\sum_{i=1}^n \|g_i\|_{M^1(\mathbb{R}^d)}^2.$ 
  Therefore, $\vec\Phi=\{\Phi_1,\dots,\Phi_n\}$ is a multi-Bessel vector for $\pi_{\Lambda}$.
    \end{proof}
 The following is a groupoid version of \cite[Proposition 4.3]{BEV} for the Delone set $\Lambda$ in $\R^{2d}$.
\begin{prop}\label{single}
    Let $\pi_{\Lambda}$ be $\sigma_{\Lambda}$-projective unitary representation on $R_{\Lambda}$ defined by (\ref{projetive representation of tiling groupoid case}) and $\{\Phi_1,\dots, \Phi_n\}\subset \mathcal{H}_0=C(\Omega_0(\Lambda))\odot \CS(\mathbb{R}^d),$ then the following hold:
    \begin{enumerate} [label=(\roman*)]
        \item\label{prop: single 1} The finite set $\{\Phi_1,\dots \Phi_n\} $ is an algebraic generating set for $\mathcal{E}$  if and only if $\{\Phi_1,\dots ,\Phi_n\}$ is  average multi-frame vectors for $\pi_{\Lambda}.$
        \item\label{prop: single 2} The finite set $\{\Phi_1,\dots \Phi_n\} $ is an $C_r^*(R_\Lambda,\sigma)$-linearly independent set in $\mathcal{E}$ with closed $C_r^*(R_\Lambda,\sigma)$-span if and only if $\{\Phi_1,\dots ,\Phi_n\}$ is average multi-Riesz vectors for $\pi_{\Lambda}.$
    \end{enumerate}
\end{prop}
\begin{proof}
First, $\{\Phi_1,\dots \Phi_n\} $ is a multi-Bessel vector of $\pi_\Lambda$ by Lemma \ref{lem: multi-Bessel}. Moreover, it follows from Remark \ref{rmk: frame is average frame} that $\vec\Phi$ is an average multi-Bessel frame vector for $\pi_{\Lambda}.$  
    
    For the average multi-Bessel $\vec\Phi=\{\Phi_1,\dots,\Phi_n\}$, recall its average analysis operator \[\overline{C}_{\vec\Phi}:L^2(\Omega_0(\Lambda), L^2(\mathbb{R}^d),\mu)\to \int_{\Omega_0(\Lambda)}^\oplus \ell^2(R_\Lambda^T\times \{1,\dots,n\})\text{d}\mu(T)=L^2(R_\Lambda, \nu)^n\] is defined to be \[\overline{C}_{\vec\Phi}(\Psi)((T,T-z),i)=\langle\Psi(T),\pi(z)\Phi_i(T-z)\rangle \] for $\Psi \in L^2(\Omega_0(\Lambda), L^2(\mathbb{R}^d),\mu). $ Moreover, it follows from Proposition~\ref{prop: pre module}~\ref{pre module 1} that $\overline{C}_{\vec\Phi}$ maps $\CH_0$ to $\int_{\Omega_0(\Lambda)}^\oplus \ell^1(R_\Lambda^T\times \{1,\dots,n\})\text{d}\mu(T)$, which is also identified with $L^1(R_\Lambda, \nu)^n$.

    Let $\mathscr{C}:\mathcal{E}\to C_r^*(R_\Lambda,\sigma_{\Lambda})^n $ be the module analysis operator associated to $\{\Phi_1,\dots \Phi_n\}$ (see Section \ref{sec: Hilbert mod} for the definition). Proposition~\ref{prop: pre module}~\ref{pre module 1} implies that 
       \[\mathscr{C}|_{\CH_0}:\CH_0\to \ell_1(R_{\Lambda},\sigma_{\Lambda})^n\subseteq L^1(R_{\Lambda},\nu)^n ,\]
       and
       \[\mathscr{C}(\Psi)(T,T-z)=(_{\bullet}\langle\ \Psi,\Phi_i \rangle(T,T-z))_{i=1}^n=(\langle\Psi(T),\pi(z)\Phi_i(T-z)\rangle)_{i=1}^n.\] 
      Therefore, $\mathscr{C}$ coincides with $\overline{C}_{\vec\Phi}$ on ${\mathcal{H}_0}.$ Recall $\mathcal{H}_{\mathcal{E}}^{\tau_{\mu}}\cong \mathcal{H}$ and $\mathcal{H}_{C_r^*(R_{\Lambda},\sigma_{\Lambda})}^{\tau_{\mu}} \cong L^2(R_{\Lambda},\nu)$ hold. Then the localized operator \[\mathscr{C}^{\tau_{\mu}}:\mathcal{H}_{\mathcal{E}}^{\tau_{\mu}} \to \mathcal{H}_{C_r^*(R_{\Lambda,\sigma_{\Lambda}})^n}^{\tau_{\mu}}\] of $\mathscr{C}$ can be identified with $\overline{C}_{\vec\Phi}$ by the density of $\CH_0$. Therefore, the localizations, with respect to the $\tau_\mu$, of the module synthesis operator $\mathscr{D}=\mathscr{C}^*$, the module frame operator $\mathscr{S}$, and  the module Gramian operator $\mathscr{G}$ can be identified with $\overline{D}_{\vec\Phi}$, $\overline{S}_{\vec\Phi}$ and $\overline{G}_{\vec\Phi}$, respectively. After these identifications, we can establish the results.
       
    For \ref{prop: single 1}, it follows from \cite[Lemma 3.1]{BEV} that $\vec\Phi=\{\Phi_1,\dots,\Phi_n\}\subseteq \mathcal{E} $ is a generating set for $\mathcal{E}$ if and only if the associated module frame operator $\mathscr{S}$ is invertible in $\mathcal{L}_{C_r^*(R_{\Lambda},\sigma)}(\mathcal{E})$, which is further equivalent to the invertibility of the localization operator $\mathscr{S}^{\tau_{\mu}}=\overline{S}_{\vec\Phi}$  by \cite[Lemma 3.6]{BEV} and the remark after it. Finally, it is  shown in Proposition \ref{prop:average multi-frame and average frame operator} that the invertibility of $\overline{S}_{\vec\Phi}$ is equivalent to $ \{\Phi_1,\dots,\Phi_n\}$ being average multi-frame vectors.
       
       For \ref{prop: single 2}, it follows from  \cite[Lemma 3.2]{BEV} that $\vec\Phi=\{\Phi_1,\dots,\Phi_n\}\subseteq \mathcal{E} $ is $C_r^*(R_{\Lambda},\sigma)$-linearly independent with closed $C_r^*(R_{\Lambda},\sigma)$-span if and only if the associated module Gramian operator $\mathscr{G}$ is invertible in $\mathcal{L}_{C_r^*(R_{\Lambda},\sigma)}(\mathcal{E})$, which is equivalent to the invertibility of the localization operator $\mathscr{G}^{\tau_{\mu}}=\overline{G}_{\vec\Phi}$  by \cite[Lemma 3.6]{BEV} and the remark after it. However, the equivalence of the invertibility of $\overline{G}_{\vec\Phi}$ and $ \{\Phi_1,\dots,\Phi_n\}$ being average multi-Riesz vectors follows form  Proposition \ref{prop:average multi-frame and average frame operator} as well.
\end{proof}

%\textcolor{blue}{XM: the above is a groupoid version of \cite[Prop. 4.3]{BEV}}

\begin{prop}\label{equ}
    Let $\{\Phi_1,\dots, \Phi_n\}\subset \mathcal{H}_0$ be such that each $\Phi_i$ is of the form $\Phi_i(T)=h_i(T)g_i$ for some $h_i\in C(\Omega_0(\Lambda))$ and $g_i\in \CS(\mathbb{R}^d)$. Then  we have the following.
    \begin{enumerate}[label=(\roman*)]
        \item  $\{\Phi_1,\dots \Phi_n\}$ is a multi-frame vector for $\pi_{\Lambda}$ if and only if $\{\Phi_1,\dots \Phi_n\}$ is an average multi-frame vector for $\pi_{\Lambda}$ with same frame bounds.
        \item $\{\Phi_1,\dots \Phi_n\}$ is a multi-Riesz vector for $\pi_{\Lambda}$ if and only if $\{\Phi_1,\dots \Phi_n\}$ is an average multi-Riesz vector for $\pi_{\Lambda}$ with same Riesz bounds.
    \end{enumerate}   
\end{prop}
 \begin{proof} 
 (\romannumeral 1) It follows from Remark \ref{rmk: frame is average frame} that  if $\{\Phi_1,\dots, \Phi_n\}\subset \mathcal{H}_0$ is a multi-frame vector for $\pi_{\Lambda},$ it is easy to see that $\{\Phi_1,\dots \Phi_n\}$ is an average multi-frame vector for $\pi_{\Lambda}$ with same frame bounds. 
 
Conversely, one may still assume  suppose $\{\Phi_1,\dots \Phi_n\}\subset \mathcal{H}_0$ is an average multi-frame vector with frame bounds $0<A\le B < \infty$ for $\pi_{\Lambda}.$
%It follows from Lemma \ref{lem: multi-Bessel} that $\{\Phi_1,\dots,\Phi_n\}$ is multi-Bessel vector for $\pi_{\Lambda}$ for some bound $B'.$

   %Then the compactness of $\Omega_0(\Lambda)$ allows to find a subcover  $\{N_1^{(i)}\}_{i=1}^{m_1}\subset \{N_{T,1}\}_{T\in \Omega_\punc}$ to cover $\Omega_0(\Lambda).$ Assume that $\mu(N_1^1)>0.$ 
%Noticing that $N^1_1\subseteq \bigcup\limits_{T\in N_1^1}\bar{B}(T,\delta_1),$ so there exists some $T_1\in N_1^1$ such that $\mu(\bar{B}(T,\delta_1)\cap N_1^1)>0.$ Denote $N_1^1 \cap (\bar{B}(T_1,\delta_1)$ by $N_1.$ Similarly, there exists finite clopen set $\{N_2^{(i)}\}_{i=1}^{m_2}\subset \{N_{T,2}\}_{T\in \Omega_\Lambda}$ to cover $N_1.$ We can assume $\mu(N_2^1 \cap N_1)>0$ and there exists $T_2\in N_2^1 \cap N_1$ such that $\mu(\bar{B}(T,\delta_2)\cap N_2^1 \cap N_1)>0.$
%And we define $N_2:=N_2^1 \cap N_1\cap \bar{B}(T_2,\delta_2).$
%By this method, we can construct a decreasing sequence of closed set $\{N_k\}_{k=1}^{\infty}.$ Since $\Omega_\Lambda$ is compact, there exists $T_0\in \bigcap_{k=1}^{\infty} N_k.$ 
    
For all $f\in L^2(\mathbb{R}^d)=M^2(\R^d)$ and $\epsilon >0$, using Lemma~\ref{esitimate of STFT}, there exists $k_0 \in \mathbb{N}_+$ such that for all $T\in \Omega_0(\Lambda),$ we have 
\begin{equation} \label{l2estimate}
    \sum\limits_{i=1}^n\sum\limits_{z\in T\setminus B(0,k_0)}\vert \langle f,\pi(z)h_i(T-z)g_i\rangle\vert ^2 \le \max_{i=1,\dots,n}\|h_i\|_{\infty}\sum\limits_{i=1}^n\sum\limits_{z\in T\setminus B(0,k_0)}\vert \langle f,\pi(z)g_i\rangle\vert ^2\le \frac{\epsilon}{4}.
\end{equation}

Now, let $f\in L^2(\mathbb{R}^d)$. 
 For an integer $k\in \mathbb{N}_+,$ denote by $C_k:=\sup_{x\in \R^d}\vert \Lambda \cap B(x,k)\vert <\infty.$ Then since each $h_i \in C(\Omega_0(\Lambda))$, there exits $ \delta_k >0 $ such that if $S,T \in \Omega_0(\Lambda) $ satisfying $d(S,T)<\delta_k,$ then 
    \[\vert h_i(S)-h_i(T)\vert \le \frac{\epsilon} {4 C_k\cdot\sum_{i=1}^n\|g_i\|^2_2\cdot\|h_i\|_{\infty}\|f\|_2^2}\]
   holds for any $i=1,\dots,n.$ Let $T\in \Omega_0(\Lambda)$ and $k\in \N_+$, define \[F_{T,k}:=\{ p\in B(0,k) : T-p \in \Omega_0(\Lambda)\} .\] 
   In addition, because $\Omega_0(\Lambda)$ is a flat Cantor transversal, there exists a clopen neighborhood $N_{T,k} \ni T$ such that $F_{T,k}=F_{S,k}$ for any $S\in N_{T,k}.$ If necessary, we may shrink $N_{T, k}$ such that such that $d(T-z, S-z)<\delta_k$ for any $S\in N_{T,k}$ and $z\in F_{T, k}$. 
   %Moreover, by induction, we may assume $\{N_{T, k}: k\in \N_+\}$ is decreasing and $\bigcap_{k=1}^\infty N_{T, k}=\{T\}$ because each $N_{T, k}$ is clopen. 
   Because $R_{\Lambda}$ is minimal and each $N_{T, k}$ is non-empty and open, one has $\mu(N_{T, k})>0$.
   
For the $f$ and $k_0\in \N_+$, we define $\Psi_{f}\in L^2(\Omega_0(\Lambda),L^2(\mathbb{R}^d),\mu)$ be such that $\Psi_{f}(T)=f$ for $T \in N_{T,k_0}$ and $\Psi_{f}(T)=0$ for $T \notin N_{T,k_0}.$ Since $\{\Phi_1,\dots \Phi_n\}\subset \mathcal{H}_0$ is an average multi-frame vector with frame bounds $0<A\le B < \infty$ for $\pi_{\Lambda},$ we claim that  there necessarily exists $T_f\in N_{k_0}$ such that \[\sum_{i=1}^n\sum_{z\in T_f} \vert \langle f ,\pi(z)h_i(T_f-z)g_i\rangle \vert ^2 \ge A\|f \|^2 .\]
Otherwise, suppose $\sum_{i=1}^n\sum_{z\in S} \vert \langle f ,\pi(z)h_i(S-z)g_i\rangle \vert ^2 < A\|f \|^2$ holds for all $S\in N_{T,k_0}$. Then 
\begin{align*}
\int_{\Omega_0(\Lambda)}\sum\limits _{i=1}^n\sum\limits_{z\in S} \vert \langle \Psi_{f}(S) ,\pi(z)h_i(S-z)g_i\rangle \vert ^2\text{d}\mu&=  \int_{N_{T,k_0}}\sum\limits _{i=1}^n\sum\limits_{z\in S} \vert \langle f ,\pi(z)h_i(S-z)g_i\rangle \vert ^2\text{d}\mu\\
&< A\|f \|^2 \mu(N_{T,k_0})=A\|\Psi_f\|^2,\end{align*}
\[\] which is a contradiction to Definition \ref{defn: average multi}. 

Now  as $T_f\in N_{T,k_0}$, by our construction of $N_{T,k_0}$, one has $F_{T,k_0}=F_{T_f,k_0}$ and $d(T-z,T_f-z)<\delta_{k_0}$ for all $z\in F_{T,k_0}= F_{T_f,k_0}$. Therefore, one has
\begin{align}
   & \quad \vert \sum_{i=1}^n \sum_{z\in F_{T,k_0}}\vert \langle f, \pi(z)h_i(T-z)g_i \rangle \vert ^2- \sum_{i=1}^n \sum_{z\in F_{T_f,k_0}}\vert \langle f, \pi(z)h_i(T_f-z)g_i \rangle \vert ^2\vert \notag\\
   &=\vert \sum_{z\in F_{T_f,k_0}}\sum_{i=1}^n(\vert h_i(T-z) \vert^2 - \vert h_i(T_f-z) \vert^2)\vert V_{g_i}f(z)\vert^2\vert \notag\\
   &\le  \sum_{z\in F_{T_f,k_0}}\sum_{i=1}^n 2\|f\|_2^2\cdot\|g_i\|_2^2\cdot\|h_i\|_{\infty} \frac{\epsilon} {4C_{k_0}\cdot\sum_{i=1}^n\|g_i\|^2_2\cdot\|h_i\|_{\infty}\|f\|_2^2} \le \frac{\epsilon}{2}.\notag\
\end{align}
Combing with (\ref{l2estimate}), we have 
\begin{align}
    & \quad \sum_{i=1}^n\sum_{z\in T} \vert \langle f ,\pi(z)h_i(T-z)g_i\rangle \vert ^2 \notag\\
    &\ge \sum_{i=1}^n\sum_{z\in F_{T_f,k_0}} \vert \langle f ,\pi(z)h_i(T_f-z)g\rangle \vert ^2 -\frac{\epsilon}{2} + \sum\limits_{i=1}^n\sum\limits_{z\in T_f\setminus B(0,k_0)}\vert \langle f,\pi(z)h_i(T_f-z)g_i\rangle\vert ^2 -\frac{\epsilon}{2}\notag\\
    &\ge A\|f\|_2^2 -\epsilon.\notag
\end{align}
Since $\epsilon$ is arbitrary, one has $ \sum_{i=1}^n\sum_{z\in T} \vert \langle f ,\pi(z)h_i(T-z)g_i\rangle \vert ^2 \ge A\|f\|_2^2$. By the same method, we also have 
\[ \sum_{i=1}^n\sum_{z\in T} \vert \langle f ,\pi(z)h_i(T-z)g_i\rangle \vert ^2 \le B\|f\|_2^2.\]
Therefore, for all $f\in L^2(\mathbb{R}^d)$, we obtain
\[A\|f\|^2_2\le \sum_{i=1}^n\sum_{z\in T} \vert \langle f ,\pi(z)h_i(T-z)g_i\rangle \vert ^2 \le B \|f\|_2^2, \]
%i.e.,$\{\pi_{\Lambda}(R_\Lambda ^{T_0})\Phi_i\}_{i=1}^n $ is a frame for $L^2(\mathbb{R}^d).$
%Obviously, if $T\in \Omega_\punc$ is a translate of $T_0,$ then $\{\pi_{\Lambda}(R_\Lambda ^{T})\Phi_i\}_{i=1}^n $ is also a frame for $L^2(\mathbb{R}^d)$ with the same frame bounds. 
 %   Similar to the proof above and the proof in \cite[Proposition 3.9]{ER}, we obtain that $$ \frac{A}{3}\|f \|^2 \le \sum_{z\in T} \vert \langle f ,\pi(z)\Phi(T-z)\rangle \vert ^2 \text{d}\mu\le B'\|f \|^2 \ \text{ for all} \ T\in \Omega_\punc.$$ 
which implies that $\vec\Phi=\{\Phi_1,\dots, \Phi_n\}\subset \mathcal{H}_0$ is a multi-frame vector for $\pi_{\Lambda}.$

    (\romannumeral 2)
    %First, it follows from \ref{prop: equiv bessel and upper riesz} and Lemma \ref{lem: multi-Bessel} that  $\{\Phi_1,\dots, \Phi_n\}$ has the upper Riesz bound $B':=(\max\limits_{i=1,\dots,n}\|h_i\|_{\infty})^2\cdot \text{rel}(\Lambda)\cdot\sum_{i=1}^n \|g_i\|_{M^1(\mathbb{R}^d)}^2$. We now show that the lower Riesz bound also exists as well.
   Assume that  $\{\Phi_1,\dots \Phi_n\}$ is an average multi-Riesz vector for $\pi_{\Lambda}$ with Riesz bounds $0 < A \le B < \infty$.

Fix $T \in \Omega_0(\Lambda)$. Let $\epsilon >0$ and $c_i^T=(c_i^T(z))_{z\in T}\in \ell^2(R_{\Lambda}^T)$ for $1\le i \le n$, there exist $k_0\in \IN_+$ such that 
\begin{align} \label{estimiate of multi-Riesz vector}
\sum_{i=1}^n\sum_{z\in T\setminus B(0,k_0)}|c_i^T(z)| ^2 \le\frac{\epsilon^2}{4 \cdot(\max_{i=1,\dots,n} \|h_i\|_{\infty}\|g_i\|_2)^2}.
\end{align}
Like (\romannumeral 1), for any integer $k>0$, there exists a clopen neighborhood $N'_{T,k}\ni T$  such that $\mu(N'_{T,k})>0$ and \[|h_i(S-z)-h_i(S'-z)|^2 \le \frac{\epsilon^2}{4\sum_{i=1}^n \|c_i^T\|^2\|g_i\|^2}\] for any $1 \le i\le n,z \in F_{T,k}$ and $S,S'\in N'_{T,k}$. For $1\le i \le n$, we define $c_i \in L^2(R_{\Lambda},\nu)=\int_{\Omega_0(\Lambda)}^\oplus \ell^2(R_\Lambda^S)\text{d}\mu(S)$ be such that $c_i^S(z)=c_i(S,S-z)=c_i^T(z)$ for $S\in N'_{T,k}$ and $z \in F_{T,k}$. Otherwise, $c_i^S(z)=c_i(S,S-z)=0$. 
%Since $\ell^2 \cong \ell^2(R_{\Lambda}^T)$ for $T\in \Omega_0(\Lambda),$ we view $c_i$ as an element in $\ell^2(R_{\Lambda}^T)$ with $\sum_{z\in F_{T,k_0}}|c_i(T,T-z)|^2=\sum_{k\le C_{k_0}}|c_i(k)|^2$ and $c_i(T,T-z)=c_i(S,S-z)$ if $S\in N'_{T,k_0},z\in F_{T,k_0}.$   
% Then for all $S\in N'_{T,k} ,$ we have 
%\begin{align} \label{estimiate of multi-Riesz vector}
  %  &\|\sum_{i=1}^n\sum_{z\in S \setminus B(0,k_0)}c_i^S(z)\pi(z) h_i(T-z)g_i\|^2\le\sum_{i=1}^n\sum_{z\in T \setminus B(0,k_0)}| c_i(T,T-z)h_i(T-z)|^2\|g_i\|^2_2  \notag \\&\le  \max_{i=1,\dots,n} (\|h_i\|_{\infty}\|g_i\|_2)^2\sum_{i=1}^n\sum_{\vert k \vert>M}|c_i(k)|^2 \le\frac{\epsilon^2}{16}. 
%\end{align}
%The inequality (\ref{estimiate of multi-Riesz vector}) plays the same role as (\ref{l2estimate}).
Similar to  (\romannumeral 1), by our construction of $N'_{T,k_0}$ and $\{c_i\}_{i=1}^n$, we can find $T_c\in N'_{T,k_0}$ such that 
\[\|\sum_{i=1}^n\sum_{z\in T_c }c_i(T_c,T_c-z )\pi(z)h_i(T_c-z)g_i\|^2\ge A\sum_{i=1}^n\|c_i^{T_c}\|^2.\]
We also have 
\begin{align*}
   & \| \sum_{i=1}^n \sum_{z\in F_{T,k_0}}  c_i(T,T-z)\pi(z)h_i(T-z)g_i -  \sum_{i=1}^n \sum_{z\in F_{T_c,k_0}} c_i(T_c,T_c-z)\pi(z)h_i(T_c-z)g_i \|^2 \\
   &\le \sum_{z\in F_{T_c,k_0}}\sum_{i=1}^n|h_i(T-z)- h_i(T_c-z) |^2 |c_i^T(z)|^2\|g_i\|^2\\
   &\le \sum_{i=1}^n \|c_i^T\|^2 \cdot \frac{\epsilon^2}{4\sum_{i=1}^n \|c_i^T\|^2}= \frac{\epsilon^2}{4}.
\end{align*}
Combing with (\ref{estimiate of multi-Riesz vector}), the rest is similar to (\romannumeral 1).
 \end{proof}
 
Combining Proposition \ref{single} and Proposition \ref{equ} , we have the following corollary.

\begin{cor} \label{avp} 
Let $\{\Phi_1,\dots \Phi_n\} $ be a finite subset of $ \mathcal{H}_0$ such that each $\Phi_i$ is of the form $\Phi_i(T)=h_i(T)g_i$ for some $h_i\in C(\Omega_0(\Lambda))$ and $g_i\in \CS(\mathbb{R}^d)$. Then the following hold: 
    \begin{enumerate}[label=(\roman*)]
    \item  $\{\Phi_1,\dots \Phi_n\}\subset \mathcal{H}_0 $ is an algebraic generating set for $\mathcal{E}$ if and only if $\{\Phi_1,\dots \Phi_n\}$ is a multi-frame vector for $\pi_{\Lambda}.$
    \item  $\{\Phi_1,\dots \Phi_n\}\subset \mathcal{H}_0 $ is an $C_r^*(R_\Lambda,\sigma)$-linearly independent set in $\mathcal{E}$ with closed $C_r^*(R_\Lambda,\sigma)$ span if and only if $\{\Phi_1,\dots \Phi_n\}$ is a multi-Riesz vector for $\pi_{\Lambda}.$
\end{enumerate}
\end{cor}

We remark that Lemma \ref{equ} and thus Corollary \ref{avp} are true as well for general $\Phi_1,\dots, \Phi_n\in \CH_0$ via a small adjustment of the proof based on the fact that $\CH_0$ is generated by functions of the form $T\mapsto h(T)g$ where $h\in C(\Omega_0(\Lambda))$ and $g\in \CS(\R^d)$. Nevertheless, we still have the following as an application of current Corollary \ref{avp}. 
%The proof should be compared to \cite[Proposition 3]{K}.

\begin{prop}\label{finitely generated}
    The Hilbert $C_r^*(R_{\Lambda},\sigma_{\Lambda})$-module $\mathcal{E}$ constructed in Proposition \ref{module} is finitely generated.
\end{prop}
\begin{proof}
Let $g\in \CS(\mathbb{R}^d).$ First it follows from \cite[Corollary 5.3]{GOR} that there exists $\delta_g >0$ such that whenever $\Gamma$ is relatively separated set and its hole $\rho(\Gamma)\le \delta_g,$ the $G(g,\Gamma)=\{\pi(\gamma)g:\gamma\in \Gamma\}$ forms a frame for $L^2(\mathbb{R}^{d}).$

Now, for the Delone set $\Lambda$, denote by $R=\rho(\Lambda)$, the hole of $\Lambda$. This implies $\bigcup_{\lambda \in \Lambda} \bar{B}(\lambda, R)=\mathbb{R}^{2d}$ by definition. Now because the closed ball $\bar{B}(0, R)$ in $\R^{2d}$ is compact, there exists $\{z_i\}_{i=1}^n \subset \mathbb{R}^{2d}$ such that $\bigcup_{i=1}^n\bar{B}(z_i,\delta_g)\supseteq \bar{B}(0,R).$  
Then define $\Gamma:=\bigcup_{i=1}^n(\Lambda+z_i)$, which is relatively seperated and satisfies 
\[\bigcup_{\gamma \in \Gamma} \bar{B}(\gamma,\delta_g)=\bigcup_{\lambda\in \Lambda}\bigcup_{i=1}^n(\lambda+\bar{B}(z_i,\delta_g))\supset\bigcup_{\lambda\in \Lambda}\bar{B}(\lambda, R)=\mathbb{R}^{2d}.\]
Therefore, one has $\rho(\Gamma)\le \delta_g$ and this implies that $\{\pi(\lambda+z_i)g\}_{\lambda\in \Lambda,1\le i\le n}$ is a frame for $L^2(\mathbb{R}^{2d}).$ Let $g_i:= \pi(z_i)g,i=1,\dots,n,$ then $g_i\in \CS(\mathbb{R}^d) $ and $\bigcup_{i=1}^n\mathcal{G}(g_i,\Lambda)$ is  Gabor multi-frame. Now for each $1\leq i\leq n$, define $\Phi_i\in \CH_0$ by $\Phi\equiv g_i$ on the whole $\Omega_0(\Lambda)$. Note that $\Phi_i$ is exactly the map $T\mapsto 1_{\Omega_0(\Lambda)}(T)\cdot g_i$. Then, it follows from \cite[Proposition 6, Corollary 2]{K} ( see also \cite[Theorem 3.9]{ER}) with Definition \ref{defn: multi frame} that $\{\Phi_1,\dots,\Phi_n\}$ is a multi-frame vector for $\pi_{\Lambda}.$ Then, Corollary \ref{avp} implies that the Hilbert $C_r^*(R_{\Lambda}, \sigma_{\Lambda})$-module $\mathcal{E}$ is algebraic finitely generated.
\end{proof}

\subsection{Existence of full Gabor frames and full Riesz sequences}
Let $\tau$ be a tracial state on $C_r^*(R_{\Lambda},\sigma_{\Lambda})$. Then one may define 
\[\tau(\mathcal{E}):=\sum_{i}\tau(p_{ii})\] for any  projection $p=[p_{ij}]\in M_n(C_r^*(R_{\Lambda},\sigma_{\Lambda}))$ such that $\mathcal{E} \cong C_r^*(R_{\Lambda},\sigma_{\Lambda})^np.$ See, e.g., \cite[Section 5.2]{BEV}.  To avoid ambiguity, in this subsection, we denote by $\mu$ a finite $\R^{2d}$-invariant measure on $\Omega(\Lambda)$ and denote by $\mu_0$ is the transverse measure associated to $\mu$ (see, e.g., \cite[Section 2.3]{ER}).
We then recall that the lower and upper Beurling density of a set $\Lambda \subseteq \mathbb{R}^{2d}$ are defined respectively as 
\[D^-(\Lambda)=\liminf_{R\to \infty}\inf_{z\in \mathbb{R}^{2d}}\frac{|\Lambda \cap B(z, R)|}{\text{vol}(B(0,R))} \ \ \text{and }\ \ D^+(\Lambda)=\limsup_{R\to \infty}\sup_{z\in \mathbb{R}^{2d}}\frac{|\Lambda \cap B(z, R)|}{\text{vol}(B(0, R))}.\]
The covolume of $\Lambda \subseteq \mathbb{R}^{2d}$  is then defined to be the positive number 
\[\text{covol}_{\mu}(\Lambda)=\frac{\mu(\Omega(\Lambda))}{\mu_0(\Omega_0(\Lambda)}\]
Moreover, the lower and upper covolume is defined respectively by \[\text{covol}_-(\Lambda)=\inf_{\mu \in \mathcal{P}_{\mathbb{R}^{2d}}(\Omega(\Lambda))}\text{covol}_\mu(\Lambda), \ \ \text{and } \ \ \text{covol}_+(\Lambda)=\sup_{\mu \in \mathcal{P}_{\mathbb{R}^{2d}}(\Omega(\Lambda))}\text{covol}_\mu(\Lambda).\]
When $\Lambda$ is a Delone set, it was shown in \cite[Theorem 1.2]{ER} that 
\[D^-(\Lambda)=\frac{1}{\text{covol}_+(\Lambda)},\ \ \ \text{and }\ \ D^+(\Lambda)=\frac{1}{\text{covol}_{-}(\Lambda)}.\]

\begin{prop} \label{trace}
    Let $\mu$ be a non-zero finite $\mathbb{R}^{2d}$-invariant measure on the hull $\Omega(\Lambda).$ Denote by $\tau_{\mu_0}$ the faithful tracial state on $C_r^*(R_{\Lambda},\sigma_{\Lambda})$ associated to the invariant probability Radon measure $\mu_0$ on $\Omega_0(\Lambda)$ induced by $\mu.$ Then $\tau_{\mu_0}(\mathcal{E})=\operatorname{covol}_\mu(\Lambda).$
\end{prop}
\begin{proof}
 In the proof of Proposition~\ref{finitely generated}, we may choose $\bigcup_{i=1}^n\mathcal{G}(g_i,\Lambda)$ to be Parseval Gabor multi-frame as  we can always look at $\bigcup_{i=1}^n\mathcal{G}(S^{-1/2}g_i,\Lambda)$ where $S$ is the frame operator associated to $\bigcup_{i=1}^n\mathcal{G}(g_i,\Lambda).$ Then $\{ \Phi_i:\Phi_i(T)=g_i \  \text{for all} \ T\in \Omega_0,i=1,\dots,n \}$ is a Parseval multi-frame vector for $\pi_{\Lambda}.$  Consequently, $\vec \Phi =\{\Phi_i\}_{i=1}^n$ is an average multi-Parseval frame vectors for $\pi_{\Lambda}$ and then the associated average frame operator $\overline{S}_{\vec \Phi}=I_{\CB(\CH)}$ by (\ref{ieq:average multi-frame}). Since the associated module frame operator $\mathscr{S}$ coincides with  $\overline{S}_{\vec \Phi}$ on $\CH_0$ by the proof of Proposition~\ref{single}, we deduce that $\mathscr{S}$ is identity operator and then $\{\Phi_i\}_{i=1 }^n$ is a Parseval frame for $\CE.$

Denote by $A=C^*_r(R_\Lambda,\sigma_\Lambda)$ for simplicity. Denote by $\mathscr{C}$ and $\mathscr{D}$ the module analysis operator and module synthesis operator associated to this Parseval frame $\{\Phi_1,\dots,\Phi_n\}$ in the Hilbert $A$-module $\CE$, respectively (see Section \ref{sec: Hilbert mod}). Then it follows from \cite[Proposition 3.3 (i)]{BEV} (and its proof) that the adjointable isometry $\mathscr{C}:\CE\to A^np $ is an isomorphism where $p=[p_{ij}]\in M_n(A)$ is a projection such that the $i$-th row vector of $p$ is $\mathscr{CD}e_i$. Here $e_i=(\delta_{ij}1_{A})_{j=1}^n\in A^n.$ A simple calculation shows 
   $\mathscr{C}\mathscr{D}e_i=\mathscr{C}(\Phi_i)=(_{\bullet}\langle\Phi_i,\Phi_j \rangle)_{j=1}^n,$  and therefore $p_{ij}={_{\bullet}\langle}\Phi_i,\Phi_j \rangle$. This implies 
   \begin{align}
       \tau_{\mu_0}(\mathcal{E})&=\sum_{i=1}^n\tau_{\mu_0}(p_{ii})=\int_{\Omega_0(\Lambda)}\sum_{i=1}^nE(p_{ii})\text{d}\mu_0\notag
       =\int_{\Omega_0(\Lambda)}\sum_{i=1}^np_{ii}(T,T)\text{d}\mu_0(T)=\int_{\Omega_0(\Lambda)}\sum_{i=1}^n\|g_i\|^2\text{d}\mu_0\notag\\
       &=\text{covol}_\mu(\Lambda). \notag
   \end{align}
   The third equality is due to Proposition \ref{prop: conditional expectation} as each $p_{ii}={_{\bullet}\langle} \Phi_i,\Phi_i \rangle \in  \ell_{1}(R_{\Lambda}, \sigma_\Lambda)\cap C_0(R_\Lambda)$
   and the last equality comes from the multi-window version of \cite[Theorem 3.12]{ER}, which has the same proof.
\end{proof}
\begin{thm} \label{Thm:converse Balian-low}
    Let $\Lambda  \subset \mathbb{R}^{2d}$ a FLC, repetitive and aperiodic Delone set. Then the following hold:
    \begin{enumerate} [label=(\roman*)]
        \item If $D^-(\Lambda)  >1,$ then there exist disjoint sets $\Lambda_i \subset \Lambda,i=1,\dots,n$ and $\vec{g}=\{g_i\}_{i=1}^n\subset \CS(\R^d)$ such that $\bigsqcup\limits_{i=1}^n \Lambda_i=\Lambda$ and 
        $\{\pi(\Lambda_i)g_i\}_{i=1}^n$ is a frame for $L^2(\mathbb{R}^d)$, i.e., $G_F(\Lambda, \vec{g})$ is a full Gabor frame for $L^2(\R^d)$.
        \item  If $D^+(\Lambda)  <1,$ then there exist disjoint sets $\Lambda_i \subset \Lambda,i=1,\dots,n$ and $\vec{g}=\{g_i\}_{i=1}^n\subset \CS(\R^d)$ such that $\bigsqcup\limits_{i=1}^n \Lambda_i=\Lambda$ and 
        $\{\pi(\Lambda_i)g_i\}_{i=1}^n$ is a Riesz sequence for $L^2(\mathbb{R}^d)$, i.e., $G_F(\Lambda, \vec{g})$ is a full Riesz sequence for $L^2(\R^d)$.
    \end{enumerate}
\end{thm}
\begin{proof}
(\romannumeral 1)  According to \cite[Corollary 1.15]{BHK}, there is a bijection between non-zero finite $\mathbb{R}^{2d}$-invariant measure $\mu$ on the hull $\Omega(\Lambda)$ and invariant probability Radon measure $\mu_0$ on $\Omega_0(\Lambda).$
   Meanwhile, since the tiling groupoid $R_{\Lambda}$ is principal, every tracial state $\tau$ on $C_r^*(R_{\Lambda},\sigma_{\Lambda})$ is canonical, i.e., $\tau=\tau_{\mu_0} $ for some invariant probability Radon measure $\mu_0$ on $\Omega_0(\Lambda)$ (see \cite[Proposition \uppercase\expandafter{\romannumeral2 }.5.4]{R}). Suppose  $D^-(\Lambda)>1$. Then from the discussion before Proposition \ref{trace}, we have
   \begin{equation}\label{densityeq}
       \text{covol}_{\mu}(\Lambda)\le \text{covol}_+(\Lambda)=\frac{1}{D^-(\Lambda)}<1,
   \end{equation}
    for all non-zero finite $\mathbb{R}^{2d}$-invariant measure $\mu$ on the continuous hull $\Omega(\Lambda).$ Therefore, we can combine Proposition \ref{trace}  with (\ref{densityeq}) to deduce that $\tau(\mathcal{E})<1$  for all tracial states $\tau$ on $C_r^*(R_{\Lambda},\sigma_\Lambda).$ Applying  Corollary~\ref{cor:C^*_r(R_varphi,sigma) has the strict comparison} and Proposition~\ref{finitely generated} and \cite[Proposition 5.2]{BEV} we can find $\eta \in \mathcal{E}$ such that $\{\eta\}$ is the generating set of $\mathcal{E}.$ As $\Omega_0(\Lambda)$ is Cantor set, the function space $H(\Omega_0(\Lambda)):=\{\sum_{i=1}^na_i\chi_{A_i}:a_i\in \C,A_i\subset \Omega_0  \ \text{is clopen set,} i=1,\dots,n,\ \Omega_0=\bigsqcup\limits_{i=1}^n A_i\}$ is dense in $C(\Omega_0(\Lambda)).$ Let $\Phi \in \mathcal{H}_0$ be such that $\Phi(T)=h(T)g$ for some $h\in C(\Omega_0(\Lambda))$ and $g \in \CS(\R^{d})$ and $\Phi'\in H(\Omega_0(\Lambda)) \odot \CS(\R^d)$ be such that $\Phi'(T)=h'(T)g$ for some $h'\in H(\Omega_0(\Lambda))$. Then we have 
   \begin{align*}
       \|\Phi- \Phi'\|_{\CE}^2&=\|{_\bullet \langle} \Phi-\Phi', \Phi-\Phi' \rangle\|_{C_r^*(R_{\Lambda},\sigma_{\Lambda})}\\
       &\le \|{_\bullet \langle} \Phi-\Phi', \Phi -\Phi'\rangle\|_{\ell^1(R_{\Lambda}, \sigma_{\Lambda})}=\sup_{T \in \Omega_0(\Lambda)}\sum_{z\in T}|\langle (h-h')(T)g, \pi (z) (h-h')(T-z)g \rangle|\\
       &\le \|h-h'\|_{\infty}^2 \sup_{T \in \Omega_0(\Lambda)}\sum_{z\in T}|V_gg(z)|\\
       &\le\|h-h'\|_{\infty}^2 \sup_{T \in \Omega_0(\Lambda)} \operatorname{rel}(T)\|V_gg\|_{W(L^{\infty},\ell^1)}=\|h-h'\|_{\infty}^2  \operatorname{rel}(\Lambda)\|V_g g\|_{W(L^{\infty},\ell^1)}.
   \end{align*}
   Hence $\overline{H(\Omega_0(\Lambda))\odot \CS(\R^d)}^{\|\cdot\|_{\mathcal{E}}}=\overline{\mathcal{H}_0}^{\|\cdot\|_{\mathcal{E}}}=\mathcal{E}.$ Using \cite[Proposition 3.3]{BEV}, we can choose $\eta \in H(\Omega_0(\Lambda))\odot \CS(\R^d)$ and  there exist disjoint clopen sets $A_i \subset \Omega_0(\Lambda),i=1,\dots,n$ and $\{g_i\}_{i=1}^n\subset \CS(\R^d)$ such that $\bigsqcup\limits_{i=1}^n A_i=\Omega_0(\Lambda)$ and $\eta(T)=g_i$ when $T\in A_i.$ Denote $\Lambda_i:=\{z\in\Lambda:\Lambda-z \in A_i\}$ then  $\bigsqcup\limits_{i=1}^n \Lambda_i=\Lambda$ and $\{\pi(\Lambda_i)g_i\}_{i=1}^n$ is a frame for $L^2(\mathbb{R}^d)$ by Corollary~\ref{avp}.

   (\romannumeral 2) If $D^+(\Lambda)<1,$ from the discussion before Proposition \ref{trace}, we have
   $$\text{covol}_{\mu}(\Lambda)\ge \text{covol}_-(\Lambda)=\frac{1}{D^+(\Lambda)}>1,$$
    for all non-zero finite $\mathbb{R}^{2d}$-invariant measure $\mu$ on the continuous hull $\Omega(\Lambda).$ From the proof in (\romannumeral 1), we can obtain that $\tau (\mathcal{E})>1$ for every tracial state $\tau$ on $C_r^*(R_{\Lambda},\sigma).$ Applying  Corollary \ref{cor:C^*_r(R_varphi,sigma) has the strict comparison} and \cite[Proposition 5.2]{BEV} we can find $\eta \in \mathcal{E}$ such that $\{\eta\}$ is closed $C_r^*(R_{\Lambda},\sigma)$-span in $\CE.$ Like (\romannumeral 1), using \cite[Proposition 3.3]{BEV}, we can choose $\eta \in H(\Omega_0(\Lambda))\odot \CS(\R^d)$ and  there exists disjoint clopen sets $A_i \subset \Omega_0(\Lambda)$ for $i=1,\dots,n$, and $\{g_i\}_{i=1}^n\subset \CS(\R^d)$ such that $\bigsqcup\limits_{i=1}^n A_i=\Omega_0(\Lambda)$ and $\eta(T)=g_i$ when $T\in A_i.$ Denote by $\Lambda_i:=\{z\in\Lambda:\Lambda-z \in A_i\}$. Then  $\bigsqcup\limits_{i=1}^n \Lambda_i=\Lambda$ and $\{\pi(\Lambda_i)g_i\}_{i=1}^n$ is a Riesz sequence for $L^2(\mathbb{R}^d)$ by Corollary~\ref{avp}.
\end{proof}

\begin{rmk}
    Since each part $\Lambda_i\subset \Lambda$ is uniformly separated in $\R^{2d}$ and $g_i \in \CS,$ \cite[Corollary 8]{G03} implies that the Gabor system $\{\pi(\Lambda_i)g_i\}$ is a union of Riesz sequence. So we can refine the frame obtained in Theorem~\ref{Thm:converse Balian-low}(\romannumeral 1) such that $\{\pi(\Lambda_i)g_i\}_{i=1}^n$ is a frame for $L^2(\R^d)$ while for each $i$, $\{\pi(\Lambda_i)g_i\}$ is a Riesz sequence for $L^2(\R^d)$.
\end{rmk}

\section{Balian-Low theorem for full Gabor frames and full Riesz sequences}\label{sec: Balian-Low}
In this section,  we prove the ``only if '' part of Theorem \ref{thm: main 1}. In particular, this will strengthen \cite[Theorem 1.1]{CDH}. Overall, the proof is a small modification of the arguments in \cite{GOR}. Denoted by $(\bigoplus_{i\in I} X_i)_p$ the diect sum of spaces $X_i$ under $\ell_p$-norm.

First we recall Beurling's notion of weak convergence of sets (see,e.g.,\cite[Section 4]{GOR}). A sequence $\{\Lambda_n\} $ of subsets of $\R^d$ converges weakly to $\Lambda$, denoted by $\Lambda_n \xrightarrow{w} \Lambda$, if for every $R>0$ and $\epsilon >0$, there exists $n_0\in \N $ such that for all $n \ge n_0$,
\[\Lambda \cap B(0,R) \subseteq \Lambda_n +B(0, \epsilon) \ \text{and} \ \Lambda_n \cap B(0,R) \subseteq \Lambda + B(0 ,\epsilon).\] 
Given a set $\Lambda \subset \R^d$, denoted by $W(\Lambda)$ the set of all weak limits of the translated sets $\Lambda +z$ for $z\in \R^d$. For a relatively separated set $\Lambda \subset \R^d$ and a sequence $\{z_k\} \subseteq \R^d$, it follows from \cite[Lemma 4.5]{GOR} that there exists a subsequence $\{z_{k_n}\}$ and a relatively separated set $\Gamma \subseteq \R^d$ such that $\Lambda+z_{k_n} \xrightarrow{w} \Gamma$.
Given two $n$-tuples of sets $\vec \Lambda=(\Lambda_1, \dots ,\Lambda_n)$ and $\vec \Gamma=(\Gamma_1 ,\dots ,\Gamma_n)$, we say that $\vec \Gamma \in W(\vec \Lambda)$ if there exists a sequence $\{z_k:k\ge 1\} \subseteq \R^d$ such that $\Lambda_i +z_k \xrightarrow{w} \Gamma_i$ for all $1\le i \le n $ (See \cite[Section 2]{GRS20}). 

Let $\{\Lambda_i\}_{i=1}^n\subseteq \R^{2d}$ be relatively separated sets, and denote $\vec\Lambda=(\Lambda_1,\dots,\Lambda_n)$. Let $\vec{g}=(g_1, \dots ,g_n) \in M^1(\R^d)^n$.
We denote the analysis operator associated to $\bigcup_{i=1}^n G(g_i,\Lambda_i)$ by $C_{\vec{g},\vec{\Lambda}}$, defined as
\begin{align*}
    &C_{\vec{g},\vec{\Lambda}}f:=(\langle f,\pi(\lambda)g_i \rangle)_{\lambda\in \Lambda_i,1\le i\le n}, \quad f\in M^{p}(\R^d),\\
&C_{\vec{g},\vec{\Lambda}}^* \vec{c}:=\sum_{i=1}^n\sum_{\lambda \in \Lambda_i}c_{i}(\lambda)\pi(\lambda)g_i, \quad \vec{c}=(c_1 ,\dots ,c_n)\in  (\bigoplus_{i=1}^n \ell^p(\Lambda_i))_p.
\end{align*}
According to \cite[Corollary 12.1.12]{G01} and \cite[Theorem 12.2.1]{G01}, for $p \in [1,\infty]$, it is direct to verify 
\begin{align} 
    &\|C_{\vec{g},\vec{\Lambda}}f\|_{\ell ^p}\lesssim (\sum_{i=1}^n\|g_i\|_{M^1}\operatorname{rel}(\Lambda_i))\|f\|_{M^p},\notag\\
    &\|C_{\vec{g},\vec{\Lambda}}^* \vec{c}\|_{M^p}\lesssim (\sum_{i=1}^n\|g_i\|_{M^1}\operatorname{rel}(\Lambda_i))\|\vec{c}\|_{\ell^p}. \label{full Bessel sequence}
\end{align}
%Therefore, if we restrict operators  $C_{\vec{g},\vec{\Lambda}}$ and $C_{\vec{g},\vec{\Lambda}}^*$ to $M^p(\R^d)$ and $\ell^p(\vec \Lambda)$, respectively, operators $C_{\vec{g},\vec{\Lambda}}:M^p(\R^d)\to \ell^p(\vec \Lambda)$ and $C_{\vec{g},\vec{\Lambda}}^*:\ell^p(\vec \Lambda) \to M^p(\R^d)$ are well-defined. 
We say  $\bigcup_{i=1}^n G(g_i,\Lambda_i)$ is a p-frame for $M^p(\R^d)$ if $C_{\vec{g},\vec{\Lambda}}: M^p(\R^d) \to (\bigoplus_{i=1}^n \ell^p(\Lambda_i))_p$ is bounded below and  $\bigcup_{i=1}^n G(g_i,\Lambda_i)$ is a p-Riesz sequence for $M^p(\R^d)$ if $C^*_{\vec{g},\vec{\Lambda}}:(\bigoplus_{i=1}^n \ell^p(\Lambda_i))_p \to M^p(\R^d)  $ is bounded below. 

For $z=(x,\omega)\in \R^{2d}$, denote $\vec \Lambda +z=(\Lambda_1 +z, \dots,\Lambda_n +z)$. The twisted shift operator $\kappa(z):(\bigoplus_{i=1}^n \ell^{\infty}(\Lambda_i))_{\infty} \to (\bigoplus_{i=1}^n \ell^{\infty}(\Lambda_i+z))_{\infty}$ is defined by 
\[(\kappa(z)\vec c)_{i}(\lambda+z):=e^{-2\pi x \lambda_2}c_{i}(\lambda), \quad \lambda=(\lambda_1,\lambda_2)\in \Lambda_i, 1\le i \le n.\]
The commutation relations of time-frequency shift operator (\ref{eq:commutation relation}) implies that 
\begin{equation} \label{relation between synthesis operator}
    \pi(z)C_{\vec{g},\vec{\Lambda}}^*=C_{\vec{g},\vec{\Lambda}+z}^*\kappa (z) \ \text{and}\ e^{2\pi ix \omega}C_{\vec{g},\vec{\Lambda}}\pi(-z)=e^{-2\pi ix \omega}\kappa(-z)C_{\vec{g},\vec{\Lambda}+z}. 
\end{equation}

\begin{defn}[Time-frequency molecules {\cite[section 3]{GOR}}] \label{def:time-frequncy molecules}
    We say $\{f_\lambda:\lambda\in \Lambda\}\subseteq L^2(\R^d)$ is a set of time-frequency molecules if $\Lambda\subseteq \R^{2d}$ is a relatively separated set, and there exists a non-zero function $g\in M^1(\R^d)$ as well as an envelope function $\Phi \in W(L^{\infty},\ell^1)(\R^{2d})$ such that 
    \[|V_g f_\lambda(z)| \le \Phi(z-\lambda), \quad \text{a.e.} \ z\in \R^d, \lambda \in \Lambda.\]
\end{defn}
    Let $\{\Lambda_i\}_{i=1}^n\subseteq \R^{2d}$ be relatively separated sets. Taking $g,g_1,\dots, g_n\in M^1(\R^d)$ and $\Phi=\sum_{i=1}^n|V_gg_i|\in W(L^{\infty},\ell ^1)(\R^{2d})$, we have
    \[|V_g \pi(\lambda)g_i(z)|=|V_gg_i(z-\lambda)|\le \sum_{i=1}^n|V_gg_i(z-\lambda)|=\Phi(z-\lambda).\]
    Therefore, the multi-window Gabor system $\bigcup_{i=1}^n G(g_i,\Lambda_i)$ is a set of time-frequency molecules. Using \cite[Theorem 3.2]{GOR}, we have that if  $\bigcup_{i=1}^n G(g_i,\Lambda_i)$ is a p-frame for $M^p(\R^d)$ for some $p \in [1,\infty]$, then it is a p-frame for $M^p(\R^d)$ for all $p\in [1,\infty]$. Similarly, if $\bigcup_{i=1}^n G(g_i,\Lambda_i)$ is a p-Riesz sequence for $M^p(\R^d)$ for some $p \in [1,\infty]$, then it is a p-Riesz sequence for $M^p(\R^d)$ for all $p\in [1,\infty]$.

\begin{thm}\label{thm:characterization of full frames}
    Assume that  $\Lambda_1,\dots,\Lambda_n$ are relatively separated sets in $\R^{2d}$ and  $\vec{g}=(g_1, \dots ,g_n) \in  M^1(\R^d)^n$. Set $\vec\Lambda=(\Lambda_1,\dots ,\Lambda_n)$ Then the following are equivalent.
    \begin{enumerate}[label=(\roman*)]
    \item  $\bigcup_{i=1}^n G(g_i,\Lambda_i)$ is a frame for $L^2(\R^d)$.
    \item\label{thm:characterization of full frames: condition 2} $C_{\vec{g},\vec{\Gamma}}$ is one-to-one from $M^{\infty}(\R^d)$ to $(\bigoplus_{i=1}^n \ell^{\infty}(\Gamma_i))_{\infty}$ for every weak limit $\vec{\Gamma}=(\Gamma_1,\dots ,\Gamma_n) \in W(\vec{\Lambda})$. 
\end{enumerate}
\end{thm}
\begin{proof} 
    (\romannumeral 1) $\Rightarrow$ (\romannumeral 2). Suppose $\bigcup_{i=1}^n G(g_i,\Lambda_i)$ is a frame for $L^2(\R^d)$. By \cite[Theorem 3.2 (a)]{GOR}, the multi-window system $\bigcup_{i=1}^n G(g_i,\Lambda_i)$ is a $\infty$-frame for $M^{\infty}(\R^d)$, i.e., the operator $C_{\vec{g},\vec{\Lambda}}: M^{\infty}(\R^d) \to (\bigoplus_{i=1}^n \ell^{\infty}(\Lambda_i))_{\infty}$ is bounded below. Thus, $C_{\vec{g},\vec{\Lambda}}^*$ is surjective from $(\bigoplus_{i=1}^n \ell^1(\Lambda_i))_1$ onto $M^1(\R^d)$. For $\vec\Gamma=(\Gamma_1 ,\dots ,\Gamma_n)\in W(\vec\Lambda),$ there exists a sequence $\{z_k\}_{k\in \N}\subset \R^{2d}$ such that $\Lambda_i -z_k \xrightarrow{w} \Gamma_i$. It follows from (\ref{relation between synthesis operator}) that $C_{\vec{g},\vec{\Lambda}-z_k}^*$ are also surjective from $(\bigoplus_{i=1}^n \ell^1(\Lambda_i-z_k))_1$ to $M^1(\R^d)$. So for every $f\in M^1(\R^d)$ and $k\in \N$,
    there exists a vector $\vec c^k =(c_1^k.\dots,c_n^k) \in (\bigoplus_{i=1}^n \ell^1(\Lambda_i-z_k))_1$,
    %and $k\in \N$, there exist a sequence $\{c_{i}^k(\lambda)\}_{\lambda\in \Lambda_i -z_k}\in \ell^1(\Lambda_i-z_k)$
    %with $\|\{c_{i}^k(\lambda)\}_{\lambda\in \Lambda_i -z_k}\|_1\le C$
    such that 
    \[f=\sum_{i=1}^n\sum_{\lambda\in \Lambda _i-z_k}c_{i}^{k}(\lambda)\pi(\lambda)g_i\]  converges in $M^1(\R^d)$. Using the open mapping theorem, the map  $C^*_{\vec{g},\vec{\Lambda}}: (\bigoplus_{i=1}^n \ell^{1}(\Lambda_i))_{1} \to M^{1}(R^d)$ is open. Then by(\ref{relation between synthesis operator}), there exists a constant $C>0$ such that \[\| \vec c^k\|_1=\|\kappa (z_k)^{-1}(\vec c^k)\|_1 \le C\|  C_{\vec{g},\vec{\Lambda}}^*(\kappa (z_k)^{-1}(\vec c^k))\|_{M^{1}}=C\|\pi(z_k)^* f\|_{M^1} \le C\|f\|_{M^1}\] for all $k\in \N$.
    
    %we have that the preimages of operators $C_{\vec{g},\vec{\Lambda}-z_k}^*$ are bounded  independent of $k$. 
For $1\le i\le n$, consider the measures $\mu_k^i:=\sum_{\lambda\in \Lambda_i-z_k}c_{i}^k(\lambda) \delta_\lambda$. Because $\|\mu_k^i\|_{\CM}=\|c_{i}^k\|_{\ell^1(\Lambda_i-z_k)}\le C\|f\|_{M^1}$ for all $k\in \N$, we can pass to a subsequence and then assume that $\mu_k^i \rightarrow \mu^i$ under weak*-topology $\sigma(\CM,C_0)$ for some measure $\mu^i \in \CM(\R^{2d})$. It follows from \cite[Lemma 4.3]{GOR} that  $\supp(\mu^i)\subseteq \Gamma_i$ and then $\mu^i=\sum_{\gamma\in \Gamma_i}c_{i}(\gamma)\delta_\gamma$ for some $\{c_{i}(\gamma)\}_{\gamma \in \Gamma_i}\in \ell^1(\Gamma_i)$. Define $\mu_k=\sum_{i=1}^n\mu_k^i$, and $f'=\sum_{i=1}^n\sum _{\gamma\in \Gamma_i}c_{i}(\gamma)\pi(\gamma)g_i\in M^1(\R^d)$ by (\ref{full Bessel sequence}).
Taking some nonzero function $g\in M^1(\R^d)$,
since $V_{g_i} \pi(z)g \in W(C_0,\ell^1)(\R^{2d}) \subseteq C_0(\R^{2d})$ and $\mu_k^i \rightarrow \mu^i$ under the topology $\sigma(\CM,C_0)$ for $1\le i \le n$, we compute
\begin{align*}
    \langle f,\pi(z)g\rangle&=\sum_{i=1}^k\sum_{\lambda\in \Lambda_i-z_k}c_{i}^k(\lambda) \overline{V_{g_i} \pi(z)g(\lambda)} \\
    &=\sum_{i=1}^k \int_{\R^d}\overline{V_{g_i}\pi(z)g}\text{d}\mu_k^i \rightarrow \sum_{i=1}^k \int_{\R^d}\overline{V_{g_i}\pi(z)g}\text{d}\mu^i=\langle f',\pi(z)g \rangle.
\end{align*}
 Hence $f=f'$ and then $C_{\vec{g},\vec{\Gamma}}^*$ is surjective from $(\bigoplus_{i=1}^n \ell^1(\Gamma_i))_1$ onto $M^1(\R^d)$. It follows that  $C_{\vec{g},\vec{\Gamma}}$ is one-to-one from $M^{\infty}(\R^d)$ to $(\bigoplus_{i=1}^n \ell^{\infty}(\Lambda_i))_{\infty}$.
%W(\Lambda) and $\Omega(\Lambda); f'\in M^1(\R^d).$

 (\romannumeral 2) $\Rightarrow$ (\romannumeral 1). Suppose $\bigcup_{i=1}^n G(g_i,\Lambda_i)$ is not a frame for $L^2(\R^d)$. Then it is not a $\infty$-frame for $M^{\infty}(\R^d)$ by \cite[Theorem 3.2(a)]{GOR}. So there exists a sequence of nonzero functions $\{f_k\}_k \subset M^{\infty}(\R^d)$ such that $\|V_{g_1}f_k\|_{\infty}=1$ and 
 \[\sup_{1\le i\le n , \lambda\in \Lambda_i}|V_{g_i}f_k(\lambda)|\rightarrow 0, \ \text{as} \ k\to \infty.\] It follows that, by pass to a subsequence if necessary, we can choose  $\{z_k: k\ge 1\} \subseteq \R^{2d} $ such that $|V_{g_1}f_k(z_k)|\ge 1/2$, and function $h_k:=\pi(-z_k)f_k$ converges to some $h\in M^{\infty}(\R^d)$ under the topology $\sigma(M^{\infty},M^1)$ as $k \to \infty$, and $\Lambda_i-z_k \xrightarrow{w} \Gamma_i$ for some closed relatively separated set $\Gamma_i \in W(\Lambda_i)$. Note that $|V_{g_1}h_k(0)|=|V_{g_1}f_k(z_k)|\ge 1/2,$ by \cite[Lemma 2.1(b)]{GOR}, and thus $h$ is nonzero. Given $\gamma\in \Gamma_i$, there exists $\lambda_{k,i}\in\Lambda_i $ such that $\lambda_{k,i}-z_k \rightarrow \gamma$. Using \cite[Lemma 2.1(b)]{GOR} again, we have \[|\langle h,\pi(\gamma)g_i \rangle|=|V_{g_i}h(\gamma)|=\lim_k |V_{g_i}h(\lambda_{k,i} -z_k)|=\lim_k |V_{g_i}h_k(\lambda_{k,i} -z_k)|=\lim_k|V_{g_i}f_k(\lambda_{k,i})|=0.\] 
 This implies that $\operatorname{ker}(C_{\vec{g}, \vec{\Gamma}})\neq 0$, which is a contraction to the condition \ref{thm:characterization of full frames: condition 2}.
\end{proof}

The following is a vector-valued version of Lemma 6.8 in \cite{GOR} with the same proof. So we omit it.
\begin{lem}\label{lem:deformation still in W(lambda)}
    Suppose $\Lambda_1 ,\dots,\Lambda_n$ are relatively separated sets in $\R^{d}$ and $\lim _{k \to \infty} \alpha_k =1$. Set $\vec\Lambda=(\Lambda_1 ,\dots, \Lambda_n)$. Then the following holds.  
    \begin{enumerate}[label=(\roman*)]
    \item \label{lem:deformation still in W(lambda) condition 1} Let $\vec \Gamma =(\Gamma_1, \dots, \Gamma_n)$ and $\{\lambda_k\}_{k\in \N_+} \subseteq \bigcup_{i=1}^n \Lambda_i$. If $\alpha_k\Lambda_i-\alpha_k \lambda_k \xrightarrow{w} \Gamma_i$ for all $1 \le i\le n$, then $\vec {\Gamma} \in W(\vec {\Lambda})$.
    \item Suppose that $\bigcup_{i=1}^n \Lambda_i$ is relatively dense and $\{z_k: k \ge 1\}\subseteq \R^{d}$. If $\alpha_k \Lambda_i -z_k \xrightarrow{w} \Gamma_i$, then $\vec {\Gamma}=(\Gamma_1, \dots ,\Gamma_n) \in W(\vec \Lambda)$.
    \end{enumerate}
\end{lem}

\begin{lem}\label{lem: continuity property of synthesis operator}
    Assume that $\{\Lambda_i\}_{i=1}^n$ are uniformly separated sets in $\R^{2d}$ and $\vec{g}=(g_1,\dots ,g_n)\in M^1(\R^d)^n$. Set $\vec\Lambda=(\Lambda_1, \dots ,\Lambda_n)$. For every $k \in \N$, let $\alpha_k \in \R^{2d}$ satisfies $\alpha_k \ge 1$ and $\lim_k \alpha_k =1$. Suppose $\vec{c}^k=(c_1^k,\dots ,c_n^k) \in (\bigoplus_{i=1}^n\ell^{\infty}(\alpha_k\Lambda_i))_{\infty}$ be such that $\|\vec{c}^k\|_{\infty}=1$ and $\lim_k\|\sum _{i=1}^n \sum_{\lambda \in \alpha_k \Lambda_i}c_{i}^k(\lambda)\pi(\lambda)g_i\|_{M^{\infty}}=0$. Then there exist $n$-tuples of uniformly separated sets $\vec\Gamma = (\Gamma_1,\dots ,\Gamma_n)\in W(\vec\Lambda)$, a subsequence $\{\lambda_{k_m}\} \subset  \bigcup_{i=1}^n\Lambda_i$, and a non-zero sequence $\vec c =(c_1, \dots ,c_n)\in (\bigoplus_{i=1}^n\ell^{\infty}(\Gamma_i))_{\infty}$ such that 
    \[ \alpha_{k_m}\Lambda_i-\lambda_{k_m} \xrightarrow{w} \Gamma_i  \ \text{for  \ all}\  1\le i \le n\  \text{as} \ m\to \infty \ \text{and} \ \sum_{i=1}^n\sum_{\gamma \in \Gamma_i}c_{i}(\gamma) \pi (\gamma)g_i=0.\]
\end{lem}
\begin{proof}
    Under the assumption $\|\vec{c}^k\|_{\infty}=1$ for every $k \ge 1$, we can choose $\lambda_{k} \in \alpha_k \Lambda_{i(k)}$ for all $1\le i \le n$ such that $|c^k_{i(k)}(\lambda_{k})|\ge 1/2$.
    Since $\operatorname{rel}(\alpha_k\Lambda_i)=\sup_x|\Lambda_i \cap B(x,1/\alpha_k)|$, one has $\operatorname{rel}(\alpha_k\Lambda_i)\le  \operatorname{rel}(\Lambda_i)$ for all $k \in \N$ and $1\le i \le n$.
    It follows from \cite[Lemma 4.5(b)]{GOR}, we can assume  $\alpha_k\Lambda_i-\lambda_k \xrightarrow{w} \Gamma_i$ by passing to a subsequence again  for some relatively separated set $\Gamma_i\subseteq \R^{2d}$ for all $1\le i \le n$. By Lemma \ref{lem:deformation still in W(lambda)} \ref{lem:deformation still in W(lambda) condition 1}, we have $\vec \Gamma  = (\Gamma_1 ,\dots ,\Gamma_n)\in W(\vec\Lambda)$ as $\lambda_k \in \alpha_k\Lambda_{i(k)}$. Moreover, since $\Lambda_i$ is uniformly separated, each $\Gamma_i$ is uniformly separated as well for $1\le i \le n$. 
    
    For $1\le i\le n$, consider the measure $\mu _{k,i}:=\sum_{\lambda \in \alpha_k \Lambda_i}\sigma(-\lambda_k,\lambda)c_{i}^k(\lambda) \delta_{\lambda - \lambda_k}$ in which $\sigma$ is the $2$-cocycle defined in (\ref{eq:commutation relation}). Using \cite[Lemma 4.6]{GOR}, one has $\|\mu _{k,i}\|_{W(\CM,L^{\infty})}\lesssim \operatorname{rel}(\alpha_k \Lambda_i - \lambda_k)\|c^k\|_{\infty}\le \operatorname{rel}(\Lambda_i)$. Then we assume $\mu_{k,i} \rightarrow \mu _i$ under the topology$\sigma(W(\CM,L^{\infty}),W(C_0,\ell^1))(\R^{2d})$ by passing to a subsequence for some $\mu_i \in W(\CM, L^{\infty})(\R^{2d})$.
     Since $\supp(\mu_{k,i})\subseteq \alpha_k\Lambda_i -\lambda_k$, it follows from \cite[Lemma 4.3]{GOR} that $\supp(\mu_i) \subseteq \Gamma_i$. Hence, $\mu_i=\sum_{\gamma \in \Gamma_i} c_{i}(\gamma)\delta _{\gamma}$ for some $c_i=\{c_{i }(\gamma)\}_{\gamma \in \Gamma_i} \in \ell^{\infty}(\Gamma_i)$. We define $\mu:= \sum_{i=1}^n \mu_i$ and  $\mu_k:= \sum_{i=1}^n \mu_{k,i}$. Then $\mu=\sum_{i=1}^n \sum_{\gamma \in \Gamma_i}c_{i}(\gamma) \delta_\gamma$  for $\vec c =(c_1 ,\dots ,c_n) \in (\bigoplus \ell^{\infty}( {\Gamma_i}))_{\infty}$.

     Define $r:=\min_{1\le i\le n}\inf_{k}\inf\{d(x, y): x\neq y\in \alpha_k\Lambda_{i}\}>0$. Let $\varphi \in W(C_0, \ell^1)(\R^{2d})$ be a real-valued function, supported on $B(0,r/2)$ and $\varphi(0)=1$. Then $B(\lambda_{k},r/2) \cap \alpha_k \Lambda_{i(k)} =\{\lambda_{k}\}$ for each $ k \in \N$ and 
     \[\left|\int_{\R^{2d}}\varphi \text{d}\mu_i\right|= \lim _{k} \left|\int_{\R^{2d}}\varphi \text{d}\mu_{k,i} \right|=\lim_k|c_{i(k)}^k(\lambda_k)|\ge 1/2.\]
     Thus $\mu_i \neq 0$. If follows that $c_i \neq 0$. Hence $\vec {c} \ne 0$.

     Taking $g\in M^1(\R^{2d})$ and $z\in \R^{2d}$, one has $V_{g_i} \pi (z)g \in W(C_0,\ell^1)(\R^{2d})$. We compute
     \begin{align*}
         \bigl| \bigl \langle \sum_{i=1}^n \sum_{\gamma \in \Gamma_i} c_{i}(\gamma)\pi(\gamma)g_i ,\pi(z)g  \bigr \rangle  \bigr | &= \left|\sum_{i=1}^n  \int_{\R^{2d}} \overline{V_{g_i}\pi(z)g} \text{d} \mu _i \right|=\lim_k \left| \sum_{i=1}^n  \int_{\R^{2d}} \overline{V_{g_i}\pi(z)g} \text{d} \mu _{k,i} \right|\\
         &= \lim _k \bigl| \bigl\langle \sum_{i=1}^n\sum_{\lambda \in \alpha_k\Lambda_i} \sigma(-\lambda_k , \lambda)c_{i}^k(\lambda) \pi(\lambda -\lambda_k)g_i, \pi (z) g \bigr\rangle \bigr|\\
         & \le \lim_k\|\sum_{i=1}^n\sum_{\lambda \in \alpha_k\Lambda_i} \sigma(-\lambda_k , \lambda)c_{i}^k(\lambda) \pi(\lambda -\lambda_k)g_i\|_{M^{\infty}} \|g\|_{M^1}\\
         &= \lim_k\|\pi(-\lambda_k)\sum_{i=1}^n \sum_{\lambda \in \alpha_k\Lambda_i} c_{i}^k(\lambda) \pi(\lambda )g_i\|_{M^{\infty}} \|g\|_{M^1}\\
         &= \lim_k\|\sum_{i=1}^n\sum_{\lambda \in \alpha_k\Lambda_i} c_{i}^k(\lambda) \pi(\lambda)g_i\|_{M^{\infty}} \|g\|_{M^1}=0.
         \end{align*}
     Therefore, we have $V_g(\sum_{i=1}^n \sum_{\gamma \in \Gamma} c_{i}(\gamma)\pi(\gamma)g_i) \equiv0$, which implies that $\sum_{i=1}^n \sum_{\gamma \in \Gamma} c_{i}(\gamma)\pi(\gamma)g_i=0$.
     \end{proof}

\begin{thm} \label{thm:characterization of full Riesz sequence }
      Assume that  $\Lambda_1,\dots,\Lambda_n$ are uniformly separated sets in $\R^{2d}$ and  $\vec{g}=(g_1, \dots ,g_n)\in  M^1(\R^d)^n$. Set $\vec\Lambda=(\Lambda_1 ,\dots ,\Lambda_n)$. Then the following are equivalent.
    \begin{enumerate}[label=(\roman*)]
    \item  $\bigcup_{i=1}^n G(g_i,\Lambda_i)$ is a Riesz sequence in $L^2(\R^d)$.
    \item $C_{\vec{g},\vec{\Gamma}}^*$ is one-to-one from  $(\bigoplus_{i=1}^n\ell^{\infty}( \Gamma_i))_{\infty}$  to $M^{\infty}(\R^d)$ for every weak limit $\vec\Gamma=(\Gamma_1 ,\dots ,\Gamma_n) \in W(\vec{\Lambda})$. 
    \end{enumerate}
\end{thm}
\begin{proof}
    (\romannumeral1) $\Rightarrow$ (\romannumeral2). Suppose $\bigcup_{i=1}^n G(g_i,\Lambda_i)$ is a Riesz sequence for $L^2(\R^d)$. Then $\bigcup_{i=1}^n G(g_i,\Lambda_i)$ is a $\infty$-Riesz sequence for $M^{\infty}(\R^d)$ by \cite[Theorem 3.2(b)]{GOR}. It follows that $C_{\vec{g},\vec\Lambda}$ is surjective from $M^1(\R)^d$ to $(\bigoplus_{i=1}^n\ell^1( \Lambda_i))_1$. For $\vec\Gamma=(\Gamma_1 ,\dots ,\Gamma_n) \in W(\vec\Lambda)$, we assume that $\Lambda_i -z_k \xrightarrow{w} \Gamma_i$ for all $1 \le i\le n$. Fix $1\le i \le n$ and $\gamma \in \Gamma _i$, there exists $\{\lambda_{k,i}\}_{k\in \N_+} \subset \Lambda_i $ such that
    $\lambda_{k,i} -z_k \rightarrow \gamma$. For each $k\in \N$ and $1\le i\le n$, let $\vec c^k=(c_1^k,\dots,c_n^k)\in  (\bigoplus_{i=1}^n\ell ^1 (\Lambda_i-z_k))_i$ be such that
       \[\begin{cases}
  &c_j^k (\lambda-z_k) =1 \text{ if } (j,\lambda)= (i,\lambda_{k,i}),\\
  &c_j^k (\lambda-z_k) =0 \text{ if } (j,\lambda) \ne(i,\lambda_{k,i}).
\end{cases}\]
    By (\ref{relation between synthesis operator}) and the open mapping theorem similar to Theorem \ref{thm:characterization of full frames}, the operators $C_{\vec{g},\vec \Lambda-z_k}$ are also surjective from $M^1(\R^d)$ to $(\bigoplus _{i=1}^n\ell^1(\lambda_i-z_k))_1$ and with bounds of preimages independent of $k$. It follows that we can find a function $h_{k,i}^{\gamma}\in M^1(\R^d)$ such that  $\|h_{k,i}^\gamma\|_{M^1} \le 1$ and $\vec c^k=C_{\vec{g},\vec \Lambda-z_k}(h_{k,i}^{\gamma})$, i.e.,  
    \begin{equation}\label{eq:surjective of analysuis op}
        V_{g_i}h_{k,i}^\gamma(\lambda_{k,i}-z_k)=1 \ \text{and} \ V_{g_j}h_{k,i}^\gamma(\lambda-z_k)= 0 \ \text{when} \ (j,\lambda)\ne (i,\lambda_k^i).
    \end{equation}
     Without loss of generality, we can assume that $h_{k,i}^\gamma$ converges to some function $h^{\gamma}_i\in M^1(\R^{2d})$ under the topology $\sigma(M^1,M^0)$. It follows from \cite[Lemma 2.1(b)]{GOR} that 
    \[V_{g_i}h^{\gamma}_i(\gamma)=\lim _k V_{g_i}h_{k,i}^\gamma(\lambda_{k,i}-z_k)=1.\]
    For $(j,\gamma')\ne (i,\gamma) $,
    there exists $\lambda'_{k,j}\in \Lambda_j$ such that $\lambda' _{k,j} -z_k \rightarrow \gamma'$. Using (\ref{eq:surjective of analysuis op}) and \cite[Lemma 2.1(b)]{GOR} again,  we have
    \[V_{g_j}h^{\gamma}_i(\gamma')=\lim _k V_{g_j}h_{k,i}^\gamma(\lambda'_{k,i}-z_k)=0.\]
    Therefore, for each $1\le i\le n$ and $\gamma \in \Lambda_i$, we have already obtained a function $h^\gamma_i\in M^1(\R^d)$ such that $\|h_i^{\gamma}\|_{M^1}\le 1$, $V_{g_i} h^{\gamma}_i(\gamma)=1$ and $V_{g_j} h^{\gamma}_i(\gamma')=0$ when $(j,\gamma')\ne (i,\gamma)$. For $\vec c=(c_1,\dots ,c_n)\in (\bigoplus _{i=1}^n\ell^1(\Gamma_i))_1$, we can define a function 
    \[f:=\sum_{i=1}^n \sum_{\gamma \in \Gamma_i}c_{i}(\gamma) h^{\gamma}_i.\]
    It follows that $f\in M^1(\R^d)$ and that $C_{\vec{g},\vec \Gamma}f=\vec c$. Hence  $C_{\vec{g},\vec \Gamma}$ is surjective from $M^1(\R^d) $ to $(\bigoplus_{i=1}^n\ell^1(\Gamma_i))_1$. It follows that $C^*_{\vec{g},\vec\Gamma}$ is one-to-one from $M^{\infty}(\R^d)$ to $(\bigoplus_{i=1}^n\ell^{\infty}( \Gamma_i))_1$.

    (\romannumeral 2) $\Rightarrow$ (\romannumeral 1). Suppose $\bigcup_{i=1}^n G(g_i,\Lambda_i)$ is not a Riesz sequence for $L^2(\R^d)$. Then it is not a $\infty$-Riesz sequence for $M^{\infty}(\R^d)$ by \cite[Theorem 3.2(b)]{GOR}. So there exists a sequence $\{\vec{c}^k\}_{k \in \N} \subseteq (\bigoplus_{i=1}^n\ell^\infty (\Lambda_i))_1$  such that $\|\vec {c}^k\|_{\infty}=1$ and 
    \[\lim_{k\to \infty}\|\sum_{i=1}^n \sum_{\lambda \in \Lambda_i}c_{i}^k(\lambda) \pi(\lambda)g_i\|_{M^{\infty}}=0.\]

    Applying Lemma~\ref{lem: continuity property of synthesis operator} with $\alpha_k=1$, we can obtain n-tuples of  uniformly separated sets $\vec \Gamma=(\Gamma_1, \dots ,\Gamma_n) \in W(\vec \Lambda)$ and a non-zero sequence $\vec c=(c_1 ,\dots,c_n)\in (\bigoplus_{i=1}^n\ell^\infty ( \Gamma_i))_1$ such that $\sum_{i=1}^n\sum_{\gamma\in \Gamma_i}c_{i}(\gamma)\pi(\gamma)g_i=C^*_{\vec{g},\vec \Gamma}\vec c=0$. This leads to a contraction.
\end{proof}
The following proposition for full Riesz sequence is a dual result of  \cite[Theorem 1.1]{CDH} for frames.
\begin{prop} \label{prop:Balian-low of full Riesz sequence}
     Let $\Lambda=\bigsqcup_{i=1}^n \Lambda_i \subset \R^d$ be a discrete set of disjoint of union of $\Lambda_i$. For each $1\le i \le n$, choose a nonzero function $g_i\in L^2(\R^d)$. If $\{\pi(\Lambda_i)g_i\}_{i=1}^n$ is a Riesz sequence for $L^2(\mathbb{R}^d)$, then $\Lambda_i$ is uniformly separated for $1\le i\le n$ and
         $D^+(\Lambda)  \le1$.
\end{prop}
\begin{proof}
Suppose  $\{\pi(\Lambda_i)g_i\}_{i=1}^n$ is a Riesz sequence for $L^2(\mathbb{R}^d)$ with lower Riesz bound $A>0$ and $\Lambda_i$ is not uniformly separated. Then there exists sequences $\{\lambda_n\}$ and $\{\lambda'_n\} \subset \Lambda_i$ with $\lambda_n \ne \lambda '_n$ such that $|\lambda_n-\lambda'_n| \to 0$ as $n \to \infty$. It follows that $\lim_{n\to \infty} \|\pi(\lambda_n)g_i-\pi(\lambda'_n)g_i\| =0$. This leads to a contraction that $ \|\pi(\lambda_n)g_i-\pi(\lambda'_n)g_i\|^2 \ge 2A$ for all $n\in \N $.

To prove $D^{+}(\Lambda) \le 1$ we follow the proof of \cite[Theorem 4]{BCHL}. Suppose $\phi$ is a Gaussian function and $\Lambda'=\alpha \Z^d \times \beta \Z^d$ with $\alpha \beta <1$. Then $G(\phi ,\Lambda')$ is a Gabor frame for $L^2(\R^d)$.
For $(x, \omega)\in \R^{2d}$, set $\bar x= \alpha (\lfloor x_1/\alpha_1\rfloor, \dots ,$ and $ \lfloor x_n/\alpha_n \rfloor),\bar \omega = \alpha (\lfloor \omega _1/\alpha_1\rfloor, \dots , \lfloor \omega_n/\alpha_n \rfloor)$. Define function $a:\Lambda \to \Lambda'$ by $ (x, \omega ) \mapsto (\bar x ,\bar \omega )$. 
%and \[r_{i,\lambda'}:=\sup_{\lambda\in \lambda'+[0,\alpha) ^d \times [0 ,\beta )^d}|V_{\phi}g_i(-\lambda)|, \lambda' \in \Lambda',1\le i \le n.\] By \cite[Theorem~12.2.1]{G01}, we have $r_i=\{r_{i,\lambda'}\}_{\lambda' \in \Lambda '} \in \ell^2(\Lambda')$. 
Given $\lambda' \in \Lambda'$ and $\lambda \in \Lambda_i$ for some $1\le i \le n $, by \cite[Theorem 2]{BCHL}, there exists a sequence $r_i=\{r_{i,\lambda'}\}_{\lambda '\in \Lambda'} \in \ell^2(\Lambda')$ such that $|\langle \pi(\lambda)g_i, \pi (\lambda')\phi \rangle | \le  r_{i,a(\lambda)-\lambda'}.$ Hence   $|\langle \pi(\lambda)g_i, \pi (\lambda')\phi \rangle | \le  \sum_{i=1}^nr_{i,a(\lambda)-\lambda'}$ for all $1\le i\le n$ and $\lambda \in \Lambda_i$. It follows from \cite[Theorem 3]{BCHL1} that  $(\alpha \beta)^d D^+(\Lambda) \le 1$ for any $\alpha \beta <1$. Therefore,  we conclude that $D^+(\Lambda) \le 1$. %   Let $\{\pi(\Lambda_i)g_i\}_{i=1}^n$ be a Riesz sequence for $L^2(\mathbb{R}^d)$ with Riesz bounds $A$ and $B$. Suppose $\Lambda$ is not uniform separated 
\end{proof}

\begin{thm}\label{thm: improved Balian-Low}
     Let $\Lambda=\bigsqcup_{i=1}^n \Lambda_i \subset \R^{2d}$ be a discrete set of disjoint union of $\Lambda_i$. For each $1\le i \le n$, choose a nonzero function $g_i\in M^1(\R^d)$. Then the following hold:
    \begin{enumerate} [label=(\roman*)]
        \item \label{thm: improved Balian-Low condition 1} If  
        $\{\pi(\Lambda_i)g_i\}_{i=1}^n$ is a frame for $L^2(\mathbb{R}^d)$, then $D^-(\Lambda)  >1$.
        \item \label{thm: improved Balian-Low condition 2} If $\{\pi(\Lambda_i)g_i\}_{i=1}^n$ is a Riesz sequence for $L^2(\mathbb{R}^d)$, then 
         $D^+(\Lambda)  <1$.
    \end{enumerate}
\end{thm}
\begin{proof}
   (\romannumeral 1) If $\{\pi(\Lambda_i)g_i\}_{i=1}^n$ is a frame for $L^2(\mathbb{R}^d),$ then $\Lambda$ is relatively separated and relatively dense and $D^-(\Lambda)\ge1 $ by \cite[Theorem 1.1]{CDH}. Suppose $\{\pi(\Lambda_i)g_i\}_{i=1}^n$ is a frame for $L^2(\mathbb{R}^d)$, but $D^-(\Lambda) =1$. Choose a sequence $\{\alpha_k\}_k$ such that $\alpha_k <1$ and $\lim_{k\to \infty}\alpha_k=1$. Define $\alpha_k\Lambda:=\{\alpha_k\lambda:\lambda\in \Lambda\}$. This implies $D^-(\alpha_k\Lambda)=\alpha_k^{2d}<1$, and $\alpha_k\Lambda_i$ is relatively separated and relatively dense. We will show that $\{\pi(\alpha_k\Lambda_i)g_i\}_{i=1}^n$ is a frame for $L^2(\R^d)$ when $k$ is large enough. But this contradicts \cite[Theorem 1.1]{CDH}. Thus, $D^-(\Lambda)>1$ holds.

Suppose  $\{\pi(\alpha_k\Lambda_i)g_i\}_{i=1}^n$ is not a frame for $L^2(\R^d)$. Then it is not a $\infty$-frame for $M^{\infty}(\R^d)$ by \cite[Theorem 3.2(a)]{GOR}. Thus, for $k \in \N_+$, there exists a nonzero function $f_k \in M^{\infty}(\R^d)$ such that $\|V_{g_1}f_k\|_{\infty}=1$ and 
 \[\sup_{1\le i\le n , \lambda\in \alpha_k\Lambda_i}|V_{g_i}f_k(\lambda)|\rightarrow 0.\] It follows that we can choose  $z_k\in \R^{2d} $ such that $|V_{g_1}f_k(z_k)|\ge 1/2$, and functions $h_k:=\pi(-z_k)f_k$ converging to some $h\in M^{\infty}(\R^d)$ under $\sigma(M^{\infty},M^1)$. Since $|V_{g_1}h_k(0)|=|V_{g_1}f_k(z_k)|\ge 1/2,$ we have $h$ is nonzero  by \cite[Lemma 2.1(b)]{GOR}. 
 Since $\{\alpha_k\}_k$ is bounded and $\operatorname{rel}(\alpha_k\Lambda)=\sup_x|\Lambda \cap B(x,1/\alpha_k)|$, we have that 
$\sup_k\operatorname{rel}(\alpha_k\Lambda-z_k)=\sup_k\operatorname{rel}(\alpha_k\Lambda) < \infty$. For all $1\le i \le n$, by \cite[Lemma 4.5 (b)]{GOR}, we may pass to a subsequence and assume $\alpha_k \Lambda_i-z_k \xrightarrow{w} \Gamma_i$ for some $\vec{\Gamma} =(\Gamma_1 ,\dots ,\Gamma_n)$.  Since $\Lambda=\bigsqcup_{i=1}^n\Lambda_i$ is relatively dense, it follows from  Lemma \ref{lem:deformation still in W(lambda)} that $\vec \Gamma \in W(\vec \Lambda)$.
 For $\gamma\in \Gamma_i$, there exists $\lambda_k^i\in\Lambda_i $ such that $\alpha_k\lambda_k^i-z_k \rightarrow \gamma$. Using \cite[Lemma 2.1(b)]{GOR} again, we have 
 \[|V_{g_i}h(\gamma)|=\lim_k|V_{g_i}h_k(\alpha_k\lambda_k^i-z_k)|=\lim_k|V_{g_i}f_k(\alpha_k\lambda_k^i)|=0.\] By Theorem~\ref{thm:characterization of full frames}, one has $\{\pi(\Lambda_i)g_i\}_{i=1}^n$ is not a frame which is a contraction to the assumption in \ref{thm: improved Balian-Low condition 1}.

(\romannumeral 2) 
If $\{\pi(\Lambda_i)g_i\}_{i=1}^n$ is a Riesz sequence for $L^2(\mathbb{R}^d)$, then  $\Lambda_i$ is uniformly separated for $1\le i \le n$ and
         $D^+(\Lambda)  \le 1$ by Proposition~\ref{prop:Balian-low of full Riesz sequence}. Suppose the contrary that $D^+(\Lambda) =1$. Choose a seuqence $\{\alpha_k\}$ such that $\alpha_k >1$ and $\lim_{n\to \infty}\alpha_k=1$. Define $\alpha_k\Lambda:=\{\alpha_k\lambda:\lambda\in \Lambda\}$. This implies $D^-(\alpha_k\Lambda)=\alpha_k^{2d}<1$ and $\alpha_k\Lambda_i$ is uniformly separated. We will show that $\{\pi(\alpha_k\Lambda_i)g_i\}_{i=1}^n$ is a Riesz sequence for $L^2(\R^d)$ when $k$ is large enough. But this contradicts Proposition \ref{prop:Balian-low of full Riesz sequence}. Thus, $D^+(\Lambda)<1$ holds.

 Assuming that  $\{\pi(\alpha_k\Lambda_i)g_i\}_{i=1}^n$ is not a Riesz sequence for $L^2(\R^d)$. Then it is not a $\infty$-Riesz sequence for $M^{\infty}(\R^d)$ by \cite[Theorem 3.2(b)]{GOR}. Thus, for $k \in \N$, there exists a non-zero sequence $\vec{c}^k \in (\bigoplus_{i=1}^n\ell^{\infty}(\alpha_k \Lambda_i))_{\infty}$ such that $\|\vec{c}^k\|_{\infty}=1$ and 
 $\lim_{k \to \infty}\|\sum_{i=1}^n \sum_{\lambda \in \alpha_k\Lambda_i}c_{i}^k(\lambda) \pi ( \lambda)g_i\|_{M^{\infty}} =0.$
 Applying Lemma~\ref{lem: continuity property of synthesis operator}, we obtain $n$-tuples of uniformly separated sets $\vec\Gamma=(\Gamma_1 ,\dots ,\Gamma_n)\in W(\vec \Lambda)$, a  non-zero sequence $\vec c\in (\bigoplus_{i=1}^n\ell^{\infty}(\Gamma_i))_{\infty}$ and a subsequence $\{\lambda_{k_m}\} \subset \Lambda$
 such that
     \[ \alpha_{k_m}\Lambda_i-\lambda_{k_m} \xrightarrow{w} \Gamma_i  \ \text{as} \ k\to \infty \ \text{and} \ \sum_{i=1}^n\sum_{\gamma \in \Gamma_i}c_{i}(\gamma) \pi (\gamma)g_i=0.\]
       By Theorem~\ref{thm:characterization of full Riesz sequence }, one has $\{\pi(\Lambda_i)g_i\}_{i=1}^n$ is not a Riesz sequence, which contracts the assumption in \ref{thm: improved Balian-Low condition 2}.
 \end{proof}

\section{Acknowledgment}
The first and third author are supported by NSFC No. 12471131. The second author is supported by NSFC No. 12571133 and  would like to thank Eduard Vilalta for introducing him this fascinating topic. The authors would like to thank Ulrik Enstad,  Karlheinz Gr\"ochenig, Franz Luef, and Eduard Vilalta for very helpful comments and suggestions on our first draft.

%Lipschitz deformation

\end{document}